\documentclass[11pt,a4paper,reqno]{amsart}
\usepackage[a4paper,left=2.0cm,right=2.0cm,top=2.7cm,bottom=2.7cm]{geometry}  
\newtheorem{theorem}{Theorem}[section] 
\newtheorem{lemma}[theorem]{Lemma}
\newtheorem{definition}[theorem]{Definition}
\newtheorem{proposition}[theorem]{Proposition} 
\newtheorem{remark}[theorem]{Remark} 
\usepackage{mathrsfs,amssymb,amsmath,amsthm,stmaryrd,graphicx,mathtools,xfrac}
\usepackage{upgreek}
\usepackage{epstopdf}
\usepackage[amsstyle]{cases}
\usepackage{enumitem}
\usepackage[linkcolor=blue,citecolor=cyan,colorlinks]{hyperref}
\usepackage[utf8]{inputenc}
\usepackage[T1]{fontenc}
\usepackage{xcolor,ulem}   
\numberwithin{equation}{section}
\allowdisplaybreaks
\usepackage[integrals]{wasysym}
\usepackage{subcaption} 

\renewcommand{\star}{\ast} 
\renewcommand{\ge}{\geqslant}
\renewcommand{\le}{\leqslant}
\renewcommand{\int}{\varint}
\renewcommand{\to}{\rightarrow} 

\newcommand{\vare}{\varepsilon}  
\newcommand{\norm}[1]{\left\| #1 \right\|} 
\newcommand{\module}[1]{\left| #1 \right|} 
\newcommand{\Brace}[1]{\left\{ #1 \right\}} 
\newcommand{\Paren}[1]{\left( #1 \right)}  
\newcommand{\Bracket}[1]{\left[ #1 \right]}  

\newcommand{\cov}{U}
\newcommand{\loc}{\theta}
\newcommand{\pr}{\lambda}

\newcommand{\DT}{\mathscr{D}_t }
\newcommand{\pl}{\partial}
\newcommand{\lpe}{\Delta} 
\newcommand{\divergence}{\nabla\cdot }  
\newcommand{\vorticity}{\nabla\times } 

\newcommand{\mc}{\mathscr{A}}
\newcommand{\sff}{\mathscr{B}}
\newcommand{\bdv}{\mathscr{U}}
\newcommand{\LB}{\Delta_{\sff}} 
\newcommand{\Bdnabla}{\overline{\nabla} } 
\newcommand{\Bddiv}{\Bdnabla\cdot }
\newcommand{\normal}{n}  
\newcommand{\parOmega}{\partial\Omega}
\newcommand{\parXi}{\partial\Xi}
\newcommand{\radi}{\mathscr{R}}

\newcommand{\err}{\mathfrak{R}} 
\newcommand{\one}{u}
\newcommand{\energy}{\mathscr{E}}
\newcommand{\Energy}{\mathfrak{E}}
\newcommand{\assone}{\mathscr{P}}
\newcommand{\asstwo}{\mathscr{Q}}   
\newcommand{\bd}{C_\dagger}
\begin{document}
\title[Low-regularity blow-up and self-intersection]{Low-regularity a priori estimates, blow-up criterion, and self-intersection singularities for free-boundary ideal magnetohydrodynamics with surface tension} 
\author{Tao Luo and Siqi Yang}  
\address{Tao Luo\newline
Department of Mathematics, City University of Hong Kong, Hong Kong, China} 
\address{Siqi Yang\newline School of Mathematics, Statistics and Mechanics, Beijing University of Technology, Beijing 100124, China}
\email{taoluo@cityu.edu.hk}
\email{siqiyang@bjut.edu.cn}
\begin{abstract} 
We study the three-dimensional incompressible free-boundary ideal magnetohydrodynamic (MHD) equations with surface tension and a closed free surface. Our first result establishes $H^3$ a priori estimates in general bounded domains, without graph structure, periodicity, or simple connectedness; in particular, for surface-tension ideal MHD in general domains this lowers the previously available threshold from $H^6$. Compared with the free-boundary problem for incompressible Euler equations, the feature is that the Lorentz force enters the elliptic pressure estimates, and the frozen-in magnetic field must preserve the tangential boundary constraint. Using these estimates, we prove a refined finite-time blow-up criterion for $H^3$ solutions that separates topological self-intersection, loss of boundary regularity, blow-up of the normal velocity, and interior MHD blow-up. The interior condition has an intrinsic magnetic-field asymmetry: besides $\norm{\nabla u}_{L^\infty}$ and $\norm{\nabla h}_{L^\infty}$, with $u$ and $h$ denoting the velocity and magnetic field, respectively, it requires the additional control of $\norm{\nabla^2h}_{L^2}$, a quantity arising from the Lorentz-force contribution to the pressure estimates and having no velocity analogue. Finally, we construct regular initial data whose solutions develop finite-time boundary self-intersection while the Sobolev regularity and curvature remain controlled up to the contact time. Thus, neither surface tension nor the ideal magnetic coupling precludes topological self-intersection of the free boundary.
\end{abstract} 
\maketitle
\tableofcontents

\section{Introduction} 
We consider the following three-dimensional incompressible free-boundary ideal magnetohydrodynamic (MHD) equations with surface tension: 
\begin{subnumcases}{\label{eq_MHD}} 
\DT u+\nabla p=h\cdot\nabla h,&\text{in}\ \Omega_t,\label{eq_MHD1}\\
\DT h=h\cdot\nabla u,&\text{in}\ \Omega_t,\label{eq_MHD2}\\ 
\divergence u=0,\quad\divergence h=0,&\text{in}\ \Omega_t,\label{eq_MHD3}\\
u_n=\bdv,\quad h_n=0,\quad  
p=\mc,
&\text{on}\  \parOmega_t,\label{eq_MHD4}\\ 
u(\cdot,0)=u_0,\quad h(\cdot,0)=h_0,&\text{in}\ \Omega_0,\label{eq_MHD5}
\end{subnumcases}
where $t>0$ denotes time, $u=u(x,t)$ is the velocity field of the fluid, $h=h(x,t)$ is the magnetic field, and $p=p(x,t)$ is the scalar total pressure, including $\frac12\module{h(x,t)}^2$. For each time $t$, the fluid region is denoted by $\Omega_t$, a bounded domain in $\mathbb{R}^3$ with moving boundary $\parOmega_t$. We consider the case in which the free boundary $\parOmega_t$ is a closed surface. The material derivative is $\DT\coloneqq\partial_t+u\cdot\nabla$. On $\parOmega_t$, $\normal=\normal(x,t)$ represents the unit outer normal and $\mc$ denotes the mean curvature. Without loss of generality, we assume the surface tension coefficient is $1$.  
$\bdv$ represents the normal velocity of $\parOmega_t$, which is equal to the normal component of the velocity field $u_n\coloneqq u\cdot\normal$, and $h_n\coloneqq h\cdot\normal$ denotes the normal component of the magnetic field. Additionally, $u_0,h_0$, and $\Omega_0$ are the prescribed initial data.   

Two coupled difficulties are central to the present work. First, surface tension ties the pressure to the geometry of the moving surface through the boundary condition $p=\mc$, so that elliptic pressure estimates and boundary regularity estimates must be closed simultaneously. This becomes especially delicate near a possible self-intersection, where the limiting boundary is no longer an embedded hypersurface and the normal vector, second fundamental form, and mean curvature cannot be interpreted as single-valued geometric quantities at the contact point. Second, unlike the Euler equations, the ideal MHD system contains a frozen-in magnetic field satisfying both the interior constraint $\divergence h=0$ and the boundary constraint $h\cdot\normal=0$. Thus, any self-intersection singularity construction must preserve the tangential magnetic constraint and control the Lorentz force $h\cdot\nabla h$ uniformly as two boundary branches approach each other.
\subsection{Historical works}

\subsubsection{Well-posedness for incompressible free-boundary ideal MHD equations}
In the absence of surface tension, the Taylor sign condition is typically required for the well-posedness of the incompressible free-boundary ideal MHD equations. Under this condition, Hao and the first author \cite{Hao2014} established the a priori estimates for domains homeomorphic to a ball, while demonstrating ill-posedness when the condition fails \cite{Hao2020}. For fluid domains with graph assumptions, \cite{Luo2020} derived the a priori estimates in $H^{\frac52+\delta}$ (with $\delta\in(0,\frac12)$) for small fluid volumes, and local well-posedness was established in Lagrangian coordinates \cite{Gu2019} and Eulerian coordinates \cite{Zhao2024}. For general bounded domains without simple connectedness, periodicity, or graph assumptions, \cite{Ifrim2024a} established local well-posedness for low-regularity $H^s$ solutions ($s>\frac n2+1$, where $n$ is the spatial dimension). For the well-posedness of the plasma-vacuum and current-vortex sheet problems, we refer to \cite{Hao2017,Morando2014,Sun2019,Sun2018}.

For fluid domains with graph boundaries and surface tension, \cite{Luo2021} first derived the a priori estimates in $H^{\frac72}$. Subsequently, \cite{Gu2022} established the zero surface tension limit, and \cite{Gu2023} proved local well-posedness in $H^{\frac92}$ via an artificial viscosity approximation. For general domains without graph assumptions, Liu and the first author \cite{Liu2026}, together with \cite{Liu2023}, established the local well-posedness and the zero surface tension limit for the plasma-vacuum problems (where the plasma region is simply connected), requiring $H^{\frac92}$ regularity. See \cite{Liu2023a} for the related current-vortex sheet problem. If the simple connectedness assumption is further removed, the only available result is the a priori estimates for $H^6$ solutions established by Hao and the second author \cite{Hao2025}. Although local well-posedness at low regularity has been achieved for general fluid domains without surface tension \cite{Ifrim2024a,Ifrim2025}, in the presence of surface tension the corresponding $H^3$ local well-posedness theory is available only for the free-boundary Euler equations \cite{Shatah2011}. Thus, for surface-tension ideal MHD in general bounded domains without the above topological and geometric restrictions, the $H^3$ local well-posedness problem remains open.  

\subsubsection{Blow-up criteria}
The global well-posedness of the incompressible free-boundary ideal MHD equations remains open. Available global results concern dissipative variants incorporating viscosity \cite{Wang2019} or magnetic diffusion \cite{Wang2021}. In the ideal setting, blow-up theory has mainly sought criteria that separate interior loss of regularity from geometric degeneration of the free surface. In the absence of surface tension, such criteria have been obtained for the free-boundary Euler equations in graph domains \cite{Wang2021a}, domains homeomorphic to a ball \cite{Ginsberg2021}, and general bounded domains \cite{Ifrim2025}. For the ideal MHD equations, Fu, Hao, the second author, and Zhang established a Beale-Kato-Majda (BKM) type blow-up criterion for domains homeomorphic to a ball \cite{Fu2023}; building upon the Euler analysis in \cite{Ifrim2025}, \cite{Ifrim2024a} obtained a blow-up criterion in general bounded domains under identical regularity conditions. 

With surface tension, the pressure boundary condition involves the mean curvature and the blow-up analysis must also track the geometry of the evolving hypersurface. In this direction, \cite{Julin2024} established the a priori estimates and a blow-up criterion for regular solutions of the Euler equations coupled with an electric field, and \cite{Luo2024} established a BKM-type criterion for regular solutions (in $H^s$, $s>\frac92$) of the Euler equations in graph domains. For ideal MHD with surface tension, the only existing result is the $H^6$ blow-up criterion of Hao and the second author in general bounded domains \cite{Hao2025}. Earlier criteria based on fixed graph representations or Lagrangian coordinates do not naturally capture the approach of two distant boundary portions toward self-intersection. The work \cite{Hao2025} introduced a method of altering reference surfaces to include such geometric degeneration in the blow-up analysis. However, its regularity requirement is high, and the criterion does not reduce to an Euler-type BKM criterion when the magnetic field vanishes. Subsequently, for the free-boundary Euler equations with surface tension, Hao and the present authors reduced the regularity threshold and obtained a refined blow-up theory, both with and without the simple connectedness assumption \cite{Hao2025a}.  
\subsubsection{Self-intersection singularities}
Self-intersection singularities provide a different, genuinely topological mechanism for the failure of a free boundary evolution. They can be broadly classified into splash singularities, where distinct boundary points meet at one or several points, and splat singularities, where two boundary portions meet along a set of positive measure. Their defining feature is that local parametrizations may remain regular up to the contact time, while the global boundary ceases to be embedded. Such singularities were first constructed for the 2D irrotational water wave problem without surface tension \cite{Castro2013}. For rotational fluids, \cite{Coutand2014} proved that the 3D incompressible Euler equations can develop splash or splat singularities under the Taylor sign condition. 

The role of surface tension in self-intersection problems is more subtle. On the one hand, it supplies a curvature force that tends to regularize the interface. On the other hand, it does not automatically preclude topological contact. For the 2D irrotational water wave equations with surface tension, \cite{Castro2012b} proved the existence of splash singular solutions. In the two-phase free-boundary problem of the rotational incompressible Euler equations, \cite{Fefferman2016} provided boundary regularity conditions under which splash singularities are prevented, while \cite{Cordoba2021} constructed steady self-intersecting singular solutions in the irrotational setting. Furthermore, \cite{Coutand2019a} considered a class of symmetric initial data and proved that either the free boundary self-intersects in finite time, or certain natural norms of the fluid blow up. These results show that the influence of surface tension depends strongly on the equation, the dimension, and the type of interface under consideration.

For MHD, the available self-intersection constructions are much more limited. Previous results have primarily focused on viscous models: Hao and the second author constructed a splash singularity for the 2D viscous MHD equations \cite{Hao2024}, and Hong, the first author, and Zhao \cite{Hong2025} constructed a class of initial data, close to self-intersection, that induces boundary self-intersection for the 2D or 3D viscous MHD equations. These constructions do not extend directly to the ideal system, where there is no parabolic smoothing and the magnetic field is frozen into the flow. In the ideal setting, current results are limited to a splash-squeeze type self-intersection singularity for the 2D plasma-vacuum problem \cite{Cordoba2025}, where the constructed solutions preserve Sobolev regularity but instantaneously lose analyticity at a certain moment. Thus, before the present work, the construction of a 3D self-intersection singularity for the ideal MHD equations remained open, and it was unclear whether the curvature force generated by surface tension could prevent such a singularity.

The present paper addresses these issues at the level needed for blow-up and self-intersection analysis. Although \eqref{eq_MHD} reduces to the free-boundary Euler equations when $h=0$, the results below are not obtained by simply inserting magnetic terms into the Euler argument. The Lorentz force enters the elliptic pressure estimates, the induction equation transports a frozen-in magnetic field, and the boundary constraint $h\cdot\normal=0$ must be preserved even as two boundary branches approach contact. We establish $H^3$ a priori estimates for \eqref{eq_MHD} in general bounded domains and derive a comprehensive blow-up criterion in this MHD framework. We further construct a class of regular initial data whose solutions develop finite-time boundary self-intersection while remaining regular up to the contact time. The construction shows that the surface-tension restoring force and the ideal magnetic coupling can be controlled simultaneously near self-intersection.

\vspace{2mm}
\noindent\textbf{Main contributions.}
The contributions of this paper are threefold.
\begin{enumerate}[label={\rm (\roman*)},topsep=5pt,itemsep=2pt]
\item We establish low-regularity a priori estimates for \eqref{eq_MHD}: the estimates are obtained at the $H^3$ level in general bounded three-dimensional domains, without assuming graph structure, periodicity, or simple connectedness. This lowers the regularity threshold from the previously available $H^6$ theory in general domains. Although the resulting regularity level agrees with that of the free-boundary Euler theory with surface tension, the estimates require MHD-specific control of the Lorentz force, the pressure structure, and the tangential magnetic boundary condition.
\item We prove a refined finite-time blow-up criterion that includes boundary self-intersection, loss of boundary regularity, blow-up of the normal velocity, and interior MHD quantities in a single framework. A distinctive feature of this criterion is the magnetic-field asymmetry: besides controlling $\norm{\nabla h}_{L^\infty}$, the low-regularity pressure estimates require the additional control of $\norm{\nabla^2h}_{L^2}$, a quantity with no velocity analogue and no counterpart in the Euler reduction. In particular, when the magnetic field vanishes, the criterion reduces to an Euler-type criterion in the same $H^3$ setting.
\item We construct finite-time self-intersection singularities for the 3D ideal MHD equations with surface tension. The construction keeps the magnetic field divergence-free, tangent to the moving boundary, and uniformly controlled on the approximating domains, showing that the curvature boundary condition and the frozen-in magnetic field do not prevent topological contact of the free boundary.
\end{enumerate}
\subsection{Main results} 
To state our main results, we introduce the energy functional $\Energy(t)$ as
\begin{align}
\Energy(t)\coloneqq{}&\norm{\DT^2u}_{L^2(\Omega_t)}^2+\norm{\DT^2h}_{L^2(\Omega_t)}^2+\norm{\DT u}_{H^{\frac32}(\Omega_t)}^2+\norm{\DT h}_{H^{\frac32}(\Omega_t)}^2\nonumber\\
&+\norm{u}_{H^3(\Omega_t)}^2+\norm{h}_{H^3(\Omega_t)}^2+\norm{\Bdnabla\Paren{\DT u\cdot\normal}}_{L^2(\parOmega_t)}^2+1.\label{eq_Energy}
\end{align}
Here, $\Bdnabla$ denotes the tangential differential operator: for a scalar function, 
\begin{equation*}
\Bdnabla (\,\cdot\,)\coloneqq\nabla(\,\cdot\,)-\nabla(\,\cdot\,) \cdot\normal\normal,
\end{equation*} 
and for a vector field $\Bdnabla (\,\cdot\,)\coloneqq \nabla (\,\cdot\,)-\nabla (\,\cdot\,)\normal\otimes\normal$. The tangential divergence is defined as $\Bddiv (\,\cdot\,)\coloneqq \operatorname{Tr}\Bdnabla (\,\cdot\,)$.  Consequently, the second fundamental form and the mean curvature can be expressed, respectively, as
\begin{equation*} 
\sff=\Bdnabla \normal,\ \text{and}\ \mc=\Bddiv\normal.
\end{equation*} 
The free boundary $\parOmega_t$ is parameterized over a smooth compact reference hypersurface $\Gamma=\partial\Omega$. Here, $\Omega\subset\mathbb{R}^3$ is a bounded domain satisfying the uniform interior and exterior ball conditions with a maximal radius $\radi>0$, defined by 
\begin{align*}
\radi:=\sup  \bigg\lbrace r > 0 : \forall x \in \Gamma, \ \exists B(y,r) \subset\Omega, \ B(z,r) \subset \Omega^c,\ \text{such that}\ x \in \partial B(y,r) \cap \partial B(z,r)\bigg\rbrace,
\end{align*} 
where $B(x,r)$ denotes the open ball centered at $x$ with radius $r>0$.
For any $t\ge0$, we represent the free boundary
\begin{equation*}
\parOmega_t=\{x+\eta(x,t)\normal(x):x\in\Gamma\}
\end{equation*}
utilizing the height function 
\begin{equation*}
\eta(\cdot,t):\Gamma\to\mathbb{R},\ \text{provided that}\ \norm{\eta(\cdot,t)}_{L^\infty(\Gamma)}<\radi.
\end{equation*}  

\vspace{2mm}
\noindent\textbf{Initial data.} Let $\Omega_0$ be the initial fluid domain, and assume that the initial 
boundary can be represented as
\begin{equation*}
\parOmega_0=\{x+\eta_0(x)\normal(x):x\in\Gamma\},\ \text{where the height function satisfies}\ \norm{\eta_0}_{L^\infty(\Gamma)}<\radi.
\end{equation*}
Let $u_0, h_0\in H^3(\Omega_0)$ be divergence-free initial velocity and 
magnetic fields, respectively, satisfying 
\begin{equation*}
h_0\cdot\normal=0\ \text{on}\  \parOmega_0.
\end{equation*}

Given a solution $(u,h,\Omega_t)$ to system \eqref{eq_MHD} on a time interval $\left[0,T\right)$, we define the following geometric and analytic control quantities: 
\begin{align}
\assone_T&\coloneqq\radi-\sup_{t\in[0,T)}\norm{\eta(\cdot,t)}_{L^\infty(\Gamma)},\label{eq_ass1}\\
\asstwo_T &\coloneqq\sup_{t\in[0,T)}\Paren{\norm{\nabla u}_{L^{\infty}(\Omega_t)}+\norm{\nabla h}_{L^{\infty}(\Omega_t)}+\norm{\nabla^2h}_{L^{2}(\Omega_t)}+\norm{\eta(\cdot,t)}_{H^{\frac52}(\Gamma)}+\norm{u_n}_{H^{2}(\parOmega_t)}}. \label{eq_ass2}
\end{align}

Our first main result is as follows:
\begin{theorem}[The a priori estimates]\label{thm_main1}
Let $\Brace{(u,h,\Omega_t):0\le t<T}$ be a solution to system 
\eqref{eq_MHD} for some time $T>0$, and suppose that there exists a constant $\bd>0$ such that the following a priori assumptions hold:
\begin{equation}\label{eq_aprioriassumption}
\assone_T^{-1}+\asstwo_T<\bd.
\end{equation}
Then the following a priori estimates hold:
\begin{equation}\label{eq_mainresult1}
\sup_{t\in[0,T)}\Paren{\Energy(t)+\norm{\DT p}_{H^{1}(\Omega_t)}+\norm{\sff}_{H^2(\parOmega_t)}}\le  C(\bd)e^{C(\bd)(1+T)}\Energy(0),
\end{equation}
where $C(\bd)$ is a positive constant depending on $\bd$, and the initial energy $\Energy(0)$ depends only on the initial data
$\norm{u_0}_{H^3(\Omega_0)}, \norm{h_0}_{H^3(\Omega_0)}$, 
and $\norm{\mc}_{H^2(\parOmega_0)}$.

Moreover, there exist constants $C_0>0$ and $T_0>0$, depending on the same initial data, such that the 
a priori assumptions \eqref{eq_aprioriassumption} hold on $[0,T_0)$ with $\bd=C_0$.
\end{theorem}  
Prior to the present work, a priori estimates and well-posedness results for incompressible free-boundary ideal MHD equations \eqref{eq_MHD} have predominantly relied on simplifying geometric or topological assumptions, such as spatial periodicity, simple connectedness, or graph representations of the free boundary. We establish, for the first time, a priori estimates for $H^3$ solutions with surface tension in general bounded fluid domains. This result substantially lowers the regularity threshold required in the literature, improving upon the stringent $H^6$ theory in \cite{Hao2025} for general domains and the $H^{\frac72}$ theory in \cite{Luo2021} under graph assumptions. It also provides the estimates needed for a possible $H^3$ local existence theory beyond the simply connected setting of \cite{Liu2023}. 
 
The asymmetric a priori assumptions in \eqref{eq_ass2} did not arise in the previous $H^6$ a priori estimates \cite{Hao2025}, where symmetric a priori assumptions were imposed on the interior velocity and magnetic fields. At the $H^3$ regularity level, these asymmetric assumptions reflect the intrinsic structural difference between momentum equation \eqref{eq_MHD1}, which contains the pressure term, and induction equation \eqref{eq_MHD2}. More precisely, the $H^{\frac{1}{2}}(\Omega_t)$ estimate for $\divergence \DT^{2}u$ in Proposition \ref{lem_ddt} requires control of $\nabla D_tu$ in $L^2(\Omega_t)$. By momentum equation \eqref{eq_MHD1}, this control involves the Lorentz-force term $\nabla(h\cdot\nabla h)$, whose $L^2$ estimate requires the additional a priori bound $\norm{\nabla^2h}_{L^2(\Omega_t)}$. 
By contrast, no analogous requirement arises from induction equation \eqref{eq_MHD2}. 

Our second main result addresses the finite-time blow-up criterion.
\begin{theorem}[Finite-time blow-up]\label{thm_main2}
Let $(u,h,\Omega_t)$ be a solution to the free-boundary problem 
\eqref{eq_MHD} with initial data $(u_0,h_0,\Omega_0)$, and let 
$[0,T_{\operatorname{max}})$ be its maximal interval of existence, satisfying
\begin{equation*}
(u,h)\in C\Paren{[0,T_{\operatorname{max}});H^3(\Omega_t)\times H^3(\Omega_t)}\ \text{and}\ \parOmega_t\in C\Paren{[0,T_{\operatorname{max}});H^4}.
\end{equation*}
If $T_{\operatorname{max}}<\infty$, for any sufficiently small $\varepsilon > 0$ independent 
of $T_{\operatorname{max}}$, at least one of the 
following scenarios must occur:
\begin{enumerate}[label={\rm (\arabic*)},topsep=5pt,itemsep=0pt] 
\item The free boundary $\parOmega_t$ self-intersects 
at $t=T_{\operatorname{max}}$.
\item Loss of boundary regularity:
\begin{equation*}
\limsup_{t \nearrow T_{\operatorname{max}}}\Paren{\norm{\mc}_{H^{\frac12}(\parOmega_t)} +\norm{\parOmega_t}_{H^{2+\varepsilon}}}=\infty.
\end{equation*}
\item Blow-up of the normal velocity field:
\begin{equation*}
\limsup_{t \nearrow T_{\operatorname{max}}}\norm{u_n}_{H^{2}(\parOmega_t)}=\infty.
\end{equation*} 
\item Interior blow-up:    
\begin{equation*} 
\limsup_{t \nearrow T_{\operatorname{max}}}\Paren{\norm{\nabla u}_{L^{\infty}(\Omega_t)}+\norm{\nabla h}_{L^{\infty}(\Omega_t)}+\norm{\nabla^2h}_{L^{2}(\Omega_t)}}=\infty.
\end{equation*}
\end{enumerate}
\end{theorem}  
Theorem \ref{thm_main2} establishes a comprehensive blow-up criterion for $H^3$ strong solutions, incorporating both boundary self-intersection and loss of boundary regularity. A distinctive feature is the asymmetry between the interior criteria for the velocity and magnetic fields: the velocity is controlled solely through $\norm{\nabla u}_{L^\infty(\Omega_t)}$, whereas the magnetic field requires control of both $\norm{\nabla h}_{L^\infty(\Omega_t)}$ and $\norm{\nabla^2h}_{L^2(\Omega_t)}$. This asymmetry reflects the pressure structure discussed above and identifies an additional magnetic-field quantity in the blow-up criterion that does not appear in the high-regularity framework \cite{Hao2025}. 
\begin{figure}[htbp] 
\centering
\includegraphics[width=14cm]{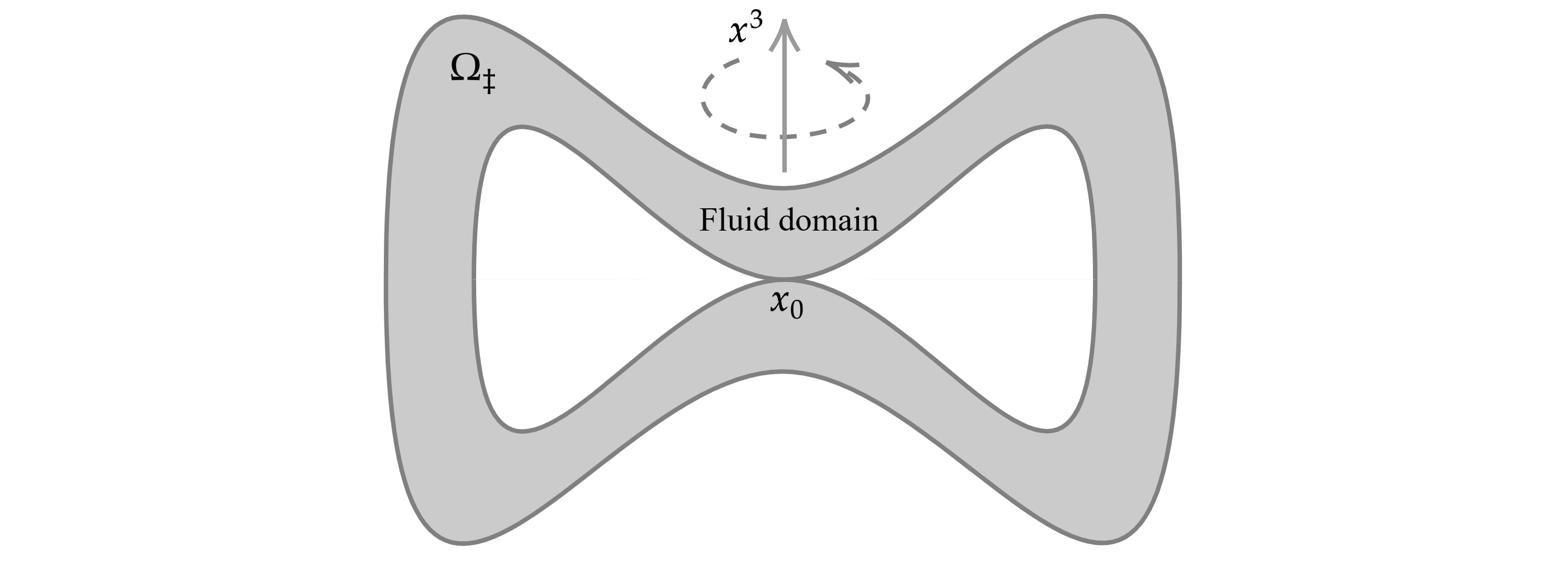} 
\caption{The three-dimensional axisymmetric self-intersecting domain 
generated by rotating a two-dimensional profile about the $x^3$-axis.}
\label{f6}
\end{figure}
\begin{figure}[htbp] 
\centering
\includegraphics[width=14cm]{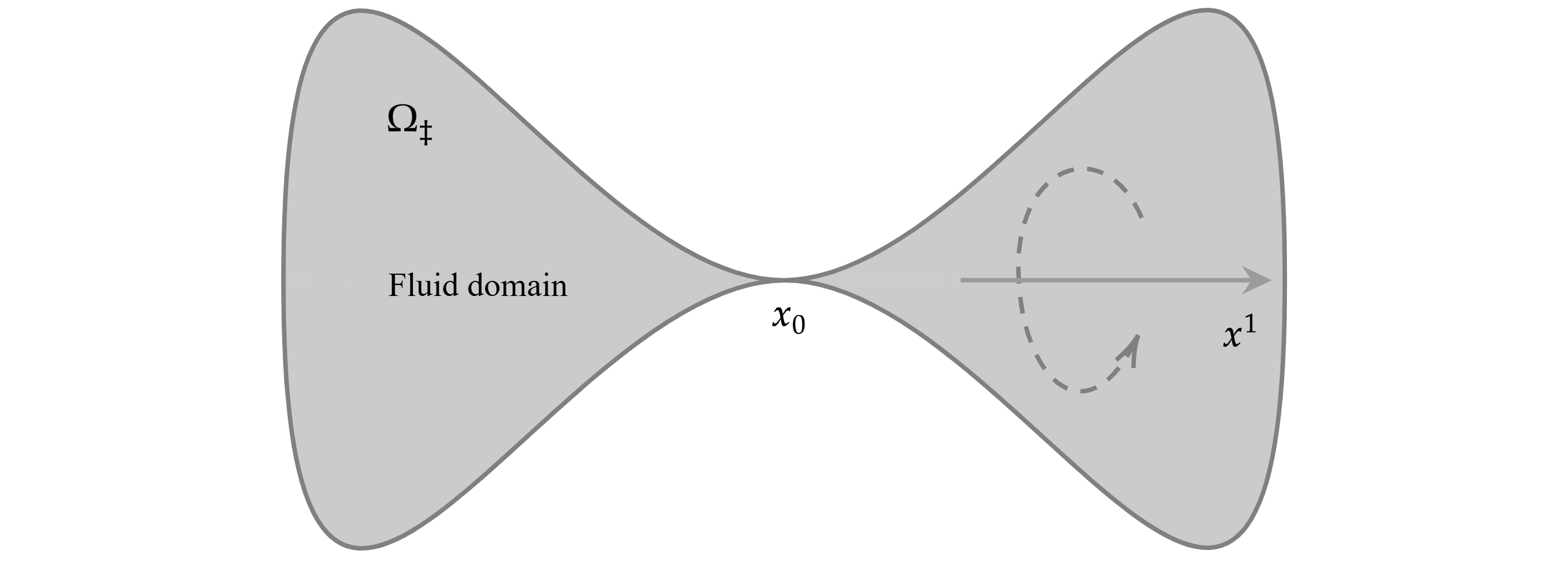} 
\caption{The three-dimensional axisymmetric self-intersecting domain 
generated by rotating a two-dimensional profile about the $x^1$-axis.}
\label{f1}
\end{figure} 

The boundary blow-up scenarios in Cases (1)--(3) describe distinct, though not necessarily mutually exclusive, mechanisms. The following examples illustrate how these scenarios may interact near the blow-up time. If the free boundary self-intersects as in Case (1) (see Fig.\,\ref{f6}), the boundary norms appearing in Cases (2) and (3) may remain uniformly bounded strictly before the contact time. As another example, suppose Case (2) occurs because the mean curvature loses its $H^{\frac12}$ regularity (see Fig.\,\ref{f1}). Then the $H^{\frac32}$ control of the normal vector is lost, and the corresponding $H^2$ control of the normal velocity $u_n$ is no longer available within the present framework. 

\vspace{2mm}
\noindent\textbf{Reduction to the blow-up criterion for Euler equations.}
When the magnetic field vanishes, problem \eqref{eq_MHD} reduces to the free-boundary Euler equations, and we obtain a finite-time blow-up criterion for the Euler system:
\begin{subnumcases}{} 
\text{The free boundary $\parOmega_{T_{\operatorname{max}}}$ self-intersects,}\label{BKME1}\\
\limsup\limits_{t \nearrow T_{\operatorname{max}}}\Brace{\norm{\mc}_{H^{\frac12}(\parOmega_t)} +\norm{\parOmega_t}_{H^{2+\varepsilon}}}=\infty,\label{BKME2}\\ 
\limsup\limits_{t \nearrow T_{\operatorname{max}}}\norm{u_n}_{H^{2}(\parOmega_t)}=\infty,\label{BKME3}\\
\limsup\limits_{t \nearrow T_{\operatorname{max}}}\norm{\nabla u}_{L^{\infty}(\Omega_t)}=\infty.\label{BKME4} 
\end{subnumcases}
Compared to the previous blow-up criterion (for free-boundary Euler equations) established under identical regularity and domain settings \cite[Theorem 1.1]{Hao2025a}:
\begin{subnumcases}{} 
\text{The free boundary $\parOmega_{T_{\operatorname{max}}}$ self-intersects,}\nonumber\\
\limsup\limits_{t \nearrow T_{\operatorname{max}}}\Brace{\norm{\mc}_{H^{\frac32}(\parOmega_t)} +\norm{\parOmega_t}_{H^{2+\varepsilon}}}=\infty,\nonumber\\ 
\limsup\limits_{t \nearrow T_{\operatorname{max}}}\norm{u_n}_{H^{\frac52}(\parOmega_t)}=\infty,\nonumber\\
\int_{0}^{T_{\operatorname{max}}}\norm{\nabla u}_{L^{\infty}(\Omega_t)}dt=\infty,\nonumber
\end{subnumcases}
we observe a delicate trade-off between the singular behavior on the free boundary and that in the fluid interior: \textit{the reduction in the boundary regularity requirements \eqref{BKME2} and \eqref{BKME3} is obtained at the expense of replacing the time-integrability condition on $\norm{\nabla u}_{L^\infty}$ by the stronger $L^\infty_t$ control in \eqref{BKME4}.} 

In light of this trade-off, the purpose of the present analysis is to identify weaker regularity conditions on the curvature and normal velocity that still yield a closed boundary blow-up criterion for free-boundary problem \eqref{eq_MHD}.

We now state our third main result, which constructs self-intersecting singular solutions corresponding to the first scenario in our blow-up criterion:
\begin{theorem}[Existence of the self-intersection singularity]\label{thm_main3}
The singularity scenario (1) predicted in Theorem \ref{thm_main2} is attainable. Specifically, for incompressible free-boundary ideal MHD equations \eqref{eq_MHD}, there exist suitable initial data $(u_0,h_0,\Omega_0)$ such that the corresponding solution has the regularity specified in Theorem \ref{thm_main2} and develops a finite-time topological self-intersection of the free boundary.
\end{theorem} 
One major obstruction in Theorem \ref{thm_main3} is the presence of surface tension.  Unlike self-intersection constructions without surface tension \cite{Castro2013,Coutand2014}, the pressure is prescribed by the mean curvature on the moving boundary. This condition provides a geometric restoring force, but at the limiting self-intersection point it no longer has a classical meaning: the boundary is not an embedded hypersurface and there is no single-valued normal vector or mean curvature. The construction therefore cannot be carried out by imposing $p=\mc$ directly on the singular boundary. Instead, the surface-tension condition is imposed on smooth approximating domains, and the curvature, pressure, and energy estimates are proved uniformly as the distance between the two boundary branches tends to zero. Theorem \ref{thm_main3} shows, in particular, that surface tension does not preclude finite-time self-intersection, while the relevant curvature bounds remain uniform up to the contact time. 

The local contact geometry is analogous in spirit to splash constructions for Euler flows, but the MHD problem is not obtained by simply carrying over the Euler argument. The magnetic field must remain divergence-free in the fluid, tangent to the moving boundary, and uniformly controlled in $H^3$ on the approximating domains. At the same time, its Lorentz-force contribution must be incorporated into the pressure and energy estimates. As explained in the next subsection, this control is compatible with the prescribed inward boundary motion up to the contact time. The novelty in the MHD setting is precisely to prove that the frozen-in magnetic field and the Lorentz force remain compatible with the self-intersection construction throughout the evolution. When the magnetic field vanishes, the construction reduces to the corresponding construction for the free-boundary Euler equations and yields the self-intersecting solutions conjectured in \cite{Hao2025a}.

\subsection{Strategy for constructing self-intersecting singular solutions}\label{sec_strategy}
The proof of Theorem \ref{thm_main3} relies on a backward-in-time construction via a sequence of regular approximating domains. This approximation is essential because the surface-tension condition $p=\mc$ is imposed on embedded free surfaces and cannot be used directly at the limiting self-intersection point. The main difficulty is to obtain a common existence interval for the approximate solutions, despite the absence of a uniform positive lower bound on the separation between the two approaching boundary branches, which approach self-intersection as the parameter tends to zero. Such a uniform interval is essential for the subsequent compactness argument. 

The proof is structurally divided into four steps.
\begin{figure}[htbp] 
\centering
\includegraphics[width=14cm]{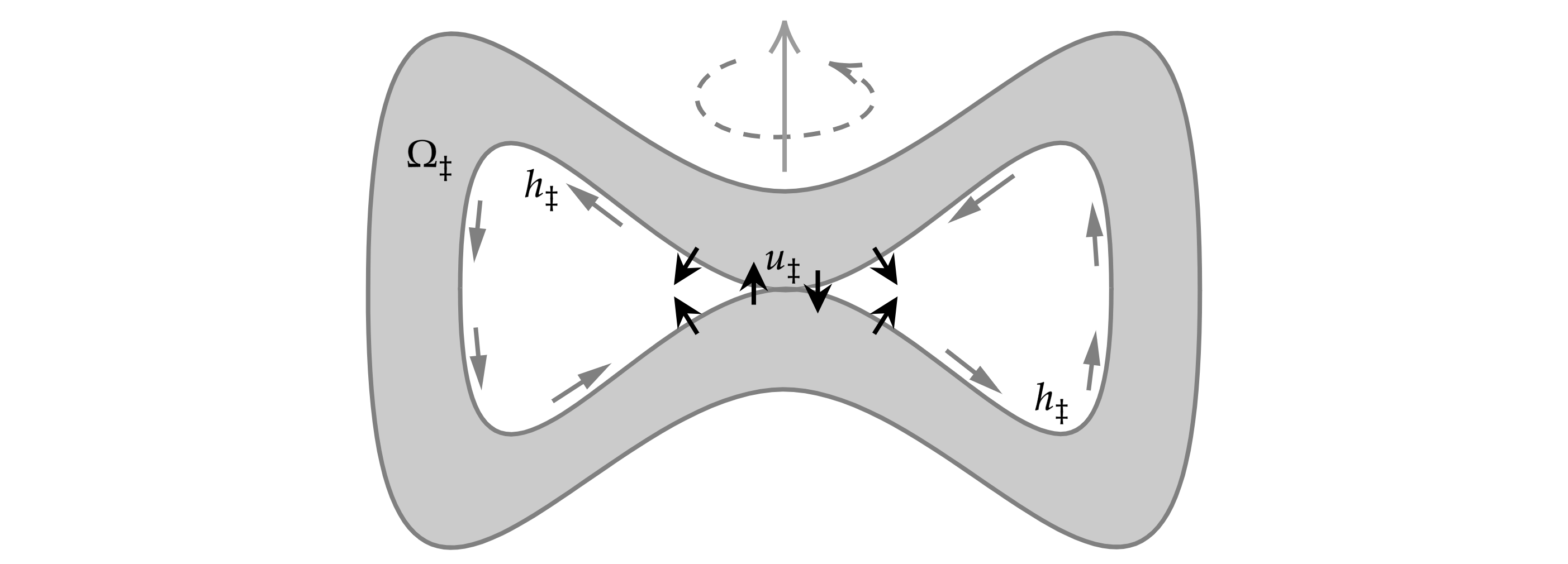} 
\caption{The self-intersecting domain.}
\label{f6i}
\end{figure}
\begin{figure}[htbp] 
\centering
\includegraphics[width=14cm]{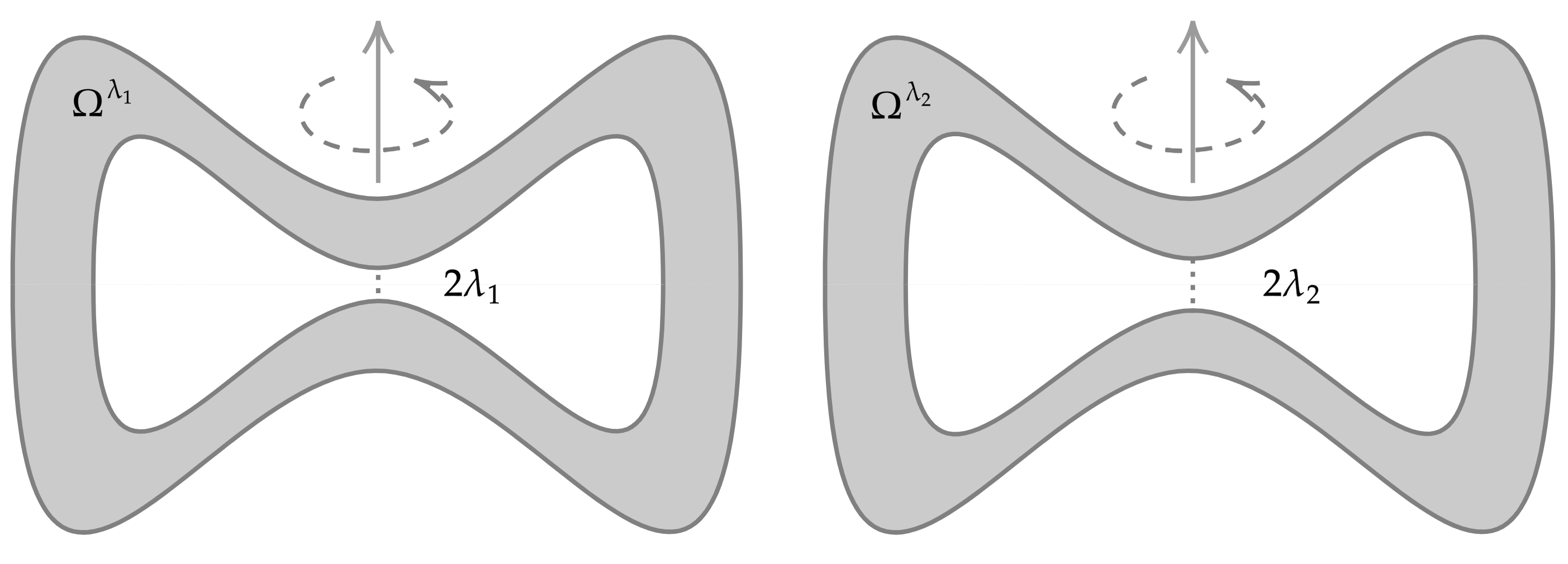} 
\caption{The approximate domains $\Omega^{\pr_1}$ and  $\Omega^{\pr_2}$ with $0<\pr_1<\pr_2$.}
\label{f8}
\end{figure}

\vspace{2mm}
\noindent\textbf{Step 1. Construction of approximate domains.} We first define the self-intersecting domain $\Omega_{\ddagger}$  (see Fig.\,\ref{f6i}). To approximate this singular domain, we introduce a family of regular, non-self-intersecting domains $\Omega^\pr$ ($\pr> 0$) by separating the two boundary branches near the contact point (see Fig.\,\ref{f8}). Through local coordinate decomposition techniques, we establish Sobolev and geometric estimates on $\Omega^\pr$ that are uniform in $\pr$.

\vspace{2mm}
\noindent\textbf{Step 2. Configuration of singular states.} On the singular domain $\Omega_{\ddagger}$, we prescribe a velocity field featuring strictly inward normal components at the contact point, alongside a compatible tangential magnetic field (see Fig.\,\ref{f6i} and Fig.\,\ref{f10}). Using local coordinates, these singular fields define a sequence of regular target states $(u_{\ddagger}^\pr, h_{\ddagger}^\pr)$ on the approximating domains $\Omega^\pr$ at the final time $t=0$ (see Fig.\,\ref{f9}).

\begin{figure}[htbp] 
\centering
\includegraphics[width=14cm]{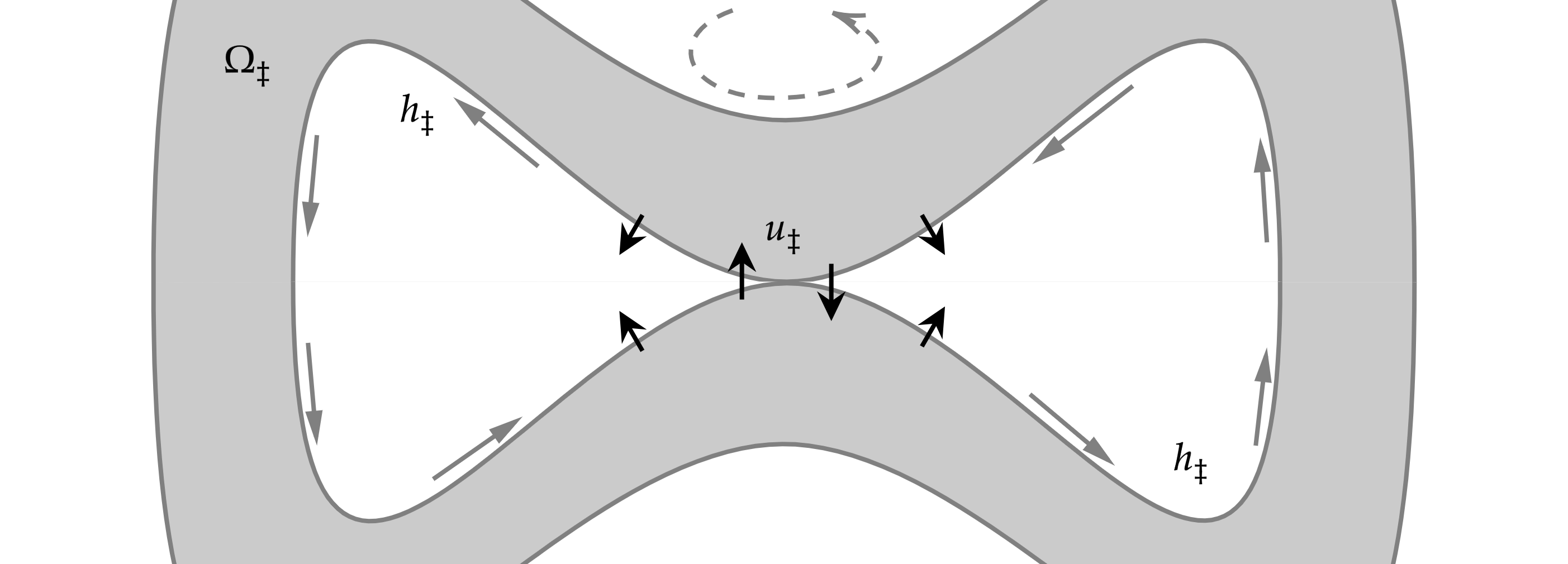} 
\caption{The singular velocity and magnetic fields on  $\parOmega_\ddagger$.}
\label{f10}
\end{figure}
\begin{figure}[htbp] 
\centering
\includegraphics[width=14cm]{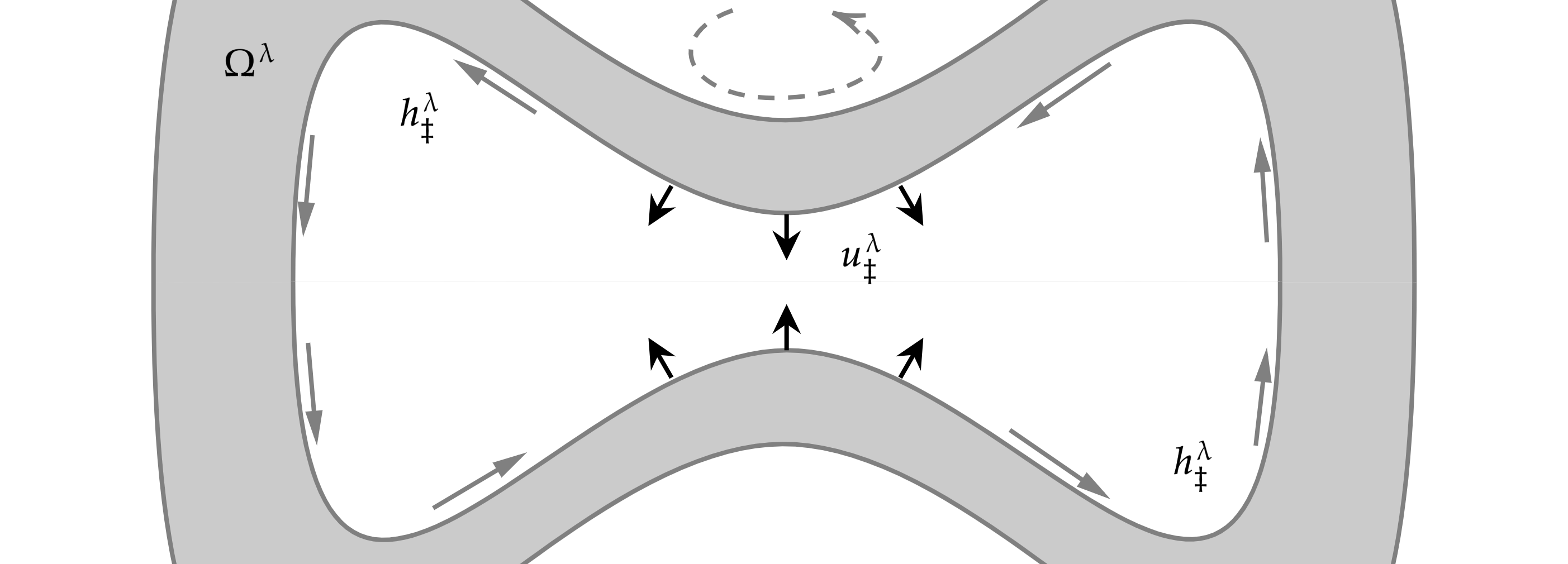} 
\caption{The approximate velocity and magnetic fields on  $\parOmega^{\pr}$.}
\label{f9}
\end{figure}  
\begin{figure}[htbp] 
\centering
\includegraphics[width=14cm]{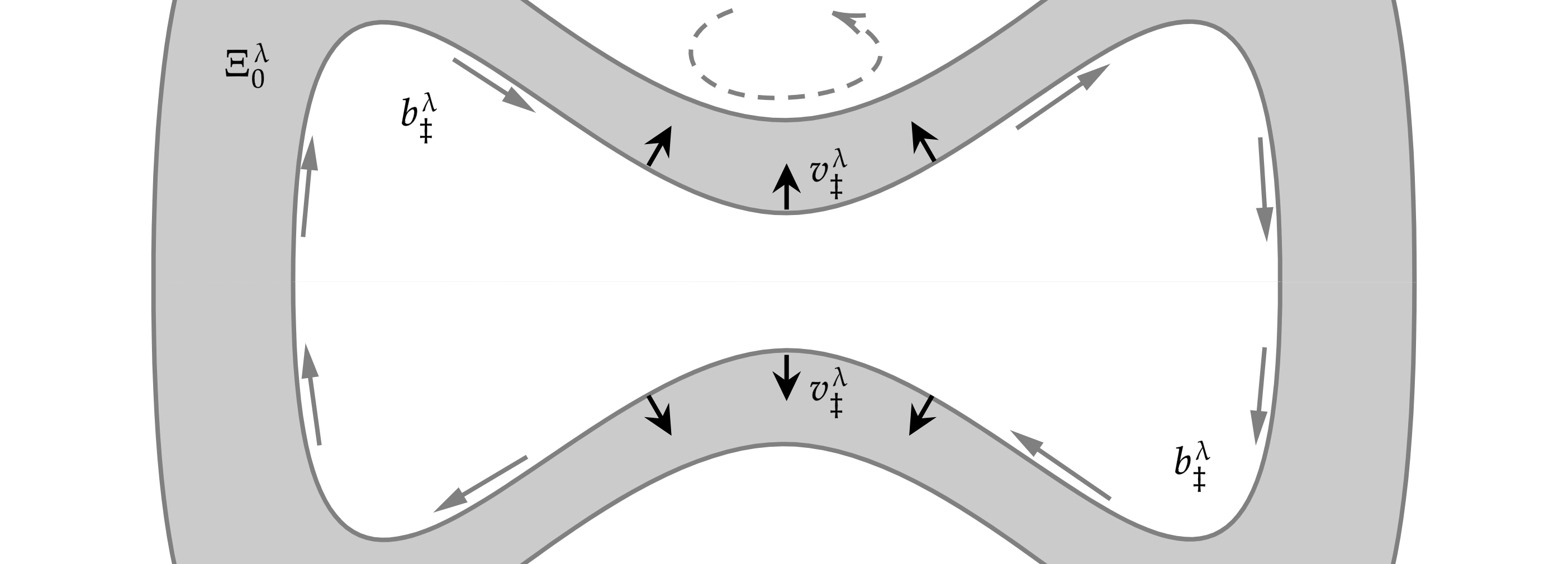} 
\caption{The backward-in-time approximate velocity and magnetic fields on $\parXi^{\pr}_0$.}
\label{f11}
\end{figure}
\begin{remark}
The tangential magnetic boundary condition remains compatible with the limiting geometry. It is imposed separately on the two approaching branches before contact. At the self-intersection point, their one-sided normals are collinear, and hence their tangent planes coincide. The corresponding magnetic traces need only be tangent to this common plane, with no condition imposed on their relative orientation.
\end{remark}
\vspace{2mm}
\noindent\textbf{Step 3. Uniform backward-in-time well-posedness.} We solve the free-boundary MHD equations backward in time on $\Omega^\pr$ from the prescribed terminal states. After the space-time reversal
\begin{equation*}
\Paren{v^\pr,b^\pr,q^\pr}(x,t)=\Paren{u^\pr,h^\pr,p^\pr}(-x,-t),\ \text{and}\ \Xi_t^\pr=\Brace{(-x^1,-x^2,-x^3):(x^1,x^2,x^3)\in \Omega^\pr_{t} },
\end{equation*} 
this is equivalent to solving the system forward in time from the initial domain $\Xi_0^\pr=-\Omega^\pr$ with initial velocity and magnetic fields $(v_\ddagger^\pr,b_\ddagger^\pr)(x)=(u_\ddagger^\pr,h_\ddagger^\pr)(-x)$ (see Fig.\,\ref{f11}).
Since $\Omega^\pr$, and hence $\Xi_0^\pr$, is simply connected, the local existence result  \cite{Liu2023} yields an $H^3$ solution on a time interval $[0,T^\pr]$ for each $\pr>0$.

To extract a convergent subsequence, we establish a uniform existence time $T_\star > 0$ independent of $\pr$, such that $v^\pr,b^\pr\in C^0 (0,T_\star;H^3(\Xi_t^\pr))$ and $ \parXi_t^\pr\in C^0 (0,T_\star;H^4)$. To this end, we need to guarantee that the inward boundary velocity does not reverse direction prematurely. A natural strategy would be to bound the $L^\infty$-norm of the boundary acceleration, which would immediately yield a lower bound for the time before the prescribed normal velocity can change sign. However, working within an $H^3$-regularity framework, we can only control the $H^{\frac{3}{2}}$-norm of the material derivative $\DT u$ in the fluid interior. This is insufficient to obtain an $L^\infty$ estimate for the acceleration on the free boundary. 
\begin{figure}[htbp]
\centering 
\begin{subfigure}{\textwidth}
\centering
\includegraphics[width=14cm]{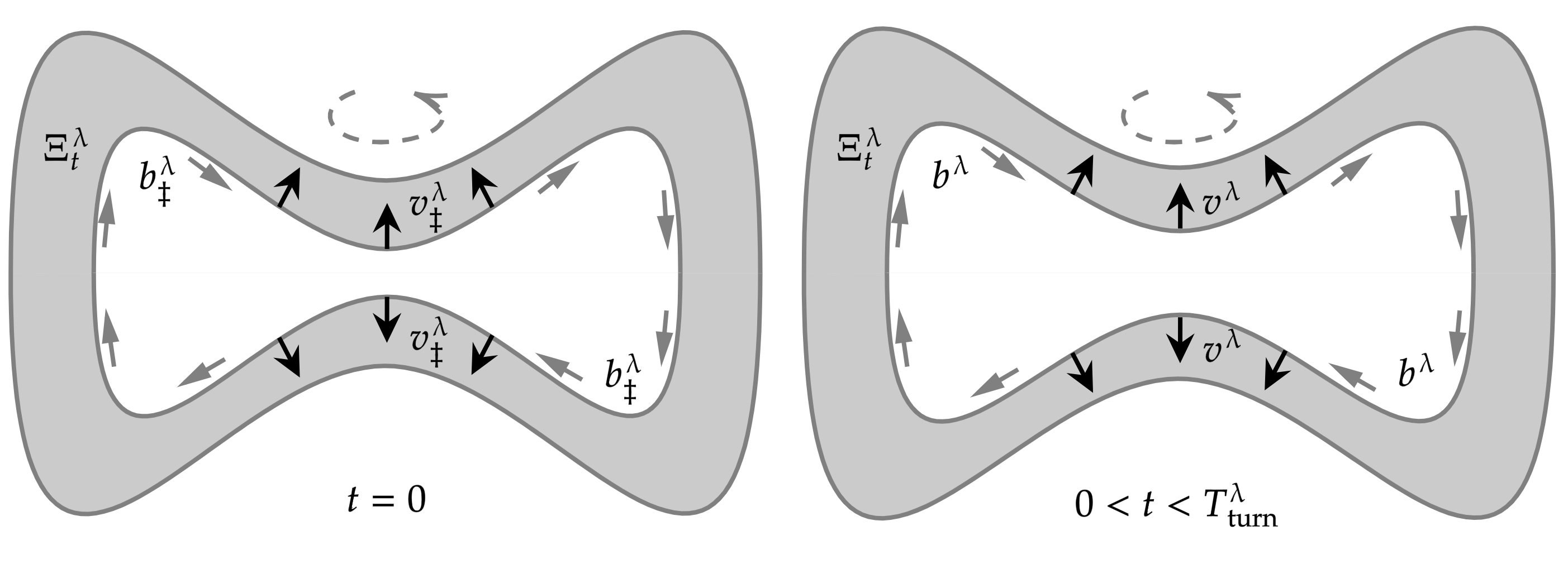} 
\end{subfigure} 
\vspace{2em}  
\begin{subfigure}{\textwidth}
\centering 
\includegraphics[width=14cm]{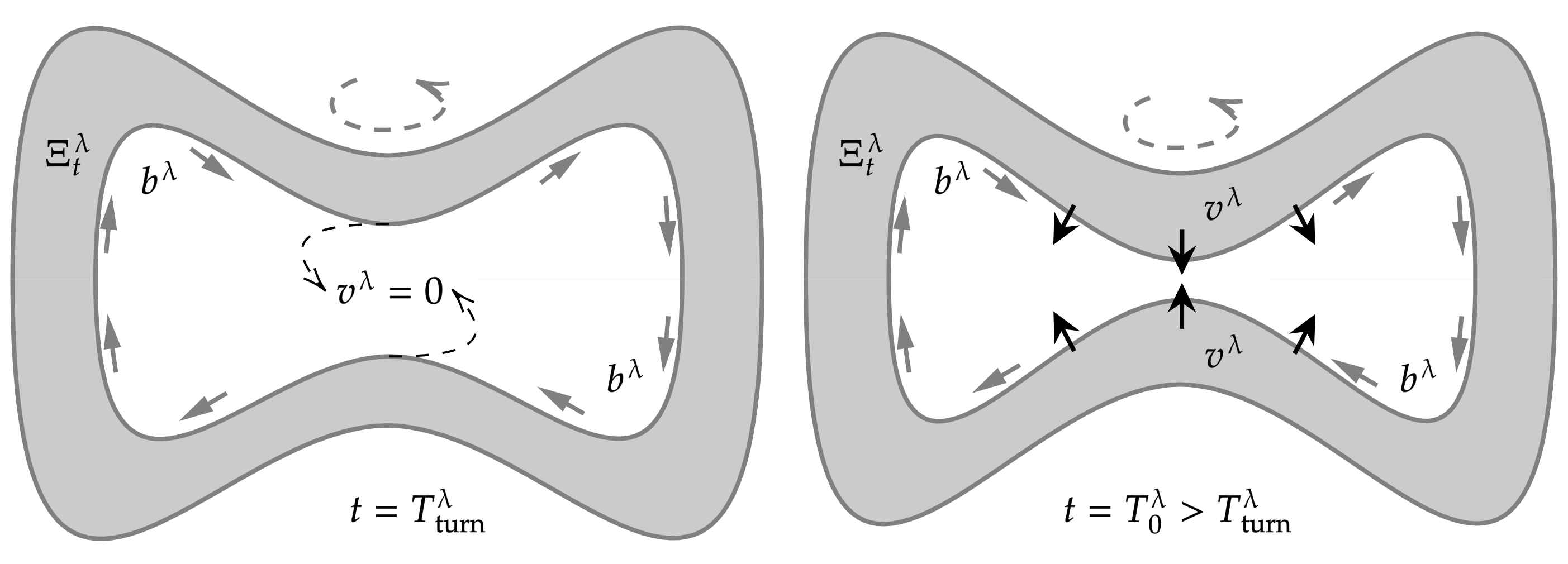}  
\end{subfigure}
\caption{The uniform lower bound for the turning time.}
\label{f15}
\end{figure}

To circumvent this difficulty, we establish a H\"{o}lder-type estimate for the temporal variation of the boundary velocity (cf. Proposition \ref{pro_5.1}), i.e.,
\begin{equation*}
\norm{v_\normal^\pr\Paren{x+\eta^\pr(x,t)\normal^\pr(x,0),t}-v_\normal^\pr(x,0)}_{L^\infty_x(\parXi^\pr_0)}\le Ct^\alpha, 
\end{equation*}
where $\eta^\pr$ denotes the height function and $\alpha\in (0,1)$.
This estimate prevents the prescribed inward normal velocity from changing sign before a time independent of $\pr$, and yields a uniform lower bound for the turning time $T^\pr_{\text{turn}}$ (see Fig.\,\ref{f15}, where $T_0^\pr$ is defined in \eqref{e:time_T0}):
\begin{equation*}
T^\pr_{\text{turn}}\ge \Paren{C_\pm/C}^{\frac{1}{\alpha}}>0,
\end{equation*}
where $C_\pm$ denotes a lower bound for the normal component of the initial velocity $v_\ddagger^\pr$. This ensures a common time interval on which the energy functionals remain uniformly bounded.

\begin{remark}
For a general family of fluid domains and initial data, obtaining a uniform existence time is difficult when the initial boundary separation has no uniform positive lower bound, since unconstrained boundary velocities may drive nearly touching boundary portions into immediate contact. In our construction, a uniform directional constraint is imposed on the approximate velocity fields, yielding a positive lower bound for the existence time without requiring uniform initial boundary separation. 
\end{remark}

\vspace{2mm}
\noindent\textbf{Step 4. Compactness analysis.} At the uniform initial time $t=-T_\star$, the approximate boundaries remain regular and non-self-intersecting, with geometric and $H^3$ bounds uniform in $\pr$. These estimates allow us to apply the weak compactness argument and extract a subsequence converging to regular initial data  $\Paren{u_{-T_\star}, h_{-T_\star}, \Omega_{-T_\star}}$. Evolving these data forward under the ideal MHD equations yields a solution whose free boundary self-intersects at $t=0$.

\begin{remark}
Although the construction is presented in an axisymmetric setting, its local structure suggests possible extensions to more general geometries and contact configurations. 
\end{remark}
 
\vspace{2mm}
\noindent\textbf{Structure of the paper.}
The remainder of this paper is organized as follows. In Section \ref{sec2}, we compute the time evolution of the energy functionals. Section \ref{sec3} establishes the a priori estimates and the blow-up theory (Theorems \ref{thm_main1} and \ref{thm_main2}) via uniform pressure and reverse energy estimates. In Section \ref{sec4}, we introduce the geometric and analytical setup for the self-intersection singularity by constructing a sequence of regular approximating fluid domains. Finally, Section \ref{sec5} proves the existence of this singularity (Theorem \ref{thm_main3}) using uniform backward-in-time estimates.

\section{Time evolution of the auxiliary energy functional}\label{sec2}
Throughout this section and the sequel, we adopt the Einstein summation convention and utilize the notation $S\star T$ (following \cite{Hamilton1982,Mantegazza2002}) to denote a generic linear combination of contractions between tensors $S$ and $T$ with constant coefficients. More specifically, for integers $k,l\ge 0$ and arbitrary tensor fields $f$ and $g$, the notation $\nabla^k f\star\nabla^l g$ represents a generic contraction involving derivatives of $f$ up to order $k$ and derivatives of $g$ up to order $l$. By convention, this notation encompasses lower-order derivatives (including the fields themselves) but strictly excludes isolated linear terms of the form $\nabla^i f$.

In this section, we estimate the time derivative of the auxiliary energy functional
\begin{align}
\energy(t)\coloneqq\frac12&\left(\norm{\DT^2u}_{L^2(\Omega_t)}^2+\norm{\DT^2h}_{L^2(\Omega_t)}^2+\norm{\Bdnabla\Paren{\DT u\cdot\normal}}_{L^2(\parOmega_t)}^2\right.\nonumber\\
&\quad\left.+\norm{\nabla^2\Paren{\vorticity u}}_{L^2(\Omega_t)}^2+\norm{\nabla^2\Paren{\vorticity h}}_{L^2(\Omega_t)}^2\right),\label{eq_energy} 
\end{align}
in terms of the full energy functional $\Energy(t)$. The complementary reverse estimate, namely
\begin{equation*}
\Energy(t)\le C(\bd)\Paren{1+\energy(t)},
\end{equation*} 
will be established in Section \ref{sec_closetheestimate}. The main result of this section is given in the following proposition.
\begin{proposition}\label{lem_ddt} 
Assume that the a priori assumptions \eqref{eq_aprioriassumption} hold 
for some time $T>0$. Then we have
\begin{equation}\label{eq_pro1eq}
\module{\frac{d}{dt}\energy(t)}\le C(\bd)\Paren{1+\norm{\nabla  p}_{H^1(\Omega_t)}^2}\Energy(t),\quad t\le T. 
\end{equation}
\end{proposition}
\begin{proof} 
We shall repeatedly apply the Reynolds transport formulas: 
\begin{align}
\frac{d}{dt}\int_{\Omega_t}fdx&=\int_{\Omega_t}\DT fdx,\label{eq_RT1}\\
\frac{d}{dt}\int_{\parOmega_t}fdS&=\int_{\parOmega_t}\DT f+f\Bddiv udS.\label{eq_RT2}
\end{align} 

\vspace{2mm}
\noindent\textbf{Step 1. Time evolution of $\norm{\DT^2u}_{L^2(\Omega_t)}^2$ and $\norm{\DT^2h}_{L^2(\Omega_t)}^2$.}

From \eqref{eq_MHD1}, we obtain 
\begin{align*}
\frac{d}{dt}\Paren{\frac12\norm{\DT^2u}_{L^2(\Omega_t)}^2}={}&\int_{\Omega_t}\DT^{3}u\cdot\DT^{2}udx\\
={}&\int_{\Omega_t}\DT^{2}\Paren{-\nabla p+h\cdot\nabla h}\cdot\DT^{2}udx\\
={}&-\int_{\Omega_t}[\DT^{2},\nabla]p\cdot\DT^{2}udx+\int_{\Omega_t}\DT^{2}\Paren{h^j\partial_jh_i}\DT^{2}u^idx-\int_{\Omega_t}\nabla\DT^{2}p\cdot\DT^{2}udx. 
\end{align*}
By direct computation, we have
\begin{align*}
\DT^{2}\Paren{h^j\partial_jh_i}={}&h^j\partial_j\DT^{2}h_i+\sum_{k=0}^{1}\DT^kh^j[\DT^{2-k},\partial_j]h_i+\sum_{k=1}^{2}\DT^kh^j\partial_j\DT^{2-k}h_i,\\
-\nabla\DT^{2}p\cdot\DT^{2}u={}&-\divergence\Paren{\DT^{2}p\DT^{2}u}+\DT^{2}p\divergence\DT^{2}u,
\end{align*}
and applying the divergence theorem, we deduce that
\begin{equation*}
\frac{d}{dt}\Paren{\frac12\norm{\DT^2u}_{L^2(\Omega_t)}^2}=-\int_{\Omega_t}[\DT^{2},\nabla]p\cdot\DT^{2}udx+\Upsilon_{1}(t)+\Theta_1(t)+\Theta_2(t)+\Pi_{1}(t)+\Psi(t),
\end{equation*} 
where
\begin{align*}
\Upsilon_{1}(t)\coloneqq{}&\int_{\Omega_t}h^j\partial_j\DT^{2}h_i\DT^{2}u^idx,\quad\Theta_1(t)\coloneqq\sum_{k=0}^{1}\int_{\Omega_t}\DT^kh^j[\DT^{2-k},\partial_j]h_i\DT^{2}u^idx,\\
\Theta_2(t)\coloneqq{}&\sum_{k=1}^{2}\int_{\Omega_t}\DT^kh^j\partial_j\DT^{2-k}h_i\DT^{2}u^idx,\quad\Pi_{1}(t)\coloneqq-\int_{\parOmega_t}\DT^{2}p\Paren{\DT^{2}u\cdot\normal}dS,\\
\Psi(t)\coloneqq{}&\int_{\Omega_t} \DT^{2}p\divergence\DT^{2}udx.
\end{align*}
Similarly, for the magnetic field, it follows from \eqref{eq_MHD2} that
\begin{equation*} 
\frac{d}{dt}\Paren{\frac12\norm{\DT^2h}_{L^2(\Omega_t)}^2}=\int_{\Omega_t}\DT^{2}\Paren{h^j\partial_ju^i}\DT^{2}h_idx=\Upsilon_2(t)+\Theta_3(t)+\Theta_4(t),
\end{equation*}
where
\begin{align*}
\Upsilon_2(t)&\coloneqq\int_{\Omega_t}h^j\partial_j\DT^{2}u^i \DT^{2}h_idx,\quad\Theta_3(t)\coloneqq\sum_{k=0}^{1}\int_{\Omega_t}\DT^k  h^j[\DT^{2-k},\partial_j]u^i\DT^{2}h_idx,\\
\Theta_4(t)&\coloneqq\sum_{k=1}^{2}\int_{\Omega_t}\DT^kh^j\partial_j\DT^{2-k}u^i\DT^{2}h_idx.
\end{align*}
Recalling from \eqref{eq_MHD3} and \eqref{eq_MHD4} that the magnetic 
field is solenoidal and tangent to the free boundary, integration by 
parts yields
\begin{equation*}
\Upsilon_{1}(t)+\Upsilon_{2}(t)=0.
\end{equation*}
Consequently, we conclude that
\begin{equation}\label{eq_step1}
\frac{d}{dt}\Paren{\frac12\norm{\DT^2u}_{L^2(\Omega_t)}^2+\frac12\norm{\DT^2h}_{L^2(\Omega_t)}^2}=-\int_{\Omega_t}[\DT^{2},\nabla]p\cdot\DT^{2}udx+\sum_{i=1}^{4}\Theta_i(t)+\Pi_{1}(t)+\Psi(t).
\end{equation}

\vspace{2mm}
\noindent\textbf{Step 2. Time evolution of 
$\norm{\Bdnabla\Paren{\DT u\cdot\normal}}_{L^2(\parOmega_t)}^2$.}

We apply the commutator formula
\begin{equation*}
[\DT,\Bdnabla](\,\cdot\,)=-\Paren{\Bdnabla u}^\top\Bdnabla(\, \cdot\,)
\end{equation*}
to deduce that
\begin{align}
&\frac{d}{dt}\Paren{\frac12\norm{\Bdnabla\Paren{\DT u\cdot\normal}}_{L^2(\parOmega_t)}^2}\nonumber\\
={}&\int_{\parOmega_t}\Paren{\Bdnabla\DT+[\DT,\Bdnabla]}\Paren{\DT u\cdot\normal}\cdot\Bdnabla\Paren{\DT u\cdot\normal}dS+\frac{1}{2}\int_{\parOmega_t}\module{\Bdnabla\Paren{\DT u\cdot\normal}}^2\Bddiv udS\nonumber\\ 
={}&\underbrace{\int_{\parOmega_t}\Bdnabla\Paren{\DT^{2}u\cdot\normal}\cdot\Bdnabla\Paren{\DT u\cdot\normal}dS}_{\eqqcolon\Pi_2(t)}+\int_{\parOmega_t}\Bdnabla\Paren{\DT u\cdot\DT\normal}\cdot \Bdnabla\Paren{\DT u\cdot\normal}dS\nonumber\\
&-\int_{\parOmega_t}\Paren{\Bdnabla u}^\top\Bdnabla\Paren{\DT u\cdot \normal}\cdot\Bdnabla\Paren{\DT u\cdot \normal}dS+\frac{1}{2}\int_{\parOmega_t}\module{\Bdnabla\Paren{\DT u\cdot \normal}}^2\Bddiv udS.\label{eq_last3}
\end{align}
By virtue of the assumed $H^3$ regularity of the velocity and the magnetic 
fields, the a priori assumptions \eqref{eq_aprioriassumption}, and 
\cite[Lemma B.2]{Luo2024}, we have
\begin{equation}\label{eq_Bdnabla_uinfty}
\norm{\Bdnabla u}_{L^\infty(\parOmega_t)}\le \norm{\nabla u}_{L^\infty(\Omega_t)}\le C(\bd),
\end{equation} 
where we used the fact that $\nabla u$ can be continuously extended 
to $\overline{\Omega_t}$ by the Sobolev embedding theorem.
Therefore, applying Cauchy's inequality, the last three terms in 
\eqref{eq_last3} can be controlled by 
\begin{align*}
\module{\text{last three terms in \eqref{eq_last3}}}&\le C\Paren{\norm{\Bdnabla u}_{L^\infty(\parOmega_t)}+1}\norm{\Bdnabla\Paren{\DT u\cdot\normal}}_{L^2(\parOmega_t)}^2+\norm{\Bdnabla\Paren{\DT u\cdot\DT\normal}}_{L^2(\parOmega_t)}^2\\
&\le C(\bd)\Energy(t)+\Phi_1(t),
\end{align*}
where
\begin{equation*}
\Phi_1(t)\coloneqq \norm{\Bdnabla\Paren{\DT u\cdot\DT\normal}}_{L^2(\parOmega_t)}^2.
\end{equation*}

Next, we proceed to treat the first term $\Pi_2(t)$. By the divergence 
theorem, it follows that
\begin{equation}\label{eq_divbd}
\int_{\parOmega_t}\Bdnabla f\cdot\Bdnabla gdS=-\int_{\parOmega_t}\LB fgdS,
\end{equation}
where the Laplace-Beltrami operator is defined by 
\begin{equation*}
\LB (\,\cdot\,)\coloneqq \Bddiv\Bdnabla (\,\cdot\,).
\end{equation*}
We rewrite $\Pi_2(t)$ as
\begin{equation*}
\Pi_2(t)=-\int_{\parOmega_t}\DT^{2}u\cdot\normal\LB\Paren{\DT u\cdot\normal}dS.
\end{equation*}
Noting from \eqref{eq_MHD4} that $h_\normal=0$ and thus the tangential derivative on the free boundary is independent of the magnetic field, we can apply the error formula for the pressure established in \cite[Lemma 2.5]{Hao2025a} 
\begin{equation}\label{eq_err_DT2p}
\DT^2p=-\LB\Paren{\DT  u\cdot\normal}+\err_p,
\end{equation}
to deduce that
\begin{equation*}
\Pi_1(t)+\Pi_2(t)=-\int_{\parOmega_t}\err_p\Paren{\DT^2u\cdot\normal}dS.
\end{equation*} 
Here $\Pi_1(t)$ appears in \eqref{eq_step1}, and the error term is given by
\begin{equation*}  
\err_p=\Bdnabla p\cdot\DT u+ \Bdnabla^2u\star\Bdnabla u\star\normal-|\sff|^2\DT u\cdot\normal+\Bdnabla u\star\Bdnabla u\star\sff.
\end{equation*}

Applying the normal trace theorem (cf. \cite[Theorem 3.1]{Cheng2007} 
and \cite[Lemma 5.1]{Coutand2010}), it is clear that 
\begin{align*}
\module{\Pi_1(t)+\Pi_2(t)} &\le C\Paren{\norm{\DT^{2}u\cdot\normal}_{H^{-\frac{1}{2}}(\parOmega_t)}^2+\norm{\err_p}_{H^{\frac{1}{2}}(\parOmega_t)}^2}\\
&\le C\Paren{\Energy(t)+\norm{\divergence\DT^{2}u}_{H^{-1}(\Omega_t)}^2+\norm{\err_p}_{H^{\frac{1}{2}}(\parOmega_t)}^2},
\end{align*}
and by a duality argument, 
\begin{align*}
\norm{\divergence\DT^{2}u}_{H^{-1}(\Omega_t)}\le{}&\sup \Brace{\module{\int_{\Omega_t}\divergence \DT^{2}u vdx}:v\in H^1_0(\Omega_t),\norm{v}_{H^1_0(\Omega_t)}\le 1}\\
\le{}&\sup \Brace{\module{\int_{\Omega_t}\DT^{2}u\cdot \nabla vdx}:v\in H^1_0(\Omega_t),\norm{v}_{H^1_0(\Omega_t)}\le 1}\\
\le{}& \norm{\DT^{2}u}_{L^{2}(\Omega_t)}\\
\le{}& \sqrt{\Energy(t)}. 
\end{align*} 

We claim that
\begin{equation}\label{eq_errp}
\norm{\err_p}_{H^{\frac 12}(\parOmega_t)}^2\le C\Paren{1+\norm{\nabla p}_{H^1(\Omega_t)}^2} \Energy(t).
\end{equation}
To verify this, we recall that for $u\in H^1(\Omega)$, it holds 
(cf. \cite[Appendix]{Hao2025}):
\begin{align}
\norm{u}_{H^{\frac 12}(\parOmega)}&\le \norm{u}_{L^2(\parOmega)}+\norm{\nabla u}_{L^2(\Omega)},\label{eq_Hfrac121}\\
\norm{u}_{H^1(\Omega)}&\le C\Paren{\norm{\lpe u}_{L^2(\Omega)}+\norm{u}_{H^{\frac 12}(\parOmega)}}.\label{eq_Hfrac123}
\end{align} 
We denote
\begin{equation*}
\err_p=\underbrace{\Bdnabla p\cdot\DT u}_{\eqqcolon K_1}+ \underbrace{\Bdnabla^2u\star\Bdnabla u\star\normal}_{\eqqcolon K_2}-\underbrace{|\sff|^2\DT u\cdot\normal}_{\eqqcolon K_3} +\underbrace{\Bdnabla u\star\Bdnabla u\star \sff}_{\eqqcolon K_4}.
\end{equation*}
To estimate $K_1$, by \eqref{eq_MHD1}, the a priori assumptions \eqref{eq_aprioriassumption}, the trace theorem, and \eqref{eq_Hfrac121}, we obtain
\begin{align*}
\norm{K_1}_{H^{\frac 12}(\parOmega_t)}^2\le{}&C\Paren{\norm{\Bdnabla p\cdot\DT u}_{L^{2}(\parOmega_t)}^2+\norm{\nabla\Paren{\nabla p\star\DT u}}_{L^{2}(\Omega_t)}^2}\\
\le{}& C\bigg(\norm{\Bdnabla p}_{L^{4}(\parOmega_t)}^2\norm{\DT u}_{L^{4}(\parOmega_t)}^2+\norm{\nabla\Paren{-\DT u+h\cdot\nabla h}\star\Paren{-\nabla p+h\cdot\nabla h}}_{L^{2}(\Omega_t)}^2\\
&\qquad+\norm{\nabla p\star\nabla\DT u}_{L^{2}(\Omega_t)}^2\bigg)\\
\le{}&C(\bd)\bigg[\norm{\nabla p}_{H^1(\Omega_t)}^2\Energy(t)+\norm{\nabla\DT u\star\nabla h\star h}_{L^{2}(\Omega_t)}^2+\norm{\nabla\Paren{h\cdot\nabla h}\star\DT u}_{L^{2}(\Omega_t)}^2\\
&\qquad\quad\ +\Paren{\norm{\nabla p}_{L^{3}(\Omega_t)}^2+\norm{\nabla h}_{L^{3}(\Omega_t)}^2}\norm{\nabla^2 h}_{L^{6}(\Omega_t)}^2+\norm{\nabla p}_{L^{6}(\Omega_t)}^2\norm{\nabla\DT u}_{L^{3}(\Omega_t)}^2\bigg]\\ 
\le{}&C(\bd)\Paren{1+\norm{\nabla p}_{H^{1}(\Omega_t)}^2}\Energy(t).
\end{align*}
Similarly, invoking \eqref{eq_Hfrac121} again, we find that 
\begin{equation*}
\norm{K_2}_{H^{\frac 12}(\parOmega_t)}^2 \le C\Paren{\norm{\Bdnabla^2u\star\Bdnabla u\star\normal}_{L^2(\parOmega_t)}^2+\norm{\nabla\Paren{\Bdnabla^2u\star\Bdnabla u\star\normal}}_{L^2(\Omega_t)}^2}.
\end{equation*}
It suffices to estimate the second term using \eqref{eq_aprioriassumption} 
and the Gagliardo-Nirenberg inequality:
\begin{equation*}
\norm{\nabla^2u\star\nabla^2u\star\underbrace{\normal\star\normal\star\cdots\star\normal}_{\text{finite $\star$ product}}}_{L^2(\Omega_t)}^2\le C\norm{\nabla^2u}_{L^4(\Omega_t)}^4\le C\norm{\nabla u}_{L^\infty(\Omega_t)}^2\norm{\nabla u}_{H^2(\Omega_t)}^2\le C(\bd)\Energy(t).
\end{equation*}
Regarding $K_3$, it is clear from \eqref{eq_aprioriassumption} and the 
trace theorem that
\begin{equation*}
\norm{\DT u\cdot\normal}_{L^{2}(\parOmega_t)}^2\le C\norm{\DT u}_{H^{1}(\Omega_t)}^2\le C \Energy(t),\ \text{and thus}\ \norm{\DT u\cdot\normal}_{H^1(\parOmega_t)}^2 \le C \Energy(t).
\end{equation*} 
Furthermore, Sobolev embedding yields
\begin{equation}\label{eq_sffL4}
\norm{\sff}_{L^4(\parOmega_t)}\le C\norm{\sff}_{H^{\frac 12}(\parOmega_t)}\le C(\bd),
\end{equation}  
which, combined with the Kato-Ponce inequality, gives
\begin{align*}
\norm{K_3}_{H^{\frac 12}(\parOmega_t)}^2\le{}& C\norm{|\sff|^2}_{L^4(\parOmega_t)}^2\norm{\DT u\cdot\normal}_{W^{\frac{1}{2},4}(\parOmega_t)}^2+C\norm{|\sff|^2}_{W^{\frac12,4}(\parOmega_t)}^2\norm{\DT u\cdot\normal}_{L^{4}(\parOmega_t)}^2\\
\le{}& C\norm{|\sff|^2}_{H^1(\parOmega_t)}^2\norm{\DT u\cdot\normal}_{H^{1}(\parOmega_t)}^2\\
\le{}& C(\bd)\Paren{1+\norm{\Bdnabla \sff\star \sff}_{L^2(\parOmega_t)}^2}\Energy(t)\\
\le{}& C(\bd)\Paren{1+\norm{\Bdnabla \sff}_{L^4(\parOmega_t)}^2}\Energy(t).
\end{align*}
To estimate $\norm{\Bdnabla \sff}_{L^4(\parOmega_t)}$, we need to apply the following regularity result (cf. \cite[Proposition
2.12]{Julin2024} and \cite[Lemma 2.6]{Hao2025a}):
\begin{lemma}\label{lem_sffcontrolledbymc}
Let $\Omega$ be a bounded domain with a $C^{1,\alpha}$ boundary 
($\alpha \in (0,1)$). Then the second fundamental form $\sff$ 
and the mean curvature $\mc$ satisfy the following property: 
for every $p \in (1, \infty)$, we have
\begin{equation*}
\norm{\sff}_{L^p(\parOmega)} \le C \Paren{1+\norm{\mc}_{L^p(\parOmega)}}.
\end{equation*}
If, in addition, $\norm{\sff}_{L^4(\parOmega)} \le \bd$ for some 
positive constant $\bd$, then we have
\begin{equation*}
\norm{\sff}_{H^{k}(\parOmega)}\le C(\bd)\Paren{1+\norm{\mc}_{H^{k}(\parOmega)}},
\quad k\in\Brace{1/2, 1, 3/2, 2}.
\end{equation*}
\end{lemma}
Applying \eqref{eq_sffL4}, Lemma \ref{lem_sffcontrolledbymc}, and the 
trace theorem once more, we obtain
\begin{equation*}
\norm{K_3}_{H^{\frac 12}(\parOmega_t)}^2\le C(\bd)\Paren{1+\norm{\nabla p}_{H^{1}(\Omega_t)}^2}\Energy(t).
\end{equation*}
The term $K_4$ can be estimated analogously. Thus, \eqref{eq_errp} is established. Consequently,
\begin{equation*}
\module{\Pi_1(t)+\Pi_2(t)}\le C(\bd)\Paren{1+\norm{\nabla p}_{H^1(\Omega_t)}^2} \Energy(t),
\end{equation*}
and returning to \eqref{eq_step1}, we deduce that
\begin{align}
&\module{\frac{d}{dt}\frac12\Paren{\norm{\DT^2u}_{L^2(\Omega_t)}^2+\norm{\DT^2h}_{L^2(\Omega_t)}^2+\norm{\Bdnabla\Paren{\DT u\cdot\normal}}_{L^2(\parOmega_t)}^2}}\nonumber\\
\le{}&C(\bd)\Paren{1+\norm{\nabla p}_{H^1(\Omega_t)}^2} \Energy(t)+\module{\sum_{i=1}^{4}\Theta_i(t)}+\Phi_1(t)+\Phi_2(t)+\module{\Psi(t)}, \label{eq_step2}
\end{align}
where
\begin{equation*}
\Phi_2(t)\coloneqq\norm{[\DT^{2},\nabla] p}_{L^2(\Omega_t)}^2.
\end{equation*}

\vspace{2mm}
\noindent\textbf{Step 3. Time evolution of 
$\norm{\nabla^2\Paren{\vorticity u}}_{L^2(\Omega_t)}^2$ and 
$\norm{\nabla^2\Paren{\vorticity h}}_{L^2(\Omega_t)}^2$.}

For the last two terms in $\energy(t)$, it follows from 
\cite[Lemma 2.10]{Hao2025} that
\begin{align*}
&\frac{d}{dt}\Paren{\frac12\norm{\nabla^2\Paren{\vorticity u}}_{L^2(\Omega_t)}^2+\frac12\norm{\nabla^2\Paren{\vorticity h}}_{L^2(\Omega_t)}^2}\\ 
={}&\int_{\Omega_t}\sum_{|\alpha|=1}\Paren{h\cdot\nabla}\nabla^{\alpha}\Paren{\vorticity h}:\nabla^{\alpha}\Paren{\vorticity u}dx+\int_{\Omega_t}\sum_{|\alpha|=1}\Paren{h\cdot\nabla}\nabla^{\alpha}\Paren{\vorticity u}:\nabla^{\alpha}\Paren{\vorticity h}dx\\
&+\int_{\Omega_t}\bigg[\nabla u\star \nabla^{2}\Paren{\vorticity u}+\nabla^{3}u\star\Paren{\vorticity u}+\sum_{|\beta|=2}\nabla^{1+\beta_1}h\star\nabla^{\beta_{2}}\Paren{\vorticity h}\\
&\quad\quad\quad\quad+\sum_{i\leq1}\nabla^{1+i}u\star\nabla\Paren{\vorticity u}\bigg]\star\nabla^{2}\Paren{\vorticity u}dx\\
&+\int_{\Omega_t}\bigg[\nabla u\star\nabla^2\Paren{\vorticity h}+\sum_{|\beta|=2}\nabla^{1+\beta_1}u\star\nabla^{1+\beta_{2}}h+\sum_{i\leq 1}\nabla^{1+i}u\star\nabla\Paren{\vorticity h}\bigg]\star\nabla^2\Paren{\vorticity h}dx.
\end{align*}
By virtue of \eqref{eq_MHD3} and \eqref{eq_MHD4}, applying integration by parts to the first two terms yields
\begin{align*}
&\int_{\Omega_t}\sum_{|\alpha|=1}\Paren{h\cdot\nabla}\nabla^{\alpha}\Paren{\vorticity h}:\nabla^{\alpha}\Paren{\vorticity u}dx+\int_{\Omega_t}\sum_{|\alpha|=1}\Paren{h\cdot\nabla}\nabla^{\alpha}\Paren{\vorticity u}:\nabla^{\alpha}\Paren{\vorticity h}dx\\
={}&\int_{\Omega_t}\sum_{|\alpha|=1}\Paren{h\cdot\nabla}\Bracket{\nabla^{\alpha}\Paren{\vorticity h}:\nabla^{\alpha}\Paren{\vorticity u}}dx\\
={}&-\int_{\Omega_t}\sum_{|\alpha|=1}\underbrace{\divergence h}_{=0}\Bracket{\nabla^{\alpha}\Paren{\vorticity h}:\nabla^{\alpha}\Paren{\vorticity u}}dx+\int_{\parOmega_t}\sum_{|\alpha|=1}\underbrace{h\cdot\normal}_{=0}\Bracket{\nabla^{\alpha}\Paren{\vorticity h}:\nabla^{\alpha}\Paren{\vorticity u}}dS\\
={}&0.
\end{align*} 
Consequently, invoking the a priori assumptions \eqref{eq_aprioriassumption}, 
we deduce that
\begin{align*}
&\module{\frac{d}{dt}\frac12\Paren{\norm{\nabla^2\Paren{\vorticity u}}_{L^2(\Omega_t)}^2+\norm{\nabla^2\Paren{\vorticity h}}_{L^2(\Omega_t)}^2}}\nonumber\\
\le{}& C\Paren{\norm{\nabla u}_{L^\infty(\Omega_t)}^2+\norm{\nabla h}_{L^{\infty}(\Omega_t)}^2+1} \Paren{\norm{u}_{H^3(\Omega_t)}^2+\norm{h}_{H^3(\Omega_t)}^2}\nonumber\\
\le{}& C(\bd)\Energy(t).
\end{align*} 
Combining this estimate with \eqref{eq_step2}, we arrive at
\begin{equation}\label{eq_step3}
\module{\frac{d}{dt}\energy(t)}
\le C(\bd)\Paren{1+\norm{\nabla p}_{H^1(\Omega_t)}^2} \Energy(t)+\module{\sum_{i=1}^{4}\Theta_i(t)}+\Phi_1(t)+\Phi_2(t)+\module{\Psi(t)}.
\end{equation}  

\vspace{2mm}
\noindent\textbf{Step 4. Estimates of $\Theta_i(t),i=1,\cdots,4$.} 

Let $(f,g)$ be either $(u,h)$ or $(h,u)$. By \eqref{eq_MHD2} and the 
a priori assumptions \eqref{eq_aprioriassumption}, it follows that
\begin{align*}
&\module{\sum_{k=1}^{2}\int_{\Omega_t}\DT^kh^j\partial_j\DT^{2-k}f_i\DT^{2}g^idx}\\
\le{}&C\Paren{\sum_{k=1}^{2}\norm{\DT^kh^j\partial_j\DT^{2-k}f}_{L^2(\Omega_t)}^2+\norm{\DT^2g}_{L^2(\Omega_t)}^2}\\
\le{}&C\Paren{\Energy(t)+\norm{h\cdot\nabla u}_{L^6(\Omega_t)}^2\norm{\nabla\DT f}_{L^3(\Omega_t)}^2+\norm{\DT^2h}_{L^2(\Omega_t)}^2\norm{\nabla f}_{L^\infty(\Omega_t)}^2}\\ 
\le{}&C(\bd)\Energy(t).
\end{align*} 
Combining this with the commutator formulas (cf. \cite[Lemma 2.2]{Hao2025}):
\begin{align}
[\DT,\nabla](\,\cdot\,)&=-\Paren{\nabla u}^\top\nabla(\,\cdot\,),\label{eq_Dt,nabla}\\
[\DT^j,\nabla](\,\cdot\,)&=\sum_{2\le m\le j+1}\sum_{|\beta|\leq j+1-m}\nabla\DT^{\beta_1}u\star\cdots\star\nabla\DT^{\beta_{m-1}}u\star\nabla\DT^{\beta_{m}}(\,\cdot\,),\quad j\ge 2,\nonumber
\end{align}
we deduce that
\begin{align*}
&\module{\sum_{k=0}^{1}\int_{\Omega_t}\DT^kh^j[\DT^{2-k},\partial_j]f_i \DT^{2}g^idx}\\ 
\le{}&C\Paren{\norm{\DT^2g}_{L^2(\Omega_t)}^2+\norm{h^j[\DT^{2},\partial_j]f}_{L^2(\Omega_t)}^2+\norm{\DT h^j[\DT,\partial_j]f}_{L^2(\Omega_t)}^2}\\
\le{}&C\Big(\Energy(t)+\norm{h\star
\Paren{\nabla \DT u\star \nabla f+\nabla u \star \nabla \DT f+\nabla u\star \nabla u\star \nabla f}}_{L^2(\Omega_t)}^2\\
&\qquad+\norm{h\star\nabla u\star\nabla u\star \nabla f}_{L^2(\Omega_t)}^2 \Big)\\
\le{}&C(\bd) \Energy(t).
\end{align*}
Consequently, we obtain
\begin{equation}\label{eq_step4}
\module{\sum_{i=1}^{4}\Theta_i(t)}\le C(\bd) \Energy(t).
\end{equation}

\vspace{2mm}
\noindent\textbf{Step 5. Estimates of $\Phi_1(t)$ and $\Phi_2(t)$.}

Invoking the a priori assumptions \eqref{eq_aprioriassumption}, the 
fundamental identity
\begin{equation}\label{eq_DTn}
\DT \normal=-\Paren{\Bdnabla u}^\top \normal,
\end{equation}
the trace theorem, and \eqref{eq_Bdnabla_uinfty}, we deduce that
\begin{align*}
\Phi_1(t)\le{}& \norm{\Bdnabla \DT u\star\DT\normal}_{L^2(\parOmega_t)}^2+\norm{\DT u\star\Bdnabla\DT\normal}_{L^2(\parOmega_t)}^2\\
\le{}& C\Paren{\norm{\Bdnabla u\star\normal}_{L^\infty(\parOmega_t)}^2\norm{\DT u}_{H^{\frac32}(\Omega_t)}^2+\norm{\DT u\star\Bdnabla^2u\star\normal}_{L^2(\parOmega_t)}^2+\norm{\DT u\star\Bdnabla u\star \sff}_{L^2(\parOmega_t)}^2}\\
\le{}& C(\bd)\Paren{\Energy(t)+\norm{\DT u\star\Bdnabla^2u}_{L^2(\parOmega_t)}^2+\norm{\DT u\star \sff}_{L^2(\parOmega_t)}^2}\\
\le{}& C(\bd) \Energy(t).
\end{align*} 
In the final step, we utilized  Sobolev embedding
\begin{equation*}
\norm{\DT u\star\sff}_{L^2(\parOmega_t)}^2\le C\norm{\DT u}_{L^4(\parOmega_t)}^2\norm{\sff}_{L^4(\parOmega_t)}^2\le  C(\bd)\Energy(t),
\end{equation*}
along with the estimate
\begin{align*}  
\norm{\DT u\star \Bdnabla^2u}_{L^2(\parOmega_t)}^2&\le \norm{\DT u}_{L^4(\parOmega_t)}^2\norm{\Bdnabla^2u}_{L^4(\parOmega_t)}^2\\
&\le \Paren{\norm{\nabla p}_{H^{1}(\Omega_t)}^2+\norm{\nabla h}_{L^{\infty}(\parOmega_t)}^2\norm{h}_{H^{1}(\Omega_t)}^2}\norm{u}_{H^{3}(\Omega_t)}^2\\
&\le C(\bd)\Paren{1+\norm{\nabla p}_{H^1(\Omega_t)}^2}\Energy(t),
\end{align*} 
which follows from \eqref{eq_sffL4}, the bound $\norm{\nabla h}_{L^\infty(\parOmega_t)}\le C(\bd)$ as in \eqref{eq_Bdnabla_uinfty}, and the trace theorem.

Regarding $\Phi_2(t)$, we recall from \cite[Lemma 2.9]{Hao2025} that
\begin{equation*}
[\DT^{2},\nabla]p=\sum\limits_{i\leq 1}\nabla\DT^{i} u\star\nabla h\star h+\err_{\one}+\err_{u,h},
\end{equation*} 
where
\begin{align}
\err_{\one}&=\sum_{|\beta|\leq1,|\alpha|\le1}a_{\alpha,\beta}(\nabla u) \nabla\DT^{\beta_1}u\star\nabla^{\alpha_1}\DT^{\alpha_2+\beta_{2}}u,\label{eq_err21}\\
\err_{u,h}&=\sum_{|\alpha|\le 1}a_{\alpha}(\nabla u)\nabla^{1+\alpha_1}u\star \nabla^{\alpha_{2}} h\star h.\label{eq_err22}
\end{align}
Here, $a_{\alpha,\beta}(\nabla u)$ and $a_{\alpha}(\nabla u)$ denote finite $\star$ products of $\nabla u$.

Applying \eqref{eq_err22}, the Kato-Ponce inequality, 
and the a priori assumptions \eqref{eq_aprioriassumption}, we obtain 
\begin{equation}\label{eq_errnablahh}
\norm{\err_{u,h}}_{L^2(\Omega_t)}^2
\le C(\bd)\Energy(t),
\end{equation}
and consequently,
\begin{equation*}
\Phi_2(t)\le C(\bd)\Energy(t)+C\norm{\err_{\one}}_{L^2(\Omega_t)}^2.
\end{equation*}
We proceed to show that
\begin{equation}\label{eq_errtwo} 
\norm{\err_{\one}}_{L^{2}(\Omega_t)}^2 \le C(\bd)\Paren{1+\norm{\nabla p}_{H^{1}(\Omega_t)}^2} \Energy(t).
\end{equation}
Let us consider the case where $|\beta|=1$ and $|\alpha|=1$ in \eqref{eq_err21}. 
For $\beta_1=1,$ we need to estimate
\begin{equation*}
\norm{\underbrace{\nabla u\star\nabla u\star\cdots\star\nabla u}_{\text{finite $\star$ product}}\star\nabla\DT u}_{L^2(\Omega_t)}^2\le C(\bd)\Energy(t),
\end{equation*}
and
\begin{align*}
\norm{\underbrace{\nabla u\star\nabla u\star\cdots\star\nabla u}_{\text{finite $\star$ product}}\star\nabla \DT u\star \DT u}_{L^2(\Omega_t)}^2&\le C(\bd) \norm{\nabla \DT u}_{L^3(\Omega_t)}^2\norm{-\nabla p+ h\cdot\nabla h}_{L^6(\Omega_t)}^2\\
&\le C(\bd)\Paren{1+\norm{\nabla p}_{H^1(\Omega_t)}^2} \Energy(t)
\end{align*}
where we have used \eqref{eq_aprioriassumption} and \eqref{eq_MHD1}.
The remaining case where $\beta_2=1$ is easier to bound, and we omit the details. 
Thus, the estimate \eqref{eq_errtwo} is established, which leads to 
\begin{equation}\label{eq_step5} 
\Phi_1(t)+\Phi_2(t)\le C(\bd)\Paren{1+\norm{\nabla p}_{H^{1}(\Omega_t)}^2} \Energy(t). 
\end{equation} 

\vspace{2mm}
\noindent\textbf{Step 6. Estimate of $\Psi(t)$.}

Let $f$ be the solution to 
\begin{equation}\label{eq_elliptic1}
\begin{cases}
-\lpe f=\divergence \DT^{2}u, & \text{in}\ \Omega_t,\\
f=0, &  \text{on}\ \parOmega_t.
\end{cases}
\end{equation}
Integrating by parts, we obtain
\begin{equation*}
\Psi(t)=\int_{\Omega_t} \DT^{2}p\divergence\DT^{2}udx
=-\int_{\Omega_t}\lpe\DT^{2} p fdx-\int_{\parOmega_t}\DT^{2} p\partial_\normal fdS\eqqcolon \Psi_1(t)+\Psi_2(t).
\end{equation*} 
Applying integration by parts again, alongside \cite[Lemma 2.11]{Hao2025} and the divergence theorem, it follows that
\begin{align*}
\Psi_1(t)={}&\int_{\Omega_t}\bigg[\divergence\divergence\Paren{u\otimes\DT^{2}u}+\divergence\bigg(\err_{\one}+\sum_{\substack{i\leq 1}}\nabla\DT^{i}u\star\nabla h\star h+ \err_{u,h}\bigg)\bigg]fdx\\
&-\int_{\Omega_t}\divergence\DT^{2}\Paren{h\cdot\nabla h}fdx\\
={}&\underbrace{\int_{\Omega_t}\Paren{u\otimes \DT^{2}u}:\nabla^2f-\Paren{\err_{\one}+\err_{u,h}+\sum_{\substack{i\leq 1}}\nabla\DT^{i}u\star\nabla h\star h}\cdot\nabla fdx}_{\eqqcolon\Psi_1^a(t)}\\
&\underbrace{-\int_{\parOmega_t}u^i\DT^{2} u^j\partial_i f\normal_jdS}_{\eqqcolon\Psi_1^b(t)}\underbrace{-\int_{\Omega_t}\divergence\DT^{2}\Paren{h\cdot\nabla h}  fdx}_{\eqqcolon\Psi_1^c(t)},
\end{align*}
where $\err_{\one}$ and $\err_{u,h}$ are defined as in \eqref{eq_err21} and \eqref{eq_err22}.

Using \eqref{eq_errnablahh} and \eqref{eq_errtwo}, we can bound the first term by
\begin{equation*}
\module{\Psi_1^a(t)}\le C(\bd)\Paren{1+\norm{\nabla p}_{H^{1}(\Omega_t)}^2} \Energy(t)+C(\bd)\norm{f}_{H^2(\Omega_t)}^2,
\end{equation*}
and the second term can be estimated as
\begin{align*}
\module{\Psi_1^b(t)}={}&\module{\int_{\Omega_t}\divergence\Paren{u^i\DT^{2}u\partial_if}dx}\\
={}&\module{\int_{\Omega_t}\nabla u\star\DT^{2}u\star\nabla f+u\star\divergence\DT^{2}u\star\nabla f+u\star\DT^{2}u\star\nabla^2 fdx}\\
\le{}&C(\bd)\Paren{\norm{f}_{H^2(\Omega_t)}^2+\Energy(t)+\norm{\divergence\DT^{2}u}_{L^2(\Omega_t)}^2}.
\end{align*}

To estimate the last term $\Psi_1^c(t)$, we use \cite[Lemma 2.7]{Hao2025}, which gives
\begin{equation*}
\divergence\DT^{2}\Paren{h\cdot \nabla h}=\partial_i\partial_m\DT u^j \partial_jh^mh^i+\nabla^3 u\star h\star h+\operatorname{l.o.t.},
\end{equation*}
and consequently,
\begin{equation*}
\module{\Psi_1^c(t)}\le C\Paren{\module{\int_{\Omega_t}\partial_i\partial_m\DT u^j \partial_jh^mh^ifdx}+\norm{f}_{L^2(\Omega_t)}^2+\norm{\nabla^3u\star h\star h}_{L^2(\Omega_t)}^2+\module{\Psi^{cl}_1(t)}},
\end{equation*}
where $\Psi^{cl}_1(t)$ contains lower-order terms (at most $\nabla^2u$). 
Applying boundary condition \eqref{eq_MHD4} 
and the a priori assumptions \eqref{eq_aprioriassumption}, we have
\begin{align*}
\module{\int_{\Omega_t}\partial_i\partial_m\DT u^j\partial_jh^mh^ifdx}={}&\module{\int_{\Omega_t}\partial_m\DT  u^j\partial_i \partial_jh^mh^if+\partial_m\DT u^j \partial_jh^mh^i\partial_ifdx}\\
\le{}&C(\bd)\Paren{\Energy(t)+\norm{f}_{H^1(\Omega_t)}^2}.
\end{align*}  
Furthermore, we have
\begin{equation*}
\norm{\nabla^3u\star h\star h}_{L^2(\Omega_t)}^2\le C(\bd)\Energy(t),
\end{equation*}
and $\Psi^{cl}_1(t)$ can be estimated in the same fashion as before. 
Combining these with the elliptic estimate for \eqref{eq_elliptic1},
\begin{equation*}
\norm{f}_{H^2(\Omega_t)}^2\le C\norm{\divergence \DT^{2}u}_{L^2(\Omega_t)}^2,
\end{equation*}
we obtain
\begin{equation*}
\module{\Psi_1(t)}\le C(\bd)\Paren{1+\norm{\nabla p}_{H^{1}(\Omega_t)}^2} \Energy(t)+C(\bd)\norm{\divergence\DT^{2}u}_{L^2(\Omega_t)}^2.
\end{equation*}

Note that by direct calculation, we can write
\begin{equation*}
\divergence \DT^{2}u=\sum_{|\beta|\leq 1}\nabla\DT^{\beta_1}u\star\nabla\DT^{\beta_2}u.                      
\end{equation*}
Although it suffices to bound $\norm{\divergence \DT^{2}u}_{L^2(\Omega_t)}$ 
at this stage, we estimate
$\norm{\divergence \DT^{2}u}_{H^{\frac 12}(\Omega_t)}$ for subsequent control of $\Psi_2(t)$. 
To this end, we restrict our attention to the case where $|\beta|=1$. 
From \eqref{eq_aprioriassumption} and the Kato-Ponce inequality, we see that
\begin{equation*}
\norm{\nabla u\star\nabla\DT u}_{H^{\frac 12}(\Omega_t)}
\le C\Paren{\norm{\nabla u}_{L^\infty(\Omega_t)}\norm{\nabla\DT u}_{H^{\frac 12}(\Omega_t)}+ \norm{\nabla u}_{W^{\frac 12,12}(\Omega_t)}\norm{\nabla \DT u}_{L^{\frac{12}{5}}(\Omega_t)}}. 
\end{equation*}
It is clear that
\begin{equation*}
\norm{\nabla u}_{L^\infty(\Omega_t)}\norm{\nabla \DT u}_{H^{\frac 12}(\Omega_t)}\le C(\bd)\sqrt{\Energy(t)}.
\end{equation*}
Invoking \eqref{eq_aprioriassumption} and the Gagliardo-Nirenberg 
inequality, we obtain
\begin{equation*}
\norm{\nabla u}_{W^{\frac 12,12}(\Omega_t)}\le C\sqrt{\norm{\nabla u}_{W^{1,6}(\Omega_t)}\norm{\nabla u}_{L^\infty(\Omega_t)}}\le C(\bd)\sqrt[4]{\Energy(t)},
\end{equation*}
and
\begin{align*}
\norm{\nabla\DT u}_{L^{\frac{12}{5}}(\Omega_t)}&\le C\norm{\nabla\DT u}_{H^{\frac14}(\Omega_t)}\\
&\le C\sqrt{\norm{\nabla\DT u}_{L^2(\Omega_t)}\norm{\nabla\DT u}_{H^{\frac12}(\Omega_t)}}\\
&\le C\sqrt{\norm{\nabla^2p}_{L^2(\Omega_t)}+\norm{\nabla\Paren{h\cdot\nabla h}}_{L^2(\Omega_t)}}\sqrt[4]{\Energy(t)}\\
&\le C(\bd)\sqrt{1+\norm{\nabla^2p}_{L^2(\Omega_t)}}\sqrt[4]{\Energy(t)}.
\end{align*}
As a consequence, we have
\begin{equation*}
\norm{\nabla u\star \nabla \DT u}_{H^{\frac 12}(\Omega_t)}^2\le C(\bd)\Paren{1+\norm{\nabla^2p}_{L^2(\Omega_t)}}\Energy(t),
\end{equation*} 
and it follows that
\begin{equation}\label{eq_err111}
\norm{\divergence \DT^{2}u}_{H^{\frac{1}{2}}(\Omega_t)}^2 \le C(\bd) \Paren{1+\norm{\nabla^2 p}_{L^2(\Omega_t)}}  \Energy(t).
\end{equation}
This yields
\begin{equation*}
\module{\Psi_1(t)}\le C(\bd)\Paren{1+\norm{\nabla p}_{H^{1}(\Omega_t)}^2} \Energy(t).
\end{equation*} 

We are left with $\Psi_2(t)$. Applying \eqref{eq_err_DT2p} and 
integrating by parts, we have
\begin{align*}
\Psi_2(t)&=-\int_{\parOmega_t}\Bracket{-\LB\Paren{\DT u\cdot\normal}+\err_p}\partial_\normal fdS\\
&=-\int_{\parOmega_t} \Bdnabla\Paren{\DT u\cdot\normal}\cdot\Bdnabla\partial_\normal fdS-\int_{\parOmega_t}\err_p\partial_\normal fdS.
\end{align*} 
Then, applying the elliptic estimate (cf. \cite[Proposition 3.8]{Julin2024}), 
along with \eqref{eq_err111} and \eqref{eq_errp}, we deduce that 
\begin{align*}
\module{\Psi_2(t)}&\le C\Paren{\norm{\Bdnabla\Paren{\DT u\cdot\normal}}_{L^2(\parOmega_t)}^2+\norm{\partial_\normal f}_{H^1(\parOmega_t)}^2+\norm{\err_p}_{L^2(\parOmega_t)}^2}\\
&\le C\Paren{\Energy(t)+\norm{\divergence \DT^{2}u}_{H^{\frac12}(\Omega_t)}^2+\norm{\err_p}_{L^2(\parOmega_t)}^2}\\
&\le C(\bd)\Paren{1+\norm{\nabla p}_{H^{1}(\Omega_t)}^2} \Energy(t).
\end{align*} 
We conclude that
\begin{equation}\label{eq_step6}
\module{\Psi(t)}\le\module{\Psi_1(t)}+\module{\Psi_2(t)}\le C(\bd)\Paren{1+\norm{\nabla p}_{H^{1}(\Omega_t)}^2} \Energy(t).
\end{equation}
Substituting the above calculations \eqref{eq_step4}, \eqref{eq_step5}, 
and \eqref{eq_step6} into \eqref{eq_step3}, the desired estimate 
\eqref{eq_pro1eq} follows.
\end{proof}
\section{The a priori estimates and blow-up}\label{sec3}
\subsection{Estimates for the pressure}\label{sec_pressureestimate}

In this section, we derive estimates for the pressure. For this purpose, 
we suppose that the a priori assumptions \eqref{eq_aprioriassumption} 
hold for some $T>0$.
\begin{lemma}\label{lem_pH1}
Assume that the a priori assumptions \eqref{eq_aprioriassumption} hold. 
Then we have
\begin{equation*}
\int_0^T \norm{p}_{H^1(\parOmega_t)}^2dt\le C(\bd)(1+T).
\end{equation*}
\end{lemma}
\begin{proof}
We define
\begin{equation*}
F(t):=\int_{\parOmega_t}p\Paren{\nabla u\normal\cdot\normal}dS+\vare\norm{p}_{L^2(\parOmega_t)}^2,
\end{equation*}
where the constant $\vare>0$ will be determined later. By Cauchy's 
inequality and the a priori assumptions \eqref{eq_aprioriassumption}, 
we deduce that
\begin{equation}\label{eq_F(t)}
F(t)\ge -C(\bd)\norm{p}_{L^1(\parOmega_t)}+\vare \norm{p}_{L^2(\parOmega_t)}^2\ge -C_\vare(\bd)+\frac \vare 2\norm{p}_{L^2(\parOmega_t)}^2.
\end{equation}
Applying the Reynolds transport formula \eqref{eq_RT2}, 
we obtain
\begin{equation*}
\frac{d}{dt}\int_{\parOmega_t}p\Paren{\nabla u\normal\cdot\normal}dS=\underbrace{\int_{\parOmega_t}p\Paren{\nabla u\normal\cdot\normal}\Bddiv u+\DT p\Paren{\nabla u\normal\cdot\normal}dS}_{\eqqcolon F_1(t)}+\underbrace{\int_{\parOmega_t}p\DT\Paren{\nabla u\normal\cdot\normal}dS}_{\eqqcolon F_2(t)}. 
\end{equation*} 

\vspace{2mm}
\noindent\textbf{Estimate  of $F_{1}(t)$.}

By virtue of \eqref{eq_Bdnabla_uinfty}, the first term can be bounded as
\begin{equation*}
\module{F_1(t)}\le \vare\Paren{\norm{p}_{L^2(\parOmega_t)}^2+\norm{\DT p}_{L^2(\parOmega_t)}^2}+C_\vare(\bd),
\end{equation*}
provided that $\vare>0$ is chosen sufficiently small. Note that
\begin{equation}\label{eq_Dtp}
\DT p=-\LB u\cdot\normal-2 \sff:\Bdnabla u=-\LB u_n-|\sff|^2u_n+\Bdnabla p\cdot u,
\end{equation}
which implies
\begin{equation}\label{eq_DTpL2}
\norm{\DT p}_{L^2(\parOmega_t)}^2\le C(\bd)\Paren{1+\norm{\sff}_{L^4(\parOmega_t)}^4+ \norm{p}_{H^1(\parOmega_t)}^2}\le C(\bd)\Paren{1+\norm{p}_{H^1(\parOmega_t)}^2}.
\end{equation}
where we have utilized \eqref{eq_sffL4} and the a priori assumptions 
\eqref{eq_aprioriassumption}.

Therefore, we conclude that
\begin{equation*}
\module{F_1(t)}\le \vare C(\bd)\Paren{1+\norm{p}_{H^1(\parOmega_t)}^2}+C_\vare(\bd),
\end{equation*}

\vspace{2mm}
\noindent\textbf{Estimate  of $F_{2}(t)$.}

To estimate the second term, we extend the normal vector field to 
$\Omega_t$ via harmonic extension, which we still denote by $\normal$. 
Then we have
\begin{equation}\label{eq_normalOmega}
\norm{\normal}_{H^2(\Omega_t)}\le C(\bd),
\end{equation} 
which follows from the bound $\norm{\eta(\cdot,t)}_{H^{\frac52}(\Gamma)}\le \bd$ 
given in \eqref{eq_aprioriassumption}. Then, applying the divergence theorem 
and \eqref{eq_MHD1}, we find that
\begin{align}
\int_{\parOmega_t}p\Paren{\nabla\DT u\normal\cdot\normal}dS&=\int_{\Omega_t}\divergence\Paren{p\nabla\DT u\normal}dx\nonumber\\
&=\int_{\Omega_t}\divergence [p\Paren{\nabla(-\nabla p+h\cdot \nabla h)\normal}]dx\nonumber\\
&=\int_{\parOmega_t}p\Bracket{\nabla\Paren{-\nabla p+h\cdot \nabla h}\normal\cdot\normal}dS.\label{eq_plugMHD}
\end{align}
Combining this with \eqref{eq_aprioriassumption}, the commutator formula \eqref{eq_Dt,nabla}, and \eqref{eq_DTn}, we deduce that
\begin{align*}
F_2(t)&\le \int_{\parOmega_t}p\Paren{\nabla\DT u\normal\cdot\normal}dS+\vare\norm{p}_{L^2(\parOmega_t)}^2+C_\vare(\bd)\\ 
&\le\underbrace{\int_{\parOmega_t}p\Paren{\partial^2_{il}h^jh^i\normal^l\normal_j}dS}_{\eqqcolon F_{21}(t)}\underbrace{-\int_{\parOmega_t}p\Paren{\nabla^2p\normal\cdot\normal}dS}_{\eqqcolon F_{22}(t)}+\vare\norm{p}_{L^2(\parOmega_t)}^2+C_\vare(\bd).
\end{align*}   

\vspace{2mm}
\noindent\textbf{Estimate  of $F_{21}(t)$.}

By \eqref{eq_MHD4}, we have
\begin{equation*}
\partial^2_{il}h^j\normal_j=-\partial_ih^j\partial_l\normal_j-\partial_lh^j\partial_i\normal_j-h^j\partial^2_{il}\normal_j.
\end{equation*}
As a consequence, 
\begin{equation*}
F_{21}(t)=\underbrace{-\int_{\parOmega_t}p\Paren{\partial_ih^j\partial_l\normal_jh^i\normal^l}dS}_{\eqqcolon F_{211}(t)}\underbrace{-\int_{\parOmega_t}p\Paren{\partial_lh^j\partial_i\normal_jh^i\normal^l}dS}_{\eqqcolon F_{212}(t)}\underbrace{-\int_{\parOmega_t}p\Paren{h^j\partial^2_{il}\normal_jh^i\normal^l}dS}_{\eqqcolon F_{213}(t)}.
\end{equation*}
By \eqref{eq_aprioriassumption} and \eqref{eq_normalOmega}, the first 
two terms can be bounded by
\begin{equation*}
\module{F_{211}(t)+F_{212}(t)}\le \norm{p}_{L^2(\parOmega_t)}\norm{\nabla h}_{L^\infty(\Omega_t)}\norm{h}_{L^\infty(\Omega_t)}\norm{\normal}_{H^{\frac 32}(\Omega_t)}\le \vare\norm{p}_{L^2(\parOmega_t)}^2+C_\vare(\bd).
\end{equation*}

Expanding the last term and recalling the harmonic extension of the normal vector, 
we find that
\begin{align*}
F_{213}(t)
={}&\int_{\Omega_t}\partial_l\Paren{ph^j\partial^2_{il}\normal_jh^i}dx\\
={}&-\int_{\Omega_t}\partial_l ph^j\partial^2_{il}\normal_jh^idx-\int_{\Omega_t}  p\partial_l h^j\partial^2_{il}\normal_jh^idx-\int_{\Omega_t}  p h^j\partial^2_{il}\normal_j\partial_l h^idx\\
={}&\underbrace{-\int_{\parOmega_t}\partial_\normal ph^j\partial_{i}\normal_jh^idS}_{\eqqcolon F_{213}^a(t)}+\underbrace{\int_{\Omega_t}\lpe ph^j\partial_{i}\normal_jh^idx}_{\eqqcolon F_{213}^b(t)}\\
&+\underbrace{\int_{\Omega_t}\partial_l p\partial_lh^j\partial_{i}\normal_jh^idx+\int_{\Omega_t}\partial_l ph^j\partial_{i}\normal_j\partial_lh^idx}_{\eqqcolon F_{213}^c(t)}\\
&\underbrace{-\int_{\Omega_t}p\partial_lh^j\partial^2_{il}\normal_jh^idx-\int_{\Omega_t}ph^j\partial^2_{il}\normal_j\partial_l h^idx}_{\eqqcolon F_{213}^d(t)}.
\end{align*} 
It follows that
\begin{align*}
\module{F_{213}^a(t)}\le \vare \norm{\partial_\normal p}_{L^2(\parOmega_t)}^2+C_\vare(\bd).
\end{align*}
From the div-curl type estimate (cf. \cite[Lemma 3.3]{Julin2024}) 
and \eqref{eq_Hfrac123}, we obtain
\begin{align*}
\norm{\partial_\normal p}_{L^2(\parOmega_t)}^2\le& C\Paren{\norm{\Bdnabla p}_{L^2(\parOmega_t)}^2+ \norm{\nabla p}_{L^2(\Omega_t)}^2+ \norm{\lpe p}_{L^2(\Omega_t)}^2}\\
\le &C\Paren{\norm{\Bdnabla p}_{L^2(\parOmega_t)}^2+\norm{p}_{H^{\frac 12}(\parOmega_t)}^2+\norm{\lpe p}_{L^2(\Omega_t)}^2}\\
\le & C\Paren{\norm{p}_{H^1(\parOmega_t)}^2+\norm{\lpe p}_{L^2(\Omega_t)}^2}.
\end{align*}
Note that by taking the divergence of \eqref{eq_MHD1}, we have
\begin{equation}\label{eq_lpep}
-\lpe p=\partial_iu^j\partial_ju^i-\partial_ih^j\partial_jh^i.
\end{equation}
Applying the a priori assumptions \eqref{eq_aprioriassumption}, we 
deduce that
\begin{equation}\label{eq_laplacep}
\module{\lpe p}\le C(\bd),
\end{equation} 
and therefore,
\begin{equation}\label{eq_partialnup}
\norm{\partial_\normal p}_{L^2(\parOmega_t)}^2\le C(\bd)\Paren{1+\norm{p}_{H^1(\parOmega_t)}^2}.
\end{equation}
Thus,
\begin{equation*}
\module{F_{213}^a(t)}\le \vare \norm{p}_{H^1(\parOmega_t)}^2+C_\vare(\bd).
\end{equation*}
Similarly, invoking \eqref{eq_normalOmega} and \eqref{eq_laplacep}, 
we can also deduce that
\begin{align*}
\module{F_{213}^b(t)}\le{}& C(\bd)\norm{\lpe p}_{L^2(\Omega_t)}^2 \le C(\bd),\\
\module{F_{213}^c(t)}
\le{}&\vare\Paren{\norm{\lpe p}_{L^2(\Omega_t)}^2+\norm{p}_{H^{\frac 12}(\parOmega_t)}^2}+C_\vare(\bd)\le\vare\norm{p}_{H^{1}(\parOmega_t)}^2+C_\vare(\bd),\\
\module{F_{213}^d(t)}\le{}&\vare\norm{p}_{L^2(\parOmega_t)}^2+C_\vare(\bd),
\end{align*}
Combining these estimates yields
\begin{equation*}
\module{F_{21}(t)}\le 4\vare\norm{p}_{H^{1}(\parOmega_t)}^2+C_\vare(\bd).
\end{equation*} 

\vspace{2mm}
\noindent\textbf{Estimate  of $F_{22}(t)$.}

Utilizing the identity for the Laplace-Beltrami operator 
\begin{equation*} 
\LB(\,\cdot\,)=\lpe(\,\cdot\,)-\Bracket{\nabla^2(\,\cdot\,)\normal\cdot\normal}-\mc\partial_\normal(\,\cdot\,),
\end{equation*}
along with the boundary condition \eqref{eq_MHD4}, we have
\begin{equation*}
F_{22}(t)=-\int_{\parOmega_t}p\lpe pdS+\int_{\parOmega_t}p\LB pdS+\int_{\parOmega_t}\mc^2 \partial_\normal pdS.
\end{equation*}
Then, by \eqref{eq_laplacep}, Cauchy's inequality, and the divergence 
theorem \eqref{eq_divbd}, we obtain
\begin{equation*}
F_{22}(t)\le\vare\norm{p}_{L^2(\parOmega_t)}^2+C_\vare(\bd)-\int_{\parOmega_t}|\Bdnabla p|^2dS+\int_{\parOmega_t}\vare\partial_\normal p^2+C_\vare(\bd)\mc^4dS.
\end{equation*}
From \eqref{eq_sffL4} and \eqref{eq_partialnup}, it follows that
\begin{equation*}
F_{22}(t)
\le -\norm{\Bdnabla p}_{L^2(\parOmega_t)}^2+2\vare\norm{p}_{H^1(\parOmega_t)}^2+C_\vare(\bd).
\end{equation*} 

Therefore, for $\vare>0$ chosen sufficiently small, we deduce that the 
time derivative of the first part in $F(t)$ satisfies
\begin{align*}
\frac{d}{dt}\int_{\parOmega_t}p\Paren{\nabla u\normal\cdot\normal}dS&\le-\norm{\Bdnabla p}_{L^2(\parOmega_t)}^2+ \vare C(\bd)\Paren{1+\norm{p}_{H^1(\parOmega_t)}^2}+C_\vare(\bd)\\
&\le-\frac 34\norm{\Bdnabla p}_{L^2(\parOmega_t)}^2+C_\vare(\bd),
\end{align*}
since $\norm{p}_{L^2(\parOmega_t)}\le C(\bd)$ by the a priori assumptions 
\eqref{eq_aprioriassumption}.
Similarly, for the second part of $F(t)$, we have
\begin{equation*}
\frac{d}{dt}\int_{\parOmega_t}p^2dS\le C(\bd)\Paren{1+\norm{\Bdnabla p}_{L^2(\parOmega_t)}^2}.
\end{equation*}
Combining the above calculations, we arrive at
\begin{equation*}
\frac{d}{dt}F(t)\le -\frac{1}{2}\norm{\Bdnabla p}_{L^2(\parOmega_t)}^2+C_\vare(\bd)
\end{equation*}
for $\vare>0$ sufficiently small.  
Integrating this inequality over $[0,t]$ for $0<t\le T$, and recalling 
\eqref{eq_F(t)}, we obtain
\begin{equation*}
\frac{1}{2}\int_{0}^{t}\norm{\Bdnabla p}_{L^2(\parOmega_s)}^2ds
\le F(0)-F(t)+C_\vare(\bd)t\le C_\vare(\bd)(1+t).
\end{equation*}
We conclude that 
\begin{equation*}
\int_{0}^{t}\norm{p}_{H^1(\parOmega_s)}^2ds\le C_\vare(\bd)(1+t),
\end{equation*}
for $\vare>0$ sufficiently small. Recalling that $0<t\le T$, the proof 
is complete.
\end{proof}

\begin{proposition}\label{lem_int0TnablapH1} 
Assume that the a priori assumptions \eqref{eq_aprioriassumption} 
hold for some $T>0$. Then the following estimate holds:
\begin{equation*}
\sup_{t\in[0,T)}\norm{\nabla p}_{L^2(\Omega_t)}^2+\int_{0}^{T}\norm{\nabla p}_{H^1(\Omega_t)}^2dt\le C(\bd)(1+T). 
\end{equation*} 
\end{proposition}
\begin{proof}
By virtue of \eqref{eq_Hfrac123}, \eqref{eq_sffL4}, and \eqref{eq_laplacep}, 
it is clear that
\begin{equation}\label{eq_nablapL2}
\norm{\nabla p}_{L^2(\Omega_t)}^2\le C\Paren{\norm{\lpe p}_{L^2(\Omega_t)}^2+\norm{\sff}_{H^{\frac 12}(\parOmega_t)}^2}\le C(\bd).
\end{equation}
We define 
\begin{equation*}
J(t):=-\int_{\parOmega_t}p\LB u_ndS.
\end{equation*}
From the a priori assumptions \eqref{eq_aprioriassumption}, we have 
\begin{equation}\label{eq_supJ}
\sup_{t\in[0,T)}|J(t)|\le C \sup_{t\in[0,T)} \norm{p}_{L^2(\parOmega_t)} \norm{u_n}_{H^2(\parOmega_t)}\le C(\bd).
\end{equation} 
Differentiating $J(t)$ with respect to time using \eqref{eq_RT2}, 
we obtain
\begin{equation}\label{eq_ddtJ}
\frac{d}{dt}J(t)=-\int_{\parOmega_t}p\DT \LB u_ndS -\int_{\parOmega_t}p\LB u_n\Bddiv udS-\int_{\parOmega_t}\DT p\LB u_ndS. 
\end{equation}
Applying the a priori assumptions \eqref{eq_aprioriassumption} and 
\eqref{eq_DTpL2} once more, we can bound the last two terms by
\begin{equation*}
\module{\text{last two terms in \eqref{eq_ddtJ}}}\le C(\bd)+\norm{\DT p}_{L^2(\parOmega_t)}^2\le C(\bd)\Paren{1+\norm{p}_{H^1(\parOmega_t)}^2}.
\end{equation*}
For the first term, we apply the commutator formula
\begin{equation*}
[\DT,\LB](\,\cdot\,)=\Bdnabla^2(\,\cdot\,)\star\nabla u-\Bdnabla (\,\cdot\,)\cdot \LB u+\sff\star \nabla u\star \Bdnabla(\,\cdot\,)
\end{equation*}
to deduce that
\begin{align*}
-\int_{\parOmega_t}p\DT \LB u_ndS
={}&\underbrace{-\int_{\parOmega_t}p\LB \DT u_ndS}_{\eqqcolon J_1(t)}+\underbrace{\int_{\parOmega_t}p\Bdnabla u_n\cdot\LB udS}_{\eqqcolon J_2(t)}\\
&\underbrace{-\int_{\parOmega_t}p\Bdnabla^2u_n\star\nabla udS}_{\eqqcolon J_3(t)}\underbrace{-\int_{\parOmega_t}p\sff\star\nabla u\star\Bdnabla u_ndS}_{\eqqcolon J_4(t)}.
\end{align*} 
By the a priori assumptions \eqref{eq_aprioriassumption} and \eqref{eq_sffL4}, 
the last three terms can be readily estimated as
\begin{equation*}
\module{J_2(t)+J_3(t)+J_4(t)}\le C(\bd),
\end{equation*}
and thus it remains to focus on the first term.

From the divergence theorem \eqref{eq_divbd}, it follows that
\begin{equation*}
J_1(t)=\int_{\parOmega_t}\Bdnabla p\cdot\Bdnabla\Paren{\DT u\cdot\normal}dS+\int_{\parOmega_t}\Bdnabla p\cdot\Bdnabla\Paren{\DT\normal\cdot u}dS.
\end{equation*}
We can rewrite the first integral as in \eqref{eq_plugMHD}, and apply 
the identity
\begin{equation*}
\DT\normal=-\Paren{\Bdnabla u}^\top\normal=-\Bdnabla u_n+\sff\Paren{u-u_n\normal}
\end{equation*}
from \eqref{eq_DTn} to the second integral to obtain
\begin{align*}
J_1(t)={}&\underbrace{-\int_{\parOmega_t}\Bdnabla p\cdot\Bdnabla\partial_\normal pdS}_{\eqqcolon J_{1}^a(t)}+\underbrace{\int_{\parOmega_t}\Bdnabla p\cdot\Bdnabla\Bracket{\Paren{h\cdot\nabla h}\cdot\normal}dS}_{\eqqcolon J_{1}^b(t)}\\
&+\underbrace{\int_{\parOmega_t}\Bdnabla p\cdot\Bdnabla\Brace{\Bracket{-\Bdnabla u_n+\sff\Paren{u-u_n\normal}}\cdot u}dS}_{\eqqcolon J_{1}^c(t)}.
\end{align*} 
Note that $J_{1}^a(t)$ can be estimated by applying Reilly's type 
identity (cf. \cite[(3.2)]{Julin2024}) and \eqref{eq_laplacep}:
\begin{align*}
J_{1}^a(t)&=\frac12\Paren{\norm{\lpe p}_{L^2(\Omega_t)}^2-\norm{\nabla^2p}_{L^2(\Omega_t)}^2-\int_{\parOmega_t}\sff\Bdnabla p\cdot\Bdnabla pdS-\int_{\parOmega_t}\mc\module{\partial_\normal p}^2dS}\\
&\le-\frac 12\norm{\nabla^2p}_{L^2(\Omega_t)}^2+C(\bd)+C\int_{\parOmega_t}\module{\sff}\module{\nabla p}^2dS.
\end{align*}
Applying \eqref{eq_sffL4},  Sobolev embedding, the trace theorem, 
and interpolation, we can bound the last integral by
\begin{align*}
\int_{\parOmega_t}\module{\sff}\module{\nabla p}^2dS&\le \norm{\sff}_{L^4(\parOmega_t)}\norm{\module{\nabla p}^2}_{L^{\frac43}(\parOmega_t)}\\
&\le C(\bd)\norm{\nabla p}_{L^{\frac83}(\parOmega_t)}^2\\
&\le C(\bd)\norm{\nabla p}_{L^2(\Omega_t)}^2+\frac 14\norm{\nabla^2 p}_{L^2(\Omega_t)}^2.
\end{align*}
As a result, invoking \eqref{eq_nablapL2}, we obtain
\begin{equation*}
J_{1}^a(t)\le -\frac 14\norm{\nabla^2p}_{L^2(\Omega_t)}^2+C(\bd). 
\end{equation*}

For the second term $J_{1}^b(t)$, which involves the magnetic field, 
we have
\begin{align*}
\Bdnabla p\cdot\Bdnabla\Paren{h^i\partial_ih^j\normal_j}&=-\Bdnabla p\cdot\Bdnabla\Paren{h^ih^j\partial_i\normal_j}\\
&=-\Bdnabla p\cdot\Bdnabla\Bracket{h^jh^i\Paren{\partial_i\normal_j-\partial_\normal \normal_j\normal_i}}\\
&=-\Bdnabla p\cdot\Bdnabla\Paren{h^jh^i\Bdnabla_i\normal_j}\\
&=-\Bdnabla p\cdot\Bdnabla\Bracket{\Paren{h\cdot\sff}\cdot h}
\end{align*} 
since $h\cdot\normal=0$ on $\parOmega_t$ from \eqref{eq_MHD4}. By the a priori assumptions \eqref{eq_aprioriassumption}, it suffices to bound 
the term involving the tangential derivative of the second fundamental form, namely,
\begin{equation*}
\module{\int_{\parOmega_t}\Bdnabla p\star\Bdnabla \sff\star h\star h dS}\le C(\bd)\norm{p}_{H^1(\parOmega_t)}^2,
\end{equation*}
where we have applied Lemma \ref{lem_sffcontrolledbymc}. 
The remaining terms can be estimated analogously, and we arrive at
\begin{equation*}
\module{J_{1}^b(t)}\le  C(\bd)\Paren{1+\norm{p}_{H^1(\parOmega_t)}^2}.
\end{equation*} 
Similarly, we can also deduce that
\begin{equation*}
\module{J_{1}^c(t)}\le C(\bd)\Paren{1+\norm{p}_{H^1(\parOmega_t)}^2}.
\end{equation*}
Consequently, we obtain
\begin{equation*}
\frac{d}{dt}J(t)\le-\frac 14\norm{\nabla^2p}_{L^2(\Omega_t)}^2+C(\bd)\Paren{\norm{p}_{H^1(\parOmega_t)}^2 +1}. 
\end{equation*}
Integrating this inequality over $[0,T]$ and invoking Lemma \ref{lem_pH1} 
along with \eqref{eq_supJ}, we get
\begin{equation*}
\int_{0}^{T}\norm{\nabla^2p}_{L^2(\Omega_t)}^2dt\le C(\bd)(1+T).
\end{equation*}
The claims then follow from the previous estimate \eqref{eq_nablapL2}. 
\end{proof}
\subsection{Reverse energy inequalities}\label{sec_closetheestimate}

In this section, we shall close the energy estimates. Recall that we 
have established the following inequalities in Propositions \ref{lem_ddt} and \ref{lem_int0TnablapH1}:
\begin{equation*}
\module{\frac{d}{dt}\energy(t)} \le C(\bd)\Paren{1+\norm{\nabla  p}_{H^1(\Omega_t)}^2}\Energy(t)\ \text{and}\ \int_{0}^{T}\norm{\nabla p}_{H^1(\Omega_t)}^2dt\le C(\bd)(1+T).
\end{equation*}  
It suffices to bound the initial quantities and establish the reverse inequality
\begin{equation*}
\Energy(t)\le C(\bd)\Paren{1+\energy(t)}
\end{equation*}
under suitable conditions.

We begin by showing that the initial quantities $\Energy(0), 
\norm{p}_{H^{\frac{5}{2}}(\Omega_0)}^2$, and $\norm{\DT p}_{H^{1}(\Omega_0)}^2$ 
can be bounded in terms of the initial velocity, the initial magnetic 
field, and the mean curvature.
\begin{proposition}\label{pro_E(0)}
We have the following estimates:
\begin{equation*}
\Energy(0)+\norm{p}_{H^{\frac{5}{2}}(\Omega_0)}^2+\norm{\DT p}_{H^{1}(\Omega_0)}^2\le C_{\operatorname{initial}},
\end{equation*}
where $C_{\operatorname{initial}}$ is a positive constant depending only on 
$\norm{u_0}_{H^3(\Omega_0)}, \norm{h_0}_{H^3(\Omega_0)}$, and $\norm{\mc}_{H^2(\parOmega_0)}$.
\end{proposition} 
\begin{proof}  
We first bound $\norm{\DT^{2-k}h}_{H^{\frac{3}{2}k}(\Omega_0)}^2$ for $k=0,1$ 
in terms of lower-order velocity terms using the identity
\begin{equation}\label{eq_Dt_jH}
\DT^jh=\sum_{1\le m\le j}\sum_{|\beta|\le j-m}\nabla\DT^{\beta_1}u\star\dots\star \nabla\DT^{\beta_{m}}u\star h,\quad j\ge 1,
\end{equation}  
which can be verified by induction. For $k=0$, we apply \eqref{eq_Dt_jH} 
to obtain
\begin{align*} 
\norm{\DT^2 h}_{L^2(\Omega_0)}^2
\le{}& C\Paren{\norm{\sum_{i\le 1}\nabla\DT^{i}u\star h}_{L^2(\Omega_0)}^2+\norm{\nabla u\star\nabla u\star h}_{L^2(\Omega_0)}^2}\\
\le{}& C\norm{h}_{L^\infty(\Omega_0)}^2\Paren{\sum_{i\le 1}\norm{\nabla\DT^iu}_{L^2(\Omega_0)}^2+\norm{\nabla u\star\nabla u}_{L^2(\Omega_0)}^2}\\
\le{}& C_{\operatorname{initial}}\Paren{1+\norm{\DT u}_{H^{1}(\Omega_0)}^2}.
\end{align*}
Furthermore, applying the Kato-Ponce inequality and \eqref{eq_MHD2}, 
we find that
\begin{equation}\label{eq_DTH}
\norm{\DT h}_{H^{\frac32}(\Omega_0)}^2 \le C\Paren{\norm{h}_{L^{10}(\Omega_0)}\norm{\nabla u}_{W^{\frac32,\frac52}(\Omega_0)}^2+\norm{\nabla u}_{L^3(\Omega_0)}\norm{h}_{W^{\frac32,6}(\Omega_0)}^2}\le C_{\operatorname{initial}}.
\end{equation}  

Next, we bound $\norm{\DT^{2-k}u}_{H^{\frac{3}{2}k}(\Omega_0)}^2$ for $k=0,1$ 
in terms of the pressure. Note that by the Kato-Ponce inequality and 
\eqref{eq_MHD1}, it holds
\begin{align}
\norm{\DT u}_{H^{\frac{3}{2}}(\Omega_0)}^2\le {}&C\Paren{\norm{p}_{H^{\frac{5}{2}}(\Omega_0)}^2+\norm{h}_{L^{\infty}(\Omega_0)}^2\norm{h}_{H^{\frac{5}{2}}(\Omega_0)}^2+\norm{\nabla h}_{L^{\infty}(\Omega_0)}^2\norm{h}_{H^{\frac{3}{2}}(\Omega_0)}^2}\nonumber\\
\le{}&C_{\operatorname{initial}}\Paren{\norm{p}_{H^{\frac{5}{2}}(\Omega_0)}^2+1},\label{eq_DTu1.5}
\end{align}
and invoking the commutator formula \eqref{eq_Dt,nabla}, we have
\begin{align*}
\norm{\DT^2u}_{L^{2}(\Omega_0)}^2\le{}& \norm{\DT\nabla p}_{L^{2}(\Omega_0)}^2+\norm{\DT\Paren{h\cdot\nabla h}}_{L^{2}(\Omega_0)}^2\\
\le{}&\norm{\nabla\DT p}_{L^{2}(\Omega_0)}^2+\norm{[\DT,\nabla]p}_{L^{2}(\Omega_0)}^2+\norm{\sum_{|\alpha|\le 1}a_{\alpha}\Paren{\nabla u}\nabla^{1+\alpha_1}u\star \nabla^{\alpha_{2}}h\star h}_{L^{2}(\Omega_0)}^2\\
\le{}& C_{\operatorname{initial}}\Paren{1+\norm{\nabla\DT p}_{L^{2}(\Omega_0)}^2+\norm{\nabla p}_{L^{2}(\Omega_0)}^2}.
\end{align*}

Then, we estimate $\norm{p}_{H^{\frac{5}{2}}(\Omega_0)}^2$ and 
$\norm{\DT p}_{H^1(\Omega_0)}^2$. We consider the following elliptic 
boundary value problem:
\begin{equation*}
\begin{cases}
-\lpe p=\partial_i u^j\partial_j u^i-\partial_i h^j\partial_j h^i,\quad & \text{in}\ \Omega_0,\\
p=\mc,\quad & \text{on}\ \parOmega_0.		
\end{cases}
\end{equation*}
Standard elliptic estimates then yield
\begin{equation}\label{eq_p2.5}
\norm{p}_{H^{\frac{5}{2}}(\Omega_0)}\le C\Paren{\norm{\partial_i u^j\partial_j u^i-\partial_i h^j\partial_j h^i}_{H^{\frac{1}{2}}(\Omega_0)}+\norm{\mc}_{H^2(\parOmega_0)}}\le C_{\operatorname{initial}}.
\end{equation}
Moreover, applying \eqref{eq_Hfrac123}, we obtain
\begin{equation*}
\norm{\DT p}_{H^1(\Omega_0)}\le C\Paren{\norm{\lpe \DT p}_{L^2(\Omega_0)}+\norm{\DT p}_{H^{\frac 12}(\parOmega_0)}}.
\end{equation*} 
The remaining terms on the right-hand side can be straightforwardly bounded 
using \cite[Lemma 2.11]{Hao2025} and \eqref{eq_Dtp}, since 
$\norm{p}_{H^{\frac{5}{2}}(\Omega_0)}$ is already controlled. 
Therefore, we have
\begin{equation*}
\norm{\DT p}_{H^1(\Omega_0)}\le  C_{\operatorname{initial}}.
\end{equation*}

Finally, the term $\norm{\Bdnabla\Paren{\DT u\cdot\normal}}_{L^2(\parOmega_0)}^2$ 
can be estimated via the trace theorem, exploiting the boundary regularity. 
Indeed, leveraging the bound on the mean curvature, we apply Lemma 
\ref{lem_sffcontrolledbymc} to deduce 
\begin{equation*}
\norm{\sff}_{H^2(\parOmega_0)}\le C_{\operatorname{initial}},
\end{equation*}
which in turn implies
\begin{equation*}
\norm{\Bdnabla\Paren{\DT u\cdot\normal}}_{L^2(\parOmega_0)}^2\le C\Paren{\norm{\Bdnabla\DT u\star\normal}_{L^2(\parOmega_0)}^2+\norm{\DT u\star\sff}_{L^2(\parOmega_0)}^2}\le C_{\operatorname{initial}},
\end{equation*}  
by virtue of \eqref{eq_DTu1.5} and \eqref{eq_p2.5}. 
This concludes the proof of the proposition. 
\end{proof}

We proceed to bound the full energy functional $\Energy(t)$ in terms of $\energy(t)$ 
under a slightly different hypothesis compared to the a priori assumptions 
\eqref{eq_aprioriassumption}. For this, we introduce the following modified energy functional:
\begin{align*}
\widetilde{\energy}(t)\coloneqq {}&\frac{1}{2}\bigg(\norm{\DT^{2}u}_{L^2(\Omega_t)}^2+\norm{\DT^{2}h}_{L^2(\Omega_t)}^2+\norm{\Bdnabla\Paren{\DT u\cdot\normal}}_{L^2(\parOmega_t)}^2\\
&\qquad+\norm{\vorticity u}_{H^{2}(\Omega_t)}^2+\norm{\vorticity h}_{H^{2}(\Omega_t)}^2\bigg)+1.
\end{align*}
Note that the a priori assumptions \eqref{eq_aprioriassumption} imply
\begin{equation*}
\norm{\vorticity u}_{L^2(\Omega_t)}^2+\norm{\vorticity h}_{L^2(\Omega_t)}^2\le C(\bd).
\end{equation*}
By interpolation, we deduce that
\begin{equation}\label{eq_tildeee}
\widetilde{\energy}(t) \leq C(\bd)\Paren{\energy(t)+1}.
\end{equation}
We have the following proposition.
\begin{proposition}\label{pro_Elee}
Assume that the free boundary satisfies $\parOmega_t\in C^{1,\alpha}$, and 
that the pressure, velocity, and magnetic field satisfy the bound
\begin{equation*}
\norm{p}_{L^{2}(\parOmega_t)}+\norm{\nabla p}_{L^{2}(\Omega_t)}+\norm{\nabla u}_{L^4(\Omega_t)}+\norm{\nabla h}_{L^4(\Omega_t)}\le C_\ddagger.
\end{equation*}
Then we have
\begin{equation}\label{eq_EBPlee}
\Energy(t)+\norm{\sff}_{H^{\frac 32}(\parOmega_t)}^2 +\norm{\DT p}_{H^{\frac 12}(\parOmega_t)}^2\le C(C_\ddagger)\Paren{1+\energy(t)},
\end{equation} 
where $C(C_\ddagger)$ is a positive constant depending only on $C_\ddagger$.
\end{proposition} 
\begin{proof}
By the conservation of physical energy,
\begin{equation*}
\norm{u}_{L^2(\Omega_t)}^2+\norm{h}_{L^2(\Omega_t)}^2+\int_{\parOmega_t}1dS\equiv \norm{u}_{L^2(\Omega_0)}^2+\norm{h}_{L^2(\Omega_0)}^2+\int_{\parOmega_0}1dS,
\end{equation*}
and by the Gagliardo-Nirenberg inequality, we have
\begin{equation*}
\norm{u}_{W^{1,4}(\Omega_t)}+\norm{h}_{W^{1,4}(\Omega_t)}\le C(C_\ddagger).
\end{equation*}
We shall show that 
\begin{equation*}
\Energy(t)\le C\widetilde{\energy}(t),
\end{equation*} 
for which we need to estimate
\begin{equation*}
\norm{\DT u}_{H^{\frac{3}{2}}(\Omega_t)}^2,\norm{\DT h}_{H^{\frac{3}{2}}(\Omega_t)}^2,\norm{u}_{H^3(\Omega_t)}^2,\ \text{and}\ \norm{h}_{H^3(\Omega_t)}^2.
\end{equation*} 
Recalling that we have already established 
the estimate for $\norm{\DT h}_{H^{\frac{3}{2}}(\Omega_t)}^2$ in \eqref{eq_DTH} 
(where we utilize $\norm{\nabla u}_{L^3(\Omega_t)}+\norm{h}_{L^{10}(\Omega_t)}\le C(C_\ddagger)$ 
via Sobolev embedding), we readily obtain
\begin{equation}\label{eq_DTH32}
\norm{\DT h}_{H^{\frac{3}{2}}(\Omega_t)}^2\le \vare\Energy(t)+C_\vare(C_\ddagger).
\end{equation} 
Thus, it suffices to bound $\norm{\DT u}_{H^{\frac{3}{2}}(\Omega_t)}^2, 
\norm{u}_{H^3(\Omega_t)}^2,$ and $\norm{h}_{H^3(\Omega_t)}^2$.

\vspace{2mm}
\noindent\textbf{Estimate of $\norm{\DT u}_{H^{\frac{3}{2}}(\Omega_t)}^2$.}

Since 
$\norm{p}_{L^{2}(\parOmega_t)}+\norm{\nabla p}_{L^{2}(\Omega_t)}\le C_\ddagger$, 
we have $\norm{\sff}_{H^{\frac12}(\parOmega_t)}\le C(C_\ddagger)$. 
Therefore, we can extend the outward unit normal $\normal$ to $\Omega_t$ 
such that $\norm{\normal}_{H^{2}(\Omega_t)}\le C(C_\ddagger)$ (which we 
still denote by $\normal$ as in \eqref{eq_normalOmega}). 
By the definition of $\widetilde{\energy}(t)$ and the divergence theorem, we have
\begin{align*}
\norm{\DT u\cdot \normal}_{L^2(\parOmega_t)}^2={}&\int_{\parOmega_t} \Paren{\DT u\cdot \normal}\DT u\cdot \normal dS\\
\le{}&\module{\int_{\Omega_t} \Paren{\DT u\cdot \normal}\divergence\DT udx}+\module{\int_{\Omega_t} \nabla\DT u\star\DT u\star\normal dx}+\module{\int_{\Omega_t} \DT u\star\nabla\normal\star\DT udx}\\
\le{}& C(C_\ddagger)\Paren{\norm{\DT u}_{L^{2}(\Omega_t)}^2+\norm{\DT u}_{H^{1}(\Omega_t)}^2+\norm{\DT u}_{L^{3}(\Omega_t)}^2\norm{\nabla\normal}_{L^{3}(\Omega_t)}}\\
\le{}& \vare\norm{\DT u}_{H^{\frac32}(\Omega_t)}^2+C_\vare(C_\ddagger)\norm{\DT u}_{L^{2}(\Omega_t)}^2,
\end{align*}
where we have applied interpolation in the last step. Combining this 
with the regularity bound $\norm{\sff}_{H^{\frac12}(\parOmega_t)}\le C(C_\ddagger)$, 
we invoke the div-curl estimate (cf. \cite[Theorem 3.1]{Julin2024}) to obtain
\begin{equation*}
\norm{\DT u}_{H^{\frac{3}{2}}(\Omega_t)}^2\le  C(C_\ddagger)\Paren{\norm{\DT u\cdot\normal}_{H^{1}(\parOmega_t)}^2+\norm{\DT u}_{L^{2}(\Omega_t)}^2+\norm{\divergence\DT u}_{H^{\frac12}(\Omega_t)}^2+\norm{\vorticity\DT u}_{H^{\frac 12}(\Omega_t)}^2}, 
\end{equation*} 
Applying \eqref{eq_lpep} along with the bound $\norm{\DT u}_{L^{2}(\Omega_t)}^2\le C(C_\ddagger)$ 
from \eqref{eq_MHD1}, we find that
\begin{equation*}
\norm{\DT u}_{H^{\frac{3}{2}}(\Omega_t)}^2\le C(C_\ddagger)\Paren{\widetilde{\energy}(t)+\underbrace{\norm{\nabla u\star\nabla u}_{H^{\frac12}(\Omega_t)}^2+\norm{\nabla h\star\nabla h}_{H^{\frac 12}(\Omega_t)}^2}_{\eqqcolon L_1(t)}+\underbrace{\norm{\Paren{h\cdot\nabla} \Paren{\vorticity h}}_{H^{\frac 12}(\Omega_t)}^2}_{\eqqcolon L_2(t)}}.
\end{equation*}
By the Kato-Ponce inequality and interpolation, we can estimate $L_1(t)$ as
\begin{align*}
L_1(t)&\le C(C_\ddagger)\Paren{\norm{\nabla u}_{L^3(\Omega_t)}^2\norm{\nabla u}_{W^{\frac 12,6}(\Omega_t)}^2+\norm{\nabla h}_{L^3(\Omega_t)}^2\norm{\nabla h}_{W^{\frac 12,6}(\Omega_t)}^2}\\ &\le \vare\Paren{\norm{u}_{H^3(\Omega_t)}^2+\norm{h}_{H^3(\Omega_t)}^2}+C_\vare(C_\ddagger)\\
&\le \vare \Energy(t)+C_\vare(C_\ddagger).
\end{align*} 

As for $L_2(t)$, we deduce that
\begin{align*}
L_2(t)\le{}& C(C_\ddagger)\Paren{\norm{h}_{L^6(\Omega_t)}^2\norm{\nabla\Paren{\vorticity h}}_{W^{\frac 12,3}(\Omega_t)}^2 +\norm{\nabla\Paren{\vorticity h}}_{L^{6}(\Omega_t)}^2 \norm{h}_{W^{\frac 12,3}(\Omega_t)}^2} \\
\le{}& C(C_\ddagger)\norm{\vorticity h}_{H^2(\Omega_t)}^2\\
\le{}& C(C_\ddagger)\widetilde{\energy}(t).
\end{align*}
Combining the above estimates yields
\begin{equation}\label{eq_DTv32}
\norm{\DT u}_{H^{\frac{3}{2}}(\Omega_t)}^2
\le \vare \Energy(t)+ C_\vare(C_\ddagger)\widetilde{\energy}(t). 
\end{equation}

\vspace{2mm}
\noindent\textbf{Bounding $\norm{u}_{H^3(\Omega_t)}^2$ and $\norm{h}_{H^3(\Omega_t)}^2$ by $\norm{\LB  u_n}_{H^{\frac 12}(\parOmega_t)}^2$.}

From the div-curl type estimate (cf. \cite[Lemma 3.7]{Julin2024}), we see that
\begin{align*}
\norm{u}_{H^3(\Omega_t)}^2
&\le C\Bracket{\norm{\LB  u_n}_{H^{\frac12}(\parOmega_t)}^2+\Paren{1+\norm{\mc}_{H^{2}(\parOmega_t)}^2}\norm{u}_{L^\infty(\Omega_t)}^2+\norm{\vorticity u}_{H^2(\Omega_t)}^2},\\ 
&\le C(C_\ddagger)\Paren{\norm{\LB  u_n}_{H^{\frac12}(\parOmega_t)}^2+\widetilde{\energy}(t)+\norm{\mc}_{H^{2}(\parOmega_t)}^2}.
\end{align*}
since $\norm{u}_{L^\infty(\Omega_t)}\le C\norm{u}_{W^{1,4}(\Omega_t)}\le C(C_\ddagger)$ 
by Sobolev embedding. Similarly, using \eqref{eq_MHD4}, we have
\begin{align*}
\norm{h}_{H^3(\Omega_t)}^2
&\le C\Bracket{\Paren{1+\norm{\mc}_{H^{2}(\parOmega_t)}^2}\norm{h}_{L^\infty(\Omega_t)}^2+\norm{\vorticity h}_{H^2(\Omega_t)}^2},\\ 
&\le C(C_\ddagger)\Paren{\widetilde{\energy}(t)+\norm{\mc}_{H^{2}(\parOmega_t)}^2}.
\end{align*}

Then, invoking \eqref{eq_DTv32} and interpolation, we can bound the 
mean curvature by
\begin{align*}
\norm{\mc}_{H^{2}(\parOmega_t)}^2&\le C\Paren{\norm{p}_{L^2(\parOmega_t)}^2+\norm{\nabla p}_{H^{\frac 32}(\Omega_t)}^2}\\
&\le C(C_\ddagger)\Paren{\norm{\DT u}_{H^{\frac32}(\Omega_t)}^2+\norm{h\cdot\nabla h}_{H^{\frac 32}(\Omega_t)}^2+1}\\
&\le\vare\Energy(t)+C(C_\ddagger)\Paren{\widetilde{\energy}(t)+\norm{h}_{L^\infty(\Omega_t)}^2\norm{\nabla h}_{H^{\frac 32}(\Omega_t)}^2+\norm{h}_{W^{\frac 32,4}(\Omega_t)}^2\norm{\nabla h}_{L^4(\Omega_t)}^2}\\
&\le\vare\Energy(t)+C(C_\ddagger)\widetilde{\energy}(t)+\vare\norm{h}_{H^3(\Omega_t)}^2\\
&\le2\vare\Energy(t)+C(C_\ddagger)\widetilde{\energy}(t).
\end{align*}
Consequently, we obtain
\begin{equation*}
\norm{u}_{H^3(\Omega_t)}^2+\norm{h}_{H^3(\Omega_t)}^2\le 2\vare\Energy(t)+C(C_\ddagger)\Paren{\norm{\LB  u_n}_{H^{\frac12}(\parOmega_t)}^2+\widetilde{\energy}(t)}.
\end{equation*}

\vspace{2mm}
\noindent\textbf{Estimate of $\norm{\LB  u_n}_{H^{\frac 12}(\parOmega_t)}^2$.}

It remains to bound $\norm{\LB  u_n}_{H^{\frac 12}(\parOmega_t)}^2$. 
To this end, we recall Lemma \ref{lem_sffcontrolledbymc} and \eqref{eq_MHD1}; 
applying the trace theorem, it follows that
\begin{align}
\norm{\sff}_{H^{\frac 32}(\parOmega_t)}^2\le{}& C(C_\ddagger)\Paren{1+\norm{p}_{H^{\frac 32}(\parOmega_t)}^2}\nonumber\\
\le{}& C(C_\ddagger)\Paren{1+\norm{h\cdot\nabla h-\DT u}_{H^{1}(\Omega_t)}^2}\nonumber\\
\le{}& C(C_\ddagger)\Paren{1+\norm{h}_{L^4(\Omega_t)}^2\norm{\nabla h}_{W^{1,4}(\Omega_t)}^2+\norm{h}_{H^1(\Omega_t)}^2\norm{\nabla h}_{L^\infty(\Omega_t)}^2}\nonumber\\
&+\vare\norm{\DT u}_{H^{\frac 32}(\Omega_t)}^2+C_\vare(C_\ddagger)\norm{\DT u}_{L^2(\Omega_t)}^2\nonumber\\
\le{}&\vare \Energy(t)+C(C_\ddagger)\widetilde{\energy}(t),\label{eq_BHfrac32}
\end{align}
where we have utilized Sobolev embedding, interpolation, and \eqref{eq_DTv32}. 
From \eqref{eq_Dtp}, we have
\begin{equation*}
\LB u_n=-\DT p-|\sff|^2u_n+\Bdnabla p\cdot u.
\end{equation*} 
Using Sobolev embedding theorems and interpolation, we can deduce
\begin{align}
\norm{\DT p}_{H^{\frac 12}(\parOmega_t)}^2
\le{}& C(C_\ddagger)\Paren{\norm{\DT p}_{L^{2}(\parOmega_t)}^2+\norm{\nabla\DT p}_{L^{2}(\Omega_t)}^2}\nonumber\\
\le{}& C(C_\ddagger)\bigg(\norm{u}_{H^2(\parOmega_t)}^2+\norm{\sff\star\Bdnabla u}_{L^{2}(\parOmega_t)}^2+\norm{\DT^2u}_{L^{2}(\Omega_t)}^2+\norm{\DT\Paren{h\cdot\nabla h}}_{L^{2}(\Omega_t)}^2\nonumber\\
&\qquad\qquad +\norm{\nabla u\star\Paren{h\cdot\nabla h-\DT u}}_{L^{2}(\Omega_t)}^2\bigg)\nonumber\\
\le{}& C(C_\ddagger)\Bracket{\widetilde{\energy}(t)+\vare\Paren{\norm{\DT u}_{H^{\frac 32}(\Omega_t)}^2+\norm{\DT h}_{H^{\frac 32}(\Omega_t)}^2+\norm{u}_{H^3(\Omega_t)}^2+\norm{h}_{H^3(\Omega_t)}^2}+C_\vare(C_\ddagger)}\nonumber\\ 
\le{}&\vare \Energy(t)+ C(C_\ddagger)\widetilde{\energy}(t),\label{eq_DTp12}
\end{align}
since $\norm{\DT u}_{H^{\frac 32}(\Omega_t)}^2$ and $\norm{\DT h}_{H^{\frac 32}(\Omega_t)}^2$ 
have already been controlled. Similarly, applying \eqref{eq_BHfrac32}, we have
\begin{equation*}
\norm{\Bdnabla p\cdot u}_{H^{\frac 12}(\parOmega_t)}^2\le \vare \Energy(t)+ C(C_\ddagger)\widetilde{\energy}(t),
\end{equation*}
and with the aid of the bilinear inequality, we obtain
\begin{equation*}
\norm{|\sff|^2u_n}_{H^{\frac 12}(\parOmega_t)}^2\le \norm{\normal}_{H^{\frac 32}(\parOmega_t)}^2\norm{\sff}_{H^{\frac 32}(\parOmega_t)}^4\norm{u}_{H^{\frac 12}(\parOmega_t)}^2\le C(C_\ddagger)\norm{u}_{H^1(\Omega_t)}^2\le \vare \Energy(t)+ C(C_\ddagger)\widetilde{\energy}(t).
\end{equation*}
Based on the above calculations, it follows that
\begin{align*}
\norm{\LB u_n}_{H^{\frac 12}(\parOmega_t)}^2\le{}& C(C_\ddagger)\Paren{\norm{\DT p}_{H^{\frac 12}(\parOmega_t)}^2+\norm{\Bdnabla p\cdot u}_{H^{\frac 12}(\parOmega_t)}^2+\norm{|\sff|^2u_n}_{H^{\frac 12}(\parOmega_t)}^2}\\  
\le{}& \vare \Energy(t)+ C_\vare(C_\ddagger)\widetilde{\energy}(t).
\end{align*}
As a result, we arrive at
\begin{equation*}
\norm{u}_{H^3(\Omega_t)}^2+\norm{h}_{H^3(\Omega_t)}^2\le 3\vare\Energy(t)+C(C_\ddagger)\widetilde{\energy}(t).
\end{equation*}

Recalling \eqref{eq_DTH32} and \eqref{eq_DTv32}, we conclude that
\begin{equation*}
\Energy(t)\le C(C_\ddagger)\widetilde{\energy}(t).
\end{equation*}
Finally, combining \eqref{eq_BHfrac32}, \eqref{eq_DTp12}, and \eqref{eq_tildeee}, 
the desired estimate \eqref{eq_EBPlee} follows. This completes the proof.
\end{proof}

\subsection{Proofs of Theorems \ref{thm_main1} and \ref{thm_main2}}\label{sec_profthm12}

\begin{proof}[Proof of Theorem \ref{thm_main1}]
We divide the proof into two steps.

\vspace{2mm} 
\noindent\textbf{Step 1.}
Recall that by applying the a priori assumptions \eqref{eq_aprioriassumption}, 
we have obtained in \eqref{eq_nablapL2} that
\begin{equation*}
\sup_{t\in[0,T)}\norm{\nabla p}_{L^2(\Omega_t)}^2\le C(\bd).
\end{equation*}
Consequently, the hypotheses of Proposition \ref{pro_Elee} are satisfied 
uniformly for $0\le t<T$. Combining Proposition \ref{pro_Elee} with 
Proposition \ref{lem_ddt}, we deduce that
\begin{equation*} 
\frac{d}{dt}\energy(t)\le  C(\bd)\Paren{1+\norm{\nabla  p}_{H^1(\Omega_t)}^2}\Paren{1+\energy(t)},\quad 0<t<T.
\end{equation*} 
Integrating this differential inequality over $(0,T)$ and invoking 
Propositions \ref{lem_int0TnablapH1} and \ref{pro_Elee}, we have
\begin{equation*}
\sup_{t\in[0,T)}\energy(t)\le e^{C(\bd)(1+T)}(1+\energy(0))\ \text{and}\ \sup_{t\in[0,T)}\Energy(t)\le C(\bd)e^{C(\bd)(1+T)}\Energy(0).
\end{equation*} 
We then deduce from \eqref{eq_MHD1} that
\begin{equation*}
\norm{p}_{H^{\frac {5}{2}}(\Omega_t)}^2\le  C(\bd)\Paren{1+\norm{h\cdot\nabla h-\DT u}_{H^{\frac 32}(\Omega_t)}^2}\le C(\bd)e^{C(\bd)(1+T)}\Energy(0).
\end{equation*}
Utilizing Lemma \ref{lem_sffcontrolledbymc} and applying the arguments 
from Proposition \ref{pro_E(0)}, we further obtain
\begin{equation*}
\norm{\sff}_{H^{2}(\parOmega_t)}^2+\norm{\DT p}_{H^1(\Omega_t)}^2\le C(\bd)e^{C(\bd)(1+T)}\Energy(0).
\end{equation*} 
Combining these estimates with Proposition \ref{pro_E(0)}, the desired 
bound \eqref{eq_mainresult1} follows.

\vspace{2mm}
\noindent\textbf{Step 2.}
It remains to prove that the a priori assumptions \eqref{eq_aprioriassumption} 
hold for some time $T_0\ge c_0>0$, where the constant $c_0$ depends on 
$\assone_0, \norm{u_0}_{H^3(\Omega_0)}, \norm{h_0}_{H^3(\Omega_0)}$, and 
$\norm{\mc}_{H^2(\parOmega_0)}$.

To this end, we define the following auxiliary quantity:
\begin{equation*}
\Lambda(t) \coloneqq \norm{\sff}_{L^4(\parOmega_t)}^4+\norm{\nabla p}_{L^2(\Omega_t)}^2+\norm{\nabla u}_{L^4(\Omega_t)}^4+\norm{\nabla h}_{L^4(\Omega_t)}^4+1,\quad t\ge 0.
\end{equation*} 
We define $T_0\in (0, 1]$ to be the largest number such that
\begin{equation}\label{eq_time_T0}
[0,T_0]\subset\left\lbrace t\in[0,1]:\Lambda(t)\le 2\Lambda(0),\assone_t \ge \assone_0/2,\ \text{and}\ \energy(t)\le1+\energy(0)\right\rbrace.
\end{equation}
Here, we assume that $T_0<1$, since the claim would be trivial otherwise.

Then, by \eqref{eq_MHD4}, \eqref{eq_Hfrac121}, and the first condition in 
\eqref{eq_time_T0}, we have
\begin{equation*}
\norm{\mc}_{H^{\frac 12}(\parOmega_t)}^2\le C\Paren{\norm{\sff}_{L^2(\parOmega_t)}^2+\norm{\nabla p}_{L^2(\Omega_t)}^2}\le C(\Lambda(0)).
\end{equation*}

We invoke the following regularity result (cf. \cite[Proposition A.2]{Shatah2008a}) to recover the regularity of the height function $\eta$:
\begin{lemma}\label{lem_boudr}
Let $\Omega$ be a domain whose boundary does not 
self-intersect and satisfies $\parOmega\in H^{s_0}$ for some $s_0>2$. 
Suppose that $\norm{\mc}_{H^{s-2}(\parOmega)}\le C_*$ for some 
$s>s_0$ and a positive constant $C_*$. Then $\parOmega\in H^s$, 
and $\norm{\parOmega}_{H^{s}}\le C(C_*)$.
\end{lemma}
By Lemma \ref{lem_boudr}, we have
\begin{equation}\label{eq_boundh}
\norm{\eta(\cdot,t)}^2_{H^{\frac52}(\Gamma)}\le C(\Lambda(0)).
\end{equation}  
Using the definition of $\Lambda(t)$, \eqref{eq_time_T0}, and \eqref{eq_boundh}, 
we have
\begin{align}
\asstwo_{T_0}&\le \sup_{t\in[0,T_0)}\Energy(t)+\sup_{t\in[0,T_0)}\Paren{\norm{\eta(\cdot,t)}^2_{H^{\frac52}(\Gamma)}+\norm{u_n}_{H^{2}(\parOmega_t)}}\nonumber\\ 
&\le  C(\Lambda(0))+\sup_{t\in[0,T_0)}\Energy(t)+\sup_{t\in[0,T_0)}\norm{u_n}_{H^{2}(\parOmega_t)}.\label{eq_NT0le}
\end{align}
By Proposition \ref{pro_E(0)}, $\Lambda(0)\le C_{\operatorname{initial}}$, 
and thus,
\begin{equation}\label{eq_mceta}
\norm{\mc}_{H^{\frac 12}(\parOmega_t)}^2+\norm{\eta(\cdot,t)}^2_{H^{\frac52}(\Gamma)}\le C_{\operatorname{initial}},
\end{equation}
where $C_{\operatorname{initial}}$ depends on $\norm{u_0}_{H^3(\Omega_0)}, 
\norm{h_0}_{H^3(\Omega_0)}$, and $\norm{\mc}_{H^2(\parOmega_0)}$ as defined in 
Proposition \ref{pro_E(0)}.

We only need to estimate $\norm{u_n}_{H^{2}(\parOmega_t)}$. From the decomposition
\begin{equation*}
\Bdnabla u_n=\Bdnabla u\cdot \normal-u\star\sff,
\end{equation*} 
and the second tangential gradient
\begin{equation*}
\Bdnabla^2 u_n=\Bdnabla^2 u\star \normal+\Bdnabla u\star\sff+u\star \Bdnabla \sff.
\end{equation*} 
We can apply Lemma \ref{lem_sffcontrolledbymc} (since we have \eqref{eq_mceta}) 
and the arguments in \eqref{eq_BHfrac32} to bound
\begin{equation}\label{eq_unh2}
\norm{u_n}_{H^{2}(\parOmega_t)}\le  C\Paren{\norm{u}_{H^2(\parOmega_t)},\norm{\sff}_{H^1(\parOmega_t)}}\le C\Paren{\norm{u}_{H^3(\Omega_t)},\norm{\mc}_{H^1(\parOmega_t)}}\le C\Paren{\Energy(t),C_{\operatorname{initial}}}.
\end{equation}
Consequently,
\begin{equation*}
\asstwo_{T_0}\le C\Paren{\sup_{t\in[0,T_0)}\Energy(t),C_{\operatorname{initial}}}.
\end{equation*}

By the conditions in \eqref{eq_time_T0} and $\Lambda(0)\le C_{\operatorname{initial}}$, 
an application of Proposition \ref{pro_Elee} yields
\begin{equation}\label{eq_mainthmElee}
\Energy(t)\le C_{\operatorname{initial}}\Paren{1+\energy(t)}.
\end{equation}  

We note that the last condition in \eqref{eq_time_T0}, \eqref{eq_mainthmElee}, 
together with Proposition \ref{pro_E(0)} implies that 
\begin{equation}\label{eq_mainthmsupE}
\sup_{t\in[0,T_0)}\Energy(t)\le C\Paren{1+\energy(t)}\le C\Paren{2+\energy(0)}\le C\Energy(0)\le C_{\operatorname{initial}}.
\end{equation}

Combining the above analysis, we conclude that
\begin{equation}\label{eq_NT0} 
\asstwo_{T_0} \le C_{\operatorname{initial}}.
\end{equation}

Since the a priori assumptions \eqref{eq_aprioriassumption} hold for time 
$T=T_0$, the claim follows once we show that $T_0$ specified in \eqref{eq_time_T0} 
has a lower bound $c_0>0$. From the definition of $T_0$, at least one of 
the three conditions is satisfied with equality.

\vspace{2mm}
\noindent\textbf{Case 1: $\Lambda(T_0) = 2\Lambda(0)$.}

We assume that $\Lambda(T_0) = 2\Lambda(0)$.
We will show that
\begin{equation}\label{eq_LambdaleELambda}
\frac{d}{dt}\Lambda(t)\le C\Energy(t) \Lambda(t).
\end{equation}  
Applying \eqref{eq_RT1} and \eqref{eq_Dt,nabla}, we have
\begin{align*}
\frac{d}{dt}\Paren{\norm{\nabla u}_{L^4(\Omega_t)}^4+\norm{\nabla h}_{L^4(\Omega_t)}^4}={}&\int_{\Omega_t}\DT \nabla  u\star\Paren{\nabla u}^{\star,3}+ \DT \nabla h\star\Paren{\nabla h}^{\star,3}dx\\
={}&\int_{\Omega_t} \nabla \DT u\star\Paren{\nabla u}^{\star,3}+\nabla \DT  h\star\Paren{\nabla h}^{\star,3}dx\\
&+\int_{\Omega_t}\Paren{\nabla u}^{\star,5}+\Paren{\nabla h}^{\star,5}dx,
\end{align*}
where $T^{\star,m}$ denotes the $m$-fold $\star$-product of the tensor $T$. Therefore,
\begin{align*}
\module{\frac{d}{dt}\Paren{\norm{\nabla u}_{L^4(\Omega_t)}^4+\norm{\nabla h}_{L^4(\Omega_t)}^4}}
\le{}& C\norm{\nabla\DT u}_{L^3(\Omega_t)}\norm{\nabla u}_{L^\infty(\Omega_t)}\norm{\nabla u}_{L^3(\Omega_t)}\norm{\nabla u}_{L^3(\Omega_t)}\\
&+C\norm{\nabla\DT h}_{L^3(\Omega_t)}\norm{\nabla h}_{L^\infty(\Omega_t)}\norm{\nabla h}_{L^3(\Omega_t)}\norm{\nabla h}_{L^3(\Omega_t)}\\
&+C\norm{\nabla u}_{L^\infty(\Omega_t)}\norm{\nabla u}_{L^4(\Omega_t)}\norm{\nabla u}_{L^4(\Omega_t)}\norm{\nabla u}_{L^4(\Omega_t)}\norm{\nabla u}_{L^4(\Omega_t)}\\
&+C\norm{\nabla u}_{L^\infty(\Omega_t)}\norm{\nabla h}_{L^4(\Omega_t)}\norm{\nabla h}_{L^4(\Omega_t)}\norm{\nabla h}_{L^4(\Omega_t)}\norm{\nabla h}_{L^4(\Omega_t)}\\  
\le{}& C\Energy(t)\Lambda(t).
\end{align*}
Similarly, for the pressure term, we have
\begin{align*}
\frac{d}{dt}\norm{\nabla p}_{L^2(\Omega_t)}^2={}&\int_{\Omega_t}\DT\nabla p:\nabla p dx\\
={}&\int_{\Omega_t}[\DT,\nabla] p:\nabla p dx+\int_{\Omega_t}\nabla \DT p:\nabla p dx\\ 
={}&\int_{\Omega_t}\nabla u\star\nabla p\star\nabla p dx+\int_{\Omega_t} \lpe\DT p\, p dx+\int_{\parOmega_t}\partial_\normal\DT p\, p dS\\
\eqqcolon{}&\Lambda_1(t)+\Lambda_2(t)+\Lambda_3(t).
\end{align*}
We estimate these terms as follows:
\begin{equation*}
\module{\Lambda_1(t)}\le C\norm{\nabla u}_{L^\infty(\Omega_t)}\norm{\nabla p}_{L^2(\Omega_t)}^2\le C\Energy(t)\Lambda(t).
\end{equation*}
To handle $\Lambda_2(t)$, we recall from \cite[Lemma 2.11]{Hao2025} that
\begin{align*}
-\lpe \DT p={}& \divergence\divergence\Paren{u\otimes\DT u}+\divergence \Paren{\nabla u\star\DT u+\nabla u\star\nabla u\star u}+\nabla^2 u\star \nabla h\star h\\
&+\nabla^2 h\star\nabla u\star h+\nabla^2 h\star\nabla h\star u+\nabla u\star\nabla h\star\nabla h.
\end{align*}
Consequently, we have
\begin{align*}
\lpe\DT p={}& u\star\nabla\divergence\DT u+\nabla^2 u\star\DT u+\nabla u\star\nabla\DT u+\nabla^2u\star\nabla u\star u+\nabla^2u\star\nabla h\star h\\
&+\nabla^2h\star\nabla u\star h+\nabla^2h\star\nabla h\star u+\nabla u\star\nabla u\star \nabla u+\nabla u\star\nabla h\star \nabla h.
\end{align*}
Then, we derive the following estimate:
\begin{equation*} 
\module{\Lambda_2(t)}\le  C\norm{\lpe\DT  p}_{L^{2}(\Omega_t)}\norm{p}_{L^2(\Omega_t)} \le C\Energy(t)\Lambda(t),
\end{equation*}
since the products formed by the multiplication of two terms satisfy
\begin{align*}
\norm{u\star\nabla\divergence\DT u}_{L^2(\Omega_t)}&\le C \norm{u\star\nabla u\star\nabla^2 u}_{L^2(\Omega_t)}\le C\Energy(t)\sqrt{\Lambda(t)}, \\
\norm{\nabla^2 u\star\DT u}_{L^2(\Omega_t)}&\le C \norm{\nabla^2 u}_{L^6(\Omega_t)}\norm{\DT u}_{L^3(\Omega_t)}\le C\Energy(t), \\
\norm{\nabla u\star\nabla\DT u}_{L^2(\Omega_t)}&\le C\norm{\nabla u}_{L^6(\Omega_t)}\norm{\nabla\DT u}_{L^3(\Omega_t)}\le C\Energy(t),
\end{align*}
and the multiplication of three terms satisfies
\begin{align*} 
\norm{\nabla^2u\star\nabla u\star u}_{L^2(\Omega_t)}&\le C \norm{\nabla^2u}_{L^4(\Omega_t)}\norm{\nabla u}_{L^4(\Omega_t)}\norm{u}_{L^\infty(\Omega_t)} \le C\Energy(t)\sqrt{\Lambda(t)},\\
\norm{\nabla u\star\nabla u\star\nabla u}_{L^2(\Omega_t)}&\le C \norm{\nabla u}_{L^\infty(\Omega_t)}\norm{\nabla u}_{L^4(\Omega_t)}\norm{\nabla u}_{L^4(\Omega_t)} \le C\Energy(t)\sqrt{\Lambda(t)}.
\end{align*}
For the last term $\Lambda_3(t)$, note that by the normal trace theorem 
and the trace theorem, we have
\begin{align*}
\module{\Lambda_3(t)}&\le C\norm{\partial_\normal\DT  p}_{H^{-\frac 12}(\parOmega_t)}\norm{p}_{H^{\frac 12}(\parOmega_t)}\\
&\le C\Paren{\norm{\nabla\DT  p}_{L^2(\Omega_t)}+\norm{\divergence\nabla\DT  p}_{H^{-1}(\Omega_t)}}\sqrt{\Lambda(t)}\\
&\le C \norm{\nabla\DT  p}_{L^2(\Omega_t)}\sqrt{\Lambda(t)},
\end{align*}
where we have used the following estimate for the divergence term:
\begin{align*}
\norm{\divergence\nabla\DT  p}_{H^{-1}(\Omega_t)}
\le{}&\sup \Brace{\module{\int_{\Omega_t}\divergence\nabla\DT  p qdx}:q\in H^1_0(\Omega_t),\norm{q}_{H^1_0(\Omega_t)}\le 1}\nonumber\\
\le{}&\sup \Brace{\module{\int_{\Omega_t}\nabla \DT p\cdot \nabla qdx}:q\in H^1_0(\Omega_t),\norm{q}_{H^1_0(\Omega_t)}\le 1}\nonumber\\
\le{}& \norm{\nabla\DT  p}_{L^{2}(\Omega_t)}.
\end{align*}  
Then, by \eqref{eq_MHD1} and \eqref{eq_Dt,nabla}, it suffices to bound
\begin{align*}
\norm{\nabla\DT  p}_{L^{2}(\Omega_t)}\le{}& C\Paren{\norm{[\nabla,\DT]  p}_{L^{2}(\Omega_t)}+\norm{\DT\nabla p}_{L^{2}(\Omega_t)}}\\
\le{}& C\Paren{\norm{\nabla u\star \nabla p}_{L^{2}(\Omega_t)}+\norm{\DT\Paren{-\DT u+h\cdot\nabla h}}_{L^{2}(\Omega_t)}}\\
\le{}&C\bigg(\norm{\nabla u}_{L^\infty(\Omega_t)}\norm{\nabla p}_{L^2(\Omega_t)}+\norm{\DT^2u}_{L^2(\Omega_t)}+\norm{\nabla h}_{L^\infty(\Omega_t)}\norm{\DT h}_{L^2(\Omega_t)}\\
&\qquad+\norm{h}_{L^4(\Omega_t)}\norm{\nabla u}_{L^\infty(\Omega_t)}\norm{\nabla h}_{L^4(\Omega_t)}+\norm{h}_{L^6(\Omega_t)}\norm{\nabla\DT h}_{L^3(\Omega_t)}\bigg)\\
\le{}&C\Energy(t)\sqrt{\Lambda(t)}.
\end{align*}
Therefore, we have
\begin{equation*}
\module{\frac{d}{dt}\norm{\nabla p}_{L^2(\Omega_t)}^2}\le C\Energy(t)\Lambda(t).
\end{equation*}

Similarly, invoking the identity
\begin{equation*}
\DT\sff=-\Bdnabla^2u\star\normal-\Bdnabla u\star \sff,
\end{equation*} 
we obtain 
\begin{align*}
\frac{d}{dt} \norm{\sff}_{L^4(\parOmega_t)}^4={}&\int_{\parOmega_t}\DT \sff\star\sff^{\star,3} dS+\int_{\parOmega_t}|\sff|^4\Bddiv udS\\
={}&-\int_{\parOmega_t}\Bdnabla^2u\star\normal\star\sff^{\star,3} dS-\int_{\parOmega_t}\Bdnabla u\star \sff^{\star,4} dS+\int_{\parOmega_t}|\sff|^4\Bddiv udS,
\end{align*} 
and
\begin{align*}
\module{\frac{d}{dt} \norm{\sff}_{L^4(\parOmega_t)}^4}&\le C\Paren{\norm{\Bdnabla^2u}_{L^4(\parOmega_t)}\norm{\sff}_{L^4(\parOmega_t)}^3+\norm{\Bdnabla u}_{L^\infty(\parOmega_t)}\norm{\sff}_{L^4(\parOmega_t)}^4}\\
&\le C\Energy(t)\Lambda(t).
\end{align*}
We conclude that \eqref{eq_LambdaleELambda} follows. Therefore,
\begin{equation}\label{eq_LambdaleELambda2}
\frac{d}{dt}\Lambda(t)\le C\Energy(t) \Lambda(t) \le C_{\operatorname{initial}}\Lambda(t).
\end{equation}  

Integrating \eqref{eq_LambdaleELambda2} over $(0,T_0)$ and using 
$\Lambda(T_0) = 2 \Lambda(0)$, we obtain
\begin{equation*}
\ln 2=\ln \Lambda(T_0)-\ln \Lambda(0)\le C_{\operatorname{initial}}T_0.
\end{equation*}
This yields
\begin{equation*}
T_0\ge c_0,
\end{equation*}
where the constant $c_0$ depends only on the initial data.

\vspace{2mm}
\noindent\textbf{Case 2: $\energy(T_0)= 1+\energy(0)$.}

A similar argument applies if we have an equality in the third condition, i.e.,
\begin{equation*}
\energy(T_0)= 1+\energy(0).
\end{equation*}  
In this case, the a priori assumptions hold by \eqref{eq_NT0}, whereby 
we can apply the a priori estimates and \eqref{eq_mainthmsupE} to obtain
\begin{equation*}
\frac{d}{dt}\energy(t)\le C_{\operatorname{initial}}.
\end{equation*}
Integrating over $(0,T_0)$ gives
\begin{equation*}
1=\energy(T_0)-\energy(0)\le C_{\operatorname{initial}}T_0,
\end{equation*}
which again results in
\begin{equation*}
T_0 \ge c_0>0.
\end{equation*}

\vspace{2mm}
\noindent\textbf{Case 3: $\assone_{T_0} = \assone_0/2$.}

Finally, suppose that an equality in the second condition occurs, i.e., 
$\assone_{T_0} = \assone_0/2$. Recalling that 
\begin{equation*}
\assone_{T_0}= \radi - \sup_{t\in[0,T_0)}\norm{\eta(\cdot, t)}_{L^\infty(\Gamma)},
\end{equation*} 
and $\assone_0>0$, we define $0<T_1 \le T_0$ by
\begin{equation*}
\assone_{T_0} =  \radi- \norm{\eta(\cdot,T_1)}_{L^\infty(\Gamma)}.
\end{equation*}
It is clear that $\norm{u_n}_{L^\infty(\parOmega_t)}^2\le C\Energy(t) \le C_{\operatorname{initial}}$ 
by \eqref{eq_mainthmsupE}. We apply 
the fundamental theorem of calculus to find that
\begin{align*}
\assone_{T_0}&=\radi-\norm{\eta(\cdot,T_1)}_{L^\infty(\parOmega)}\\
&\ge  \radi-\norm{\eta_0}_{L^\infty(\parOmega)}-\int_0^{T_1} \norm{u_n}_{L^\infty(\parOmega_t)}\, dt \\
&\ge \assone_{0}- C_{\operatorname{initial}}T_1,
\end{align*} 
recalling $\assone_{T_0}=\assone_0/2$, which implies
\begin{equation*}
T_0\ge T_1\ge \assone_{0}/\Paren{2C_{\operatorname{initial}}}>0.
\end{equation*}
This establishes the desired lower bound for $T_0$.
\end{proof} 
\begin{proof}[Proof of Theorem \ref{thm_main2}]
We proceed by contradiction. Assume that $T_{\operatorname{max}}<\infty$ 
and either $(u,h)(\cdot,T_{\operatorname{max}})\notin H^3(\Omega_{T_{\operatorname{max}}})\times H^3(\Omega_{T_{\operatorname{max}}})$ or $\parOmega_{T_{\operatorname{max}}}\notin H^4$. Suppose further that none 
of the scenarios (1)--(4) occur, which implies that
\begin{equation*}
\inf_{t\in[0,T_{\operatorname{max}})}\radi(\Omega_t)>0,\quad \sup_{t\in[0,T_{\operatorname{max}})}\norm{\parOmega_t}_{H^{\frac52}}<\infty,
\end{equation*} 
and
\begin{equation*}
\sup_{t\in[0,T_{\operatorname{max}})}\Paren{\norm{\nabla u}_{L^{\infty}(\Omega_t)}+\norm{\nabla h}_{L^{\infty}(\Omega_t)}+\norm{\nabla^2h}_{L^{2}(\Omega_t)}+\norm{u_n}_{H^{2}(\parOmega_t)}}<\infty,
\end{equation*}
where we have applied Lemma \ref{lem_boudr}. In particular, 
$\radi(\Omega_{T_{\operatorname{max}}})>0$, and we choose $\parOmega_{T_{\operatorname{max}}}$ 
as the reference surface to represent the free boundary over a short 
time interval before $T_{\operatorname{max}}$. More precisely, the height 
function $\eta(\cdot,t)$ is well-defined on $\left[T_{\operatorname{max}}-\vare,T_{\operatorname{max}}\right)$ 
for sufficiently small $\vare>0$, and we have
\begin{equation*}
\sup_{t\in\left[T_{\operatorname{max}}-\vare,T_{\operatorname{max}}\right)}\norm{\eta}_{H^{\frac52}(\parOmega_{T_{\operatorname{max}}})}<\infty.
\end{equation*} 
Applying the estimates in Theorem \ref{thm_main1}, it follows that 
$u(\cdot,T_{\operatorname{max}})\in H^3(\Omega_{T_{\operatorname{max}}}), 
h(\cdot,T_{\operatorname{max}})\in H^3(\Omega_{T_{\operatorname{max}}})$, 
and $ \parOmega_{T_{\operatorname{max}}}\in H^4$, which implies that the solution 
can be extended beyond $T_{\operatorname{max}}$. This contradicts the maximality 
of $T_{\operatorname{max}}$, thereby completing the proof.
\end{proof}

\section{Geometric setup for the self-intersection singularity}\label{sec4}

In the remainder of this paper, we prove that the singular scenario (1) 
predicted in Theorem \ref{thm_main2} does indeed occur. We begin by 
introducing the geometric and analytical setup required for the 
construction of such a self-intersecting solution.

\subsection{Definition of the self-intersecting domain}

Heuristically, a self-intersecting domain $\Omega_\ddagger$ is an open and 
bounded subset of $\mathbb{R}^3$ that lies locally on one side of its 
boundary, except at a single point $x_0\in\parOmega_\ddagger$, where the domain 
locally occupies both sides of the tangent plane at $x_0$. To provide a 
precise definition of such a self-intersecting domain, we first introduce 
the following notation.

Let $B(0,r)$ be the open ball centered at the origin with radius $r>0$. 
We denote
\begin{equation*}
B_r=B(0,r),\quad B^+_r=B_r\cap\Brace{x^3>0},\ \text{and}\ B^-_r=B_r\cap\Brace{x^3<0}.
\end{equation*}
For brevity, we also set
\begin{equation*}
B=B_1,\quad B^\pm=B^\pm_1,\ \text{and}\ B^{\sharp}=\overline{B}\cap\Brace{x^3=0}.
\end{equation*} 
Recall that for a non-self-intersecting domain $\Omega$ of class 
$H^s$ (where $s$ is sufficiently large), a collection of boundary charts 
$\Brace{\cov_l}_{l=1}^K$ forms an open cover of $\parOmega$ such that 
for each $l\le K$, there exists an $H^s$-diffeomorphism $\loc_l$ satisfying 
\begin{equation*}
\loc_l: B\to\cov_l\ \text{is an}\ H^{s}\ \text{diffeomorphism,}\  \loc_l(B^+)=\cov_l\cap\Omega,\  \text{and}\ \loc_l(B^{\sharp})=\overline{\cov_l}\cap\parOmega.
\end{equation*} 
The interior charts of $\Omega$, denoted by $\Brace{\cov_l}_{l=K+1}^L$, 
consist of a family of open subsets of $\Omega$ such that 
$\Brace{\cov_l}_{l=1}^K\cup \Brace{\cov_l}_{l=K+1}^L$ is an open cover 
of $\Omega$, and there exist $H^s$-diffeomorphisms
\begin{equation*}
\loc_l:B \to \cov_l,
\end{equation*}  
for $K+1\le l\le L$.

In what follows, we restrict 
our attention to axisymmetric self-intersecting fluid regions. These can 
be viewed as being generated by rotating a two-dimensional self-intersecting 
domain, which is symmetric with respect to both the $x^1$ and $x^3$ axes, 
around a specified coordinate axis. Fig.\,\ref{f6} and Fig.\,\ref{f1} illustrate 
self-intersecting fluid domains without curvature blow-up (obtained by 
rotation around the $x^3$-axis) and with curvature blow-up (obtained by 
rotation around the $x^1$-axis), respectively.

As discussed in Section \ref{sec_strategy}, the local structure of the construction suggests possible extensions to other configurations of boundary self-intersection, such as multiple-point contacts or splat singularities along curves or sets of
positive surface measure.
\begin{figure}[htbp] 
\centering
\includegraphics[width=14cm]{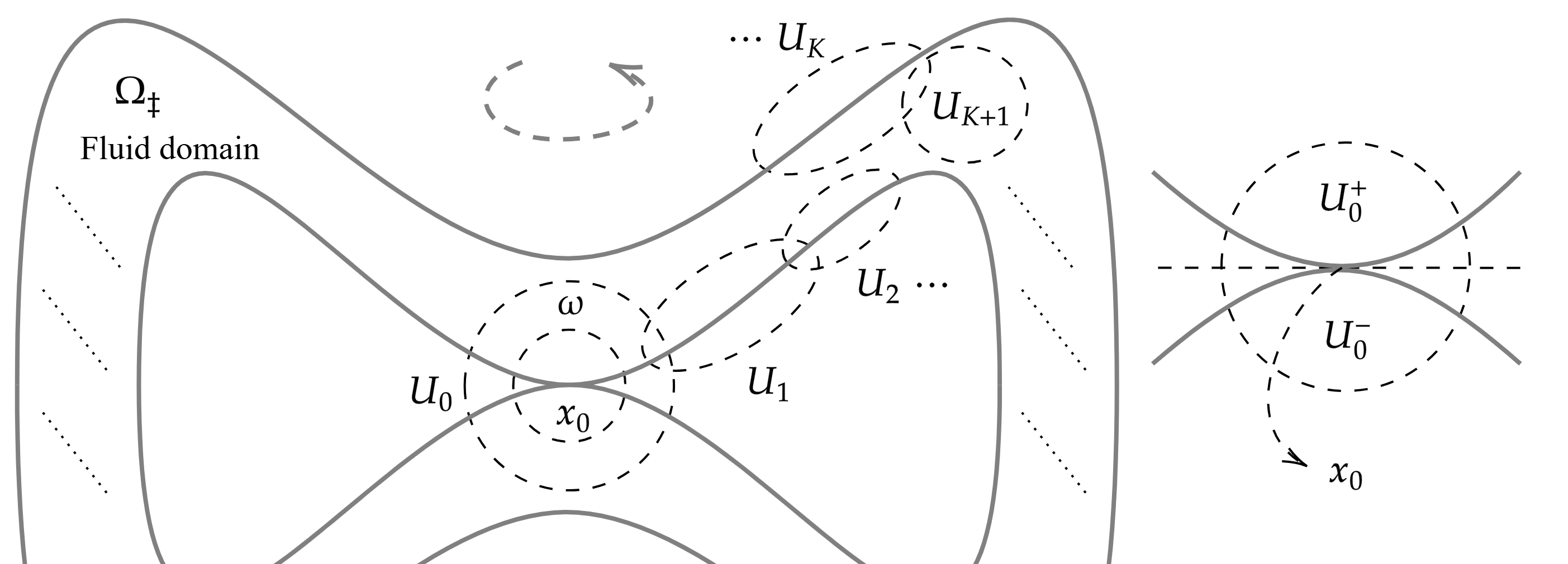} 
\caption{Open covers near the self-intersection point.}
\label{f5}
\end{figure}
\begin{definition}
Assume that $s\ge \frac92$. We say that $\Omega_\ddagger$ is an $H^s$-class self-intersecting domain (with 
a single-point contact) if it is defined by a collection of open covers
\begin{equation*}
\Brace{\cov_0}\cup\Brace{\cov_l}_{l=1}^K\cup\Brace{\cov_l}_{l=K+1}^L,
\end{equation*} 
and associated maps 
\begin{equation*}
\Brace{\loc_\pm}\cup\Brace{\loc_l}_{l=1}^K\cup\Brace{\loc_l}_{l=K+1}^L
\end{equation*} 
satisfying the following properties (1)--(5) (see Fig.\,\ref{f5}).

\begin{enumerate}
\item Without loss of generality, we assume that $x_0=\Paren{x_0^1,x_0^2,x_0^3}\in\parOmega_\ddagger$ 
is the unique self-intersection point, and that the tangent plane at 
$x_0$ is the horizontal plane $x^3=0$ (assuming $x_0^3=0$).
\item Let $\cov_0$ denote an open neighborhood of $x_0$. 
We choose additional boundary charts $\Brace{\cov_l}_{l=1}^K$ such that 
the collection $\Brace{\cov_l}_{l=0}^K$ forms an open cover of $\parOmega_\ddagger$. 
Then, let $\Brace{\cov_l}_{l=K+1}^{L}$ denote a family of open sets 
contained in $\Omega_\ddagger$ such that $\Brace{\cov_l}_{l=0}^{L}$ is an open 
cover of $\Omega_\ddagger$. We choose a sufficiently small neighborhood 
$\omega\subset\cov_0$ containing $x_0$ with the property that
\begin{equation*}
\overline{\omega}\cap\overline{\cov_l}=\emptyset,\quad 1\le l\le L.
\end{equation*} 
We define the upper and lower halves of the neighborhood in $\Omega_\ddagger$ as
\begin{equation*} 
\cov_0^+=\cov_0\cap\Omega_\ddagger\cap\Brace{x^3>x_0^3}\ \text{and}\ \cov_0^-=\cov_0\cap\Omega_\ddagger\cap\Brace{x^3<x_0^3}.
\end{equation*}
Additionally, we assume that 
\begin{equation*}
\overline{\cov_0}\cap\overline{\Omega_\ddagger}\cap\Brace{x^3=x_0^3}=\Brace{x_0},
\end{equation*} 
which implies, in particular, that $\cov_0^+$ and $\cov_0^-$ are connected. 
Due to the assumed symmetry, we have
\begin{equation*}
\cov_0^+= \Brace{(x^1,x^2,-x^3):(x^1,x^2,x^3)\in\cov_0^-},
\end{equation*} 
and we require that the one-sided mean curvatures of $\Omega_\ddagger$ 
(evaluated on the boundaries of $\cov_0^+$ and $\cov_0^-$) coincide 
at $x_0$:
\begin{equation*}
\mc^+=\mc_{\cov_0^+}(x_0)=\mc_{\cov_0^-}(x_0)=\mc^-.
\end{equation*} 
\item For each $1\le l\le K$, there exists an $H^s$-diffeomorphism $\loc_l$ 
satisfying
\begin{equation*}
\loc_l:B\to\cov_l,\ \loc_l(B^+)=\cov_l\cap\Omega_\ddagger,\ \text{and}\ \loc_l(B^{\sharp})=\overline{\cov_l}\cap\parOmega_\ddagger.  
\end{equation*}
We assume that there exist positive constants $C_l$ for $1\le l\le K$ 
such that 
\begin{equation*}
\operatorname{det}(\nabla \loc_l)=C_l.
\end{equation*}
This requirement can always be satisfied by appealing to classical 
results on prescribed Jacobian determinants \cite{Dacorogna1990}.
\item For $K+1\le l\le L$, the map $\loc_l:B\to\cov_l$ is an $H^s$-diffeomorphism. 
We similarly assume that there exist positive constants $C_l$ such that
\begin{equation*}
\operatorname{det}(\nabla \loc_l)=C_l.
\end{equation*}
\item To the open set $\cov_0$, we associate two $H^s$-diffeomorphisms 
$\loc_+$ and $\loc_-$ possessing the following properties:
\begin{equation*} 
\loc_+(B^+)=\cov_0^+,\quad \loc_-(B^-)=\cov_0^-,\quad \loc_+(B^{\sharp})=\overline{\cov_0^+}\cap\parOmega_\ddagger,\quad \loc_-(B^{\sharp})=\overline{\cov_0^-}\cap\parOmega_\ddagger,
\end{equation*}
such that (recalling that we assume that $x_0^3=0$)
\begin{align*}
\Paren{\loc_-^1,\loc_-^2,\loc_-^3}\Paren{y^1,y^2,y^3}&=\Paren{\loc_+^1,\loc_+^2,-\loc_+^3}\Paren{y^1,y^2,-y^3},\ \Paren{y^1,y^2,y^3}\in B^-,\\
\loc_+(B^{\sharp})\cap\loc_-(B^{\sharp})&=\Brace{x_0},\ \text{and}\ \loc_{\pm}(0,0,0)=x_0.
\end{align*}  
We further require that the inner regions do not overlap with the 
standard boundary or interior charts: 
\begin{equation*}
\overline{\loc_+(B^+_{1/2})}\cap\overline{\loc_l(B^+)}=\emptyset,\ l\le K,\quad  \overline{\loc_+(B^+_{1/2})}\cap\overline{\loc_l(B)}=\emptyset,\ l\ge K+1,
\end{equation*}  
and
\begin{equation*}
\overline{\loc_-(B^-_{1/2})}\cap\overline{\loc_l(B^+)}=\emptyset,\ l\le K,\quad  \overline{\loc_-(B^-_{1/2})}\cap\overline{\loc_l(B)}=\emptyset,\ l\ge K+1.
\end{equation*} 
\end{enumerate} 
\end{definition}
For a self-intersecting domain $\Omega_\ddagger$ defined as above, we note 
that near the intersection point $x_0$, the Jacobian matrix satisfies
\begin{equation*}
\nabla\loc_-\Paren{y^1,y^2,y^3}=\begin{pmatrix}
\pl_1\loc_+^1 & \pl_2\loc_+^1 & -\pl_3\loc_+^1\\
\pl_1\loc_+^2 & \pl_2\loc_+^2 & -\pl_3\loc_+^2\\
-\pl_1\loc_+^3 & -\pl_2\loc_+^3 & \pl_3\loc_+^3
\end{pmatrix}\Paren{y^1,y^2,-y^3}.
\end{equation*}
It then follows that the determinants are equal:
\begin{equation*}
\operatorname{det}(\nabla\loc_-)=-\operatorname{det}\begin{pmatrix}
\pl_1\loc_+^1 & \pl_2\loc_+^1 & \pl_3\loc_+^1\\
\pl_1\loc_+^2 & \pl_2\loc_+^2 & \pl_3\loc_+^2\\
-\pl_1\loc_+^3 & -\pl_2\loc_+^3 & -\pl_3\loc_+^3
\end{pmatrix}=\operatorname{det}\begin{pmatrix}
\pl_1\loc_+^1 & \pl_2\loc_+^1 & \pl_3\loc_+^1\\
\pl_1\loc_+^2 & \pl_2\loc_+^2 & \pl_3\loc_+^2\\
\pl_1\loc_+^3 & \pl_2\loc_+^3 & \pl_3\loc_+^3
\end{pmatrix}=\operatorname{det}(\nabla\loc_+).
\end{equation*}
By applying the prescribed Jacobian construction \cite{Dacorogna1990} to $\loc_+$ and defining $\loc_-$ through the above symmetry, we may further arrange that
\begin{equation}\label{eq_detC0}
\operatorname{det}(\nabla\loc_\pm)=C_0>0.
\end{equation}   
\subsection{Approximate fluid domains}
To approximate a self-intersecting domain $\Omega_\ddagger$ using a family of non-self-intersecting regular domains $\Omega^\pr$, we approximate the two distinguished charts $\loc_\pm$ 
by a family of charts $\loc_\pm^\pr$ such that
\begin{equation*}
\loc_-^\pr(B^{\sharp})\cap\loc_+^\pr(B^{\sharp})=\emptyset,\ \forall \pr>0,\ \text{and}\ \loc_{\pm}^\pr\to\loc_{\pm}\ \text{in}\ H^s,\ \pr \to 0,
\end{equation*}  
where $s\ge\frac92$.
We choose a sufficiently small constant $\mathfrak{r}\in (0, 1/4)$ so that
\begin{equation*}
\loc_-(B^-_{2\mathfrak{r}})\subset\omega\ \text{and}\ \loc_+(B^+_{2\mathfrak{r}})\subset\omega.
\end{equation*}
Fixing an integer $3\le k_0\le s-\frac32$, we let $\psi$ 
denote a smooth bump function satisfying
\begin{equation*}
\psi\in\mathcal{D}(B(0,\mathfrak{r})),\quad 0\le\psi\le1,\quad \psi(0)=1,
\end{equation*} 
and
\begin{equation}\label{eq_psinorm}
\norm{\psi}_{H^{k_0+1}(B)}\le \norm{\nabla\loc_\pm}_{L^2(B^\pm)}.
\end{equation}  
Then, for $\pr>0$ chosen sufficiently small, we define (see Fig.\,\ref{f7})
\begin{equation*} 
\loc_-^\pr(y)=\loc_-(y)- (0,0,\pr\psi(y))^\top,\quad \loc_+^\pr(y)=\loc_+(y)+ (0,0,\pr\psi(y))^\top.
\end{equation*}  
\begin{figure}[htbp] 
\centering
\includegraphics[width=14cm]{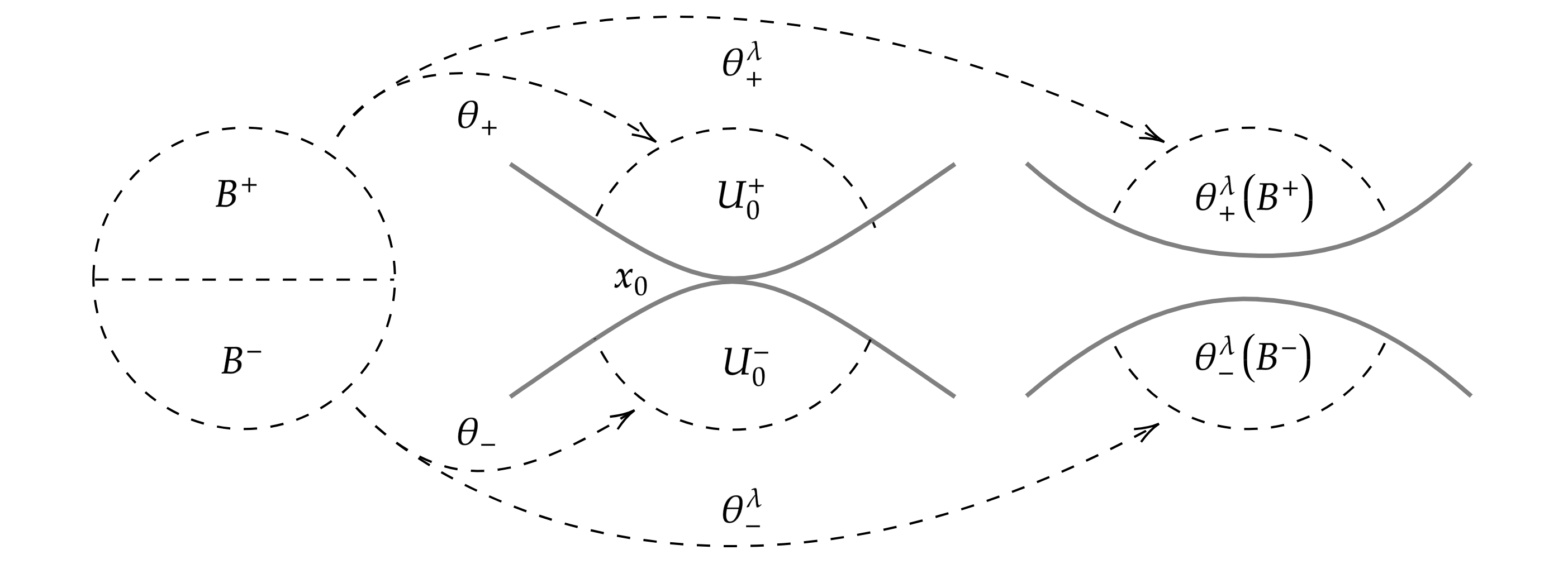} 
\caption{Small perturbations of the local coordinates.}
\label{f7}
\end{figure}

A direct calculation shows that
\begin{equation*}
\nabla\loc_\pm^\pr(y)=\nabla \loc_\pm(y)\pm\pr\begin{pmatrix}
0 & 0 & 0\\
0 & 0 & 0\\
\partial_1\psi(y) & \partial_2\psi(y) & \partial_3\psi(y)
\end{pmatrix}.
\end{equation*}
Since the Jacobian matrix $\nabla \loc_\pm$ is invertible, the identity
\begin{equation*}
\det(A+B) = \det(A)\left[ 1 + \operatorname{tr}(A^{-1}B) \right] + O(\norm{B}^2),\quad \norm{B}\to 0,
\end{equation*}
implies that
\begin{equation*}
\operatorname{det}\Paren{\nabla\loc_\pm^\pr}=\operatorname{det}\Paren{\nabla\loc_\pm}\Bracket{1\pm\pr\operatorname{tr}\Paren{\nabla\loc_\pm^{-1}\begin{pmatrix}
0 & 0 & 0\\
0 & 0 & 0\\
\partial_1\psi & \partial_2\psi & \partial_3\psi
\end{pmatrix}}}+ O\Paren{\pr^2},\quad \pr\to 0.
\end{equation*}
As a result, for $\pr>0$ sufficiently small, by \eqref{eq_detC0}, we have
\begin{equation}\label{eq_jacobian}
\frac12 C_0\le \operatorname{det}\Paren{\nabla\loc_\pm^\pr}\le \frac32 C_0,
\end{equation} 
Moreover, since 
\begin{equation*}
\norm{\nabla\loc_\pm}_{H^k(B^\pm)}-\pr\norm{\nabla\psi}_{H^k(B^\pm)}\le \norm{\nabla\loc_\pm^\pr}_{H^k(B^\pm)}\le\norm{\nabla\loc_\pm}_{H^k(B^\pm)}+\pr\norm{\nabla\psi}_{H^k(B^\pm)},\quad k\le k_0,
\end{equation*}
invoking \eqref{eq_psinorm}, we obtain the following estimate for 
sufficiently small $\pr>0$:
\begin{equation}\label{eq_nablathetavarepsilon}
\frac{1}{2}\norm{\nabla\loc_\pm}_{H^k(B^\pm)}\le \norm{\nabla\loc_\pm^\pr}_{H^k(B^\pm)}\le\frac{3}{2}\norm{\nabla\loc_\pm}_{H^k(B^\pm)}.
\end{equation}

We have ensured that the modification of the domain is localized to 
a small neighborhood of $x_0$, well away from the boundary of $\cov_0$ 
and the images of the remaining maps $\loc_l$ for $1\le l\le K$. 
Since the maps $\loc_\pm^\pr$ only modify $\loc_\pm$ within a very 
small neighborhood of the origin in $B$, it follows that for $\pr>0$ 
sufficiently small,
\begin{equation*}
\overline{\loc_+^\pr(B^+_{1/2})}\cap\overline{\loc_l(B^+)}=\emptyset,\ l\le K,\quad  \overline{\loc_+^\pr(B^+_{1/2})}\cap\overline{\loc_l(B)}=\emptyset,\ l\ge K+1,
\end{equation*}  
and
\begin{equation*}
\overline{\loc_-^\pr(B^-_{1/2})}\cap\overline{\loc_l(B^+)}=\emptyset,\ l\le K,\quad  \overline{\loc_-^\pr(B^-_{1/2})}\cap\overline{\loc_l(B)}=\emptyset,\ l\ge K+1.
\end{equation*} 
Furthermore, for all other indices $1\le l\le L$, we simply set
\begin{equation*}
\loc_l^\pr=\loc_l.
\end{equation*}
Therefore, the collection $\Brace{\loc_\pm^\pr,\loc_1^\pr,\cdots,\loc_L^\pr}$ 
serves as a family of local coordinates. The approximate domain $\Omega^\pr$ (see Fig.\,\ref{f8})
is then defined by
\begin{align*}
\Omega^\pr&=\loc_+^\pr(B^+)\cup\loc_-^\pr(B^-)\cup\bigcup_{l=1}^{K}\loc_l^\pr(B^+)\cup\bigcup_{l=K+1}^{L}\loc_l^\pr(B)\\
&=\loc_+^\pr(B^+)\cup\loc_-^\pr(B^-)\cup\bigcup_{l=1}^{K}\loc_l(B^+)\cup\bigcup_{l=K+1}^{L}\loc_l(B),
\end{align*}
and its boundary is given by
\begin{equation*}
\parOmega^\pr=\loc_+^\pr(B^{\sharp})\cup\loc_-^\pr(B^{\sharp})\cup\bigcup_{l=1}^{K}\loc_l(B^{\sharp}).
\end{equation*}
Note that, in general, $\loc_{\pm}^\pr(B^\pm)\subsetneqq \cov_0^{\pm}$.

The following lemma asserts that the approximating domain preserves 
the original regularity and is free of self-intersections.
\begin{lemma}
For each sufficiently small $\pr>0$, the set $\Omega^\pr$ defined 
by the charts
\begin{equation*}
\loc_\pm^\pr:B^\pm\to\loc_\pm^\pr\Paren{B^\pm},\ \loc_l^\pr:B^+\to\loc_l^\pr\Paren{B^+},\ l\le K,\ \loc_l^\pr:B\to \loc_l^\pr\Paren{B},\ K+1\le l\le L
\end{equation*}
is an $H^s$-class domain that lies locally on one side of its 
$H^{s-\frac 12}$ boundary.
\end{lemma}
\begin{proof}
We analyze the vertical distance between points in $\loc_+^\pr(B^+)$ 
and $\loc_-^\pr(B^-)$. We divide our analysis into two cases.

\vspace{2mm}
\noindent\textbf{Case 1.} By choosing a constant $r_0\in (0, 1/2)$ 
such that $\psi \ge 1/2$ in $B_{r_0}$, we see that
\begin{equation*}
\module{\Paren{\loc_+^\pr(y_1)-\loc_-^\pr(y_2)}\cdot (0,0,1)^\top}\ge \pr, 
\end{equation*} 
for any $y_1\in B^+_{r_0}$ and $y_2\in B^-_{r_0}$.

\vspace{2mm}
\noindent\textbf{Case 2.} Having fixed $r_0$, we recall from assumption 
(2) that the images of $\loc_+$ and $\loc_-$ intersect the plane 
$\Brace{x^3=x_0^3}$ only at the point $x_0$. Consequently, there exists 
a constant $\delta(r_0)>0$, independent of $\pr$, such that
\begin{equation*}
(\loc_+^\pr(y_1) -x_0)\cdot (0,0,1)^\top >\delta(r_0),\  (\loc_-^\pr(y_2)-x_0)\cdot (0,0,1)^\top <-\delta(r_0),\  \forall y_1\in B^+,y_2\in B^-,\module{y_1},\module{y_2}\ge r_0.
\end{equation*}  
This, in turn, implies that if we restrict $\pr\le\delta(r_0)$, we obtain
\begin{equation*}
\module{(\loc^\pr_+(y_1) -\loc^\pr_-(y_2))\cdot (0,0,1)^\top }\ge\delta(r_0)\ge\pr,\ y_1\in B^+,|y_1|\ge r_0,\ \text{and}\ y_2\in B^-,
\end{equation*} 
as well as
\begin{equation*}
\module{(\loc^\pr_+(y_1)-\loc^\pr_-(y_2))\cdot (0,0,1)^\top }\ge\delta(r_0)\ge\pr,\ y_1\in B^+\ \text{and}\ y_2\in B^-,\ |y_2|\ge r_0.
\end{equation*}

Combining the above cases, we conclude that for any $\pr \in (0, \delta(r_0)]$, 
the following inequality holds:
\begin{equation}\label{eq_verticallocal}
\module{(\loc^\pr_+(y_1)-\loc^\pr_-(y_2))\cdot (0,0,1)^\top }\ge\pr,\quad \forall y_1 \in B^+,y_2\in B^-.
\end{equation}
Thus, near the self-intersection point $x_0$, the boundaries of the 
approximating domain $\Omega^\pr$ remain strictly separated, ensuring 
that the region is free of self-intersections. Since the remaining 
local coordinates coincide with those of $\Omega_\ddagger$, the regularity 
of the approximating domain $\Omega^\pr$ is identical to that of $\Omega_\ddagger$.
\end{proof}

We have approximated the self-intersecting domain $\Omega_\ddagger$ 
with a family of $H^s$-class domains $\Omega^\pr$ (where $s\ge \frac92$,  and we assume 
$0<\pr\le \delta(r_0)$ hereafter) such that $\parOmega^\pr$ is free 
of self-intersections. Consequently, the local-in-time well-posedness 
theory applies to each of these approximating domains $\Omega^\pr$. Moreover, the $H^s$-norm of $\Omega^\pr$ is bounded 
independently of $\pr$. 
\begin{remark}
We remark that $\Omega^\pr$ coincides with $\Omega_\ddagger$ except on the 
two patches $\loc^\pr_-(B^-_{2\mathfrak{r}})$ and $\loc_+^\pr(B^+_{2\mathfrak{r}})$, 
where $B^\pm_{2\mathfrak{r}}\subset B^\pm_{1/2}$. In particular, since 
$\loc_\pm^\pr$ differ from $\loc_\pm$ only on a set properly contained 
in $\omega \subset \cov_0$, we can employ the same open cover 
$\{\cov_l\}_{l=0}^L$ for $\Omega^\pr$ as for $\Omega_\ddagger$. 
\end{remark}

\subsection{Uniform estimates on approximating domains}\label{sec_uniform}
A direct calculation shows that
\begin{equation}\label{eq_L2estimate}
\norm{u}_{L^{2}\Paren{\loc_\pm^\pr(B^{\sharp})}}^2=\int_{B^{\sharp}}\module{u\circ\loc_\pm^\pr}^2\module{J_\partial\loc_\pm^\pr}dx \le C\norm{u\circ\loc_\pm^\pr}_{L^{2}(B^{\sharp})}^2,
\end{equation}
where the constant $C$ depends only on $\loc_\pm$ and can be chosen independently of $\pr$ in view of the convergence $\loc^\pr_\pm\to\loc_\pm$.
Utilizing the upper and lower bounds for the Jacobian in \eqref{eq_jacobian}, 
we can apply the composition estimate (e.g., \cite[Proposition 1.6]{Disconzi2014}) 
to obtain the higher-order estimate
\begin{align}
\norm{u}_{H^{k}\Paren{\loc_\pm^\pr(B^{\sharp})}}&=\norm{u\circ\loc_\pm^\pr\circ(\loc_\pm^\pr)^{-1}}_{H^{k}\Paren{\loc_\pm^\pr(B^{\sharp})}}\nonumber\\
&\le C\norm{u\circ\loc_\pm^\pr}_{H^{k}(B^{\sharp})}\Paren{1+\norm{(\loc_\pm^\pr)^{-1}}_{H^{k}\Paren{\loc_\pm^\pr(B^{\sharp})}}^{k}},\label{eq_compo}
\end{align}
where $1 \le k \le k_0$ and the constant $C$ depends only on $k$ and 
$\operatorname{det}\Paren{\nabla\loc_\pm^\pr}$. By \eqref{eq_jacobian}, 
it depends on $k$ and $C_0$. To bound the term 
$\norm{(\loc_\pm^\pr)^{-1}}_{H^{k}\Paren{\loc_\pm^\pr(B^{\sharp})}}^{k}$, 
we recall that, in general, it holds
\begin{equation*}
\Paren{\nabla_y F^{-1}}(F(y)) =\Paren{\nabla F(y)}^{-1}= \frac{\operatorname{Cof}(\nabla F(y))}{\operatorname{det}(\nabla F(y))},
\end{equation*}
for a generic diffeomorphism $F$. Therefore,
\begin{equation*}
\Paren{\nabla\Paren{\loc_\pm^\pr}^{-1}}(\loc_\pm^\pr(y))=\frac{\operatorname{Cof}\Bracket{\nabla\loc_\pm^\pr(y)}}{\operatorname{det}\Bracket{\nabla\loc_\pm^\pr(y)}}=C_0^{-1}\operatorname{Cof}\Bracket{\nabla\loc_\pm^\pr(y)},
\end{equation*} 
and each $(j,i)$-entry takes the form
\begin{equation*}
\partial_i\Paren{\loc_\pm^\pr}_j^{-1}(x)=C_0^{-1}\Paren{\nabla\loc_\pm^\pr}\Bracket{\Paren{\loc_\pm^\pr}^{-1}(x)}\star\Paren{\nabla\loc_\pm^\pr}\Bracket{\Paren{\loc_\pm^\pr}^{-1}(x)}.
\end{equation*}
Then, we have
\begin{align*}
\partial_{il}\Paren{\loc_\pm^\pr}_j^{-1}(x)={}&C_0^{-1}\Paren{\nabla^2\loc_\pm^\pr}\Bracket{\Paren{\loc_\pm^\pr}^{-1}(x)}\star\nabla\Paren{\loc_\pm^\pr}^{-1}(x)\star\Paren{\nabla\loc_\pm^\pr}\Bracket{\Paren{\loc_\pm^\pr}^{-1}(x)}\\
={}&C_0^{-2}\Paren{\nabla^2\loc_\pm^\pr}\Bracket{\Paren{\loc_\pm^\pr}^{-1}(x)}\star\Brace{\Paren{\nabla\loc_\pm^\pr}\Bracket{\Paren{\loc_\pm^\pr}^{-1}(x)}^{\star,3}},
\end{align*}
where $T^{\star,m}$ denotes the $m$-fold $\star$-product of the tensor $T$.

By induction, for higher-order derivatives, it follows that
\begin{align*}
\nabla^k\Paren{\loc_\pm^\pr}^{-1}(x)={}&C_0^{-(k-1)}\Paren{\nabla^k\loc_\pm^\pr}\Bracket{\Paren{\loc_\pm^\pr}^{-1}(x)}\star\Brace{\Paren{\nabla\loc_\pm^\pr}\Bracket{\Paren{\loc_\pm^\pr}^{-1}(x)}^{\star,2k-1}}\\
&+C_0^{-(k-1)}\sum_{\substack{|i_1+\cdots+i_m|\le 2k-1,\\i_m\le\cdots\le i_1\le k-1}}\Paren{\nabla^{i_1}\loc_\pm^\pr}\Bracket{\Paren{\loc_\pm^\pr}^{-1}(x)}\star\cdots\star\Paren{\nabla^{i_m}\loc_\pm^\pr}\Bracket{\Paren{\loc_\pm^\pr}^{-1}(x)}.
\end{align*} 
We conclude that
\begin{equation*}
\norm{\Paren{\loc_\pm^\pr}^{-1}}_{H^{k}\Paren{\loc_\pm^\pr(B^{\sharp})}}^{k}\le C_0^{-k(k-1)} P_k\Paren{\norm{\nabla\loc_\pm^\pr}_{H^{k}(B^{\sharp})}},
\end{equation*}
for a suitable polynomial $P_k$ satisfying $P_k(0)=0$. Applying 
\eqref{eq_nablathetavarepsilon}, we obtain
\begin{equation*}
\norm{\Paren{\loc_\pm^\pr}^{-1}}_{H^{k}\Paren{\loc_\pm^\pr(B^{\sharp})}}^{k}\le C_0^{-k(k-1)}C\Paren{\norm{\nabla\loc_\pm}_{H^{k}(B^{\sharp})}}.
\end{equation*} 
Thus, by \eqref{eq_compo}, it follows that
\begin{equation}\label{equcircloc}
\norm{u}_{H^{k}\Paren{\loc_\pm^\pr(B^{\sharp})}}\le \widetilde{C}_0 \norm{u\circ\loc_\pm^\pr}_{H^{k}(B^{\sharp})},
\end{equation}
where the constant $\widetilde{C}_0$ depends only on $\Omega_\ddagger$ and is 
independent of $\pr$. Combining this with the $L^2$ estimate \eqref{eq_L2estimate} 
via interpolation, we conclude that the above inequality holds for any 
real index $\sigma\in[0,k_0]$.

Similarly, we prove that
\begin{equation}\label{eqcirclocl}
\norm{u}_{H^{\sigma}\Paren{\loc_l^\pr(B^{\sharp})}}\le \widetilde{C}_l \norm{u\circ\loc_l^\pr}_{H^{\sigma}(B^{\sharp})},\quad 1\le l\le K,
\end{equation}
where $\sigma\in[0,k_0]$ and the constant $\widetilde{C}_l$ depends only on 
$C_l$ and $\norm{\nabla\loc_l}_{H^{k}(B^{\sharp})}$. It depends solely 
on $\Omega_\ddagger$ and is therefore independent of $\pr$.

We have shown that $\norm{u}_{H^{\sigma}\Paren{\loc_\cdot^\pr(B^{\sharp})}}$ 
is bounded by $\norm{u\circ\loc_\cdot^\pr}_{H^{\sigma}(B^{\sharp})}$. Furthermore, 
the converse estimate holds by a similar argument, which conveniently 
bypasses the need to differentiate the inverse map $\Paren{\loc_\cdot^\pr}^{-1}$. 
More precisely, for $\sigma\in[0,k_0]$, we have
\begin{equation}\label{equcirclocin}
\begin{aligned}
\norm{u\circ\loc_\pm^\pr}_{H^{\sigma}(B^{\sharp})}&\le \widehat{C}_0 \norm{u}_{H^{\sigma}\Paren{\loc_\pm^\pr(B^{\sharp})}},\\
\norm{u\circ\loc_l^\pr}_{H^{\sigma}(B^{\sharp})}&\le \widehat{C}_l \norm{u}_{H^{\sigma}\Paren{\loc_l^\pr(B^{\sharp})}},
\end{aligned}
\end{equation}
where the constants $\widehat{C}_0$ and $\widehat{C}_l$ for $1\le l\le K$ 
depend only on $\Omega_\ddagger$.

\begin{lemma}\label{lemtracetheorem}
There exists a constant $C>0$, independent of $\pr>0$, such that
\begin{equation*}
\norm{u}_{H^{\sigma-\frac12}(\parOmega^\pr)}\le C\norm{u}_{H^{\sigma}(\Omega^\pr)},\quad \frac 12<\sigma\le k_0.
\end{equation*}
\end{lemma}
\begin{proof} 
By the standard trace theorem on $B^\pm$, there exists a constant $C>0$, 
independent of $\pr>0$, such that for any $u\in H^\sigma(\Omega^\pr)$, 
\begin{equation*}
\norm{u\circ\loc_\pm^\pr}_{H^{\sigma-\frac12}(B^{\sharp})} \le C \norm{u\circ\loc_\pm^\pr}_{H^{\sigma}(B^\pm)},\ \norm{u\circ\loc_l^\pr}_{H^{\sigma-\frac12}(B^{\sharp})} \le C \norm{u\circ\loc_l^\pr}_{H^{\sigma}(B^+)},\quad 1\le l\le K.
\end{equation*} 
By \eqref{eq_nablathetavarepsilon} and the construction $\loc_l^\pr\equiv\loc_l$ 
for $1\le l\le K$, we see that for $\pr>0$ sufficiently small,
\begin{equation*}
\frac{1}{2}\norm{\nabla\loc_\pm}_{H^{\sigma-1}(B^\pm)}\le \norm{\nabla\loc_\pm^\pr}_{H^{\sigma-1}(B^\pm)}\le\frac{3}{2}\norm{\nabla\loc_\pm}_{H^{\sigma-1}(B^\pm)},
\end{equation*}
and
\begin{equation*}
\norm{\nabla\loc_l^\pr}_{H^{\sigma-1}(B^+)}= \norm{\nabla\loc_l}_{H^{\sigma-1}(B^+)},\quad 1\le l\le K.
\end{equation*}
Invoking \eqref{equcircloc}, \eqref{eqcirclocl}, and \eqref{equcirclocin}, 
we deduce that
\begin{align*}
\norm{u}_{H^{\sigma-\frac12}\Paren{\loc_\pm^\pr(B^{\sharp})}}&\le \widetilde{C}_0 \norm{u\circ\loc_\pm^\pr}_{H^{\sigma-\frac12}(B^{\sharp})}\le C\widetilde{C}_0 \norm{u\circ\loc_\pm^\pr}_{H^{\sigma}(B^\pm)}\le C\widetilde{C}_0\widehat{C}_0 \norm{u}_{H^{\sigma}\Paren{\loc_\pm^\pr(B^\pm)}},
\end{align*}
where the positive constant $C\widetilde{C}_0\widehat{C}_0$ depends only 
on $\Omega_\ddagger$ and is independent of $\pr>0$.

Similarly, we also obtain
\begin{equation*}
\norm{u}_{H^{\sigma-\frac12}\Paren{\loc_l^\pr(B^{\sharp})}}\le C\widetilde{C}_l\widehat{C}_l \norm{u}_{H^{\sigma}\Paren{\loc_l^\pr (B^+)}},\ 1\le l\le K.
\end{equation*}
Since $\parOmega^\pr$ is the finite union of $\loc_+^\pr(B^{\sharp})$, 
$\loc_-^\pr(B^{\sharp})$, and $\loc_l^\pr(B^{\sharp})$ for $1\le l\le K$, 
we conclude that there exists a constant $C>0$, independent of $\pr$, 
such that
\begin{equation*}
\norm{u}_{H^{\sigma-\frac12}(\parOmega^\pr)}\le C\norm{u}_{H^\sigma(\Omega^\pr)}.
\end{equation*}
This completes the proof.  
\end{proof}
\begin{remark}
This approach also yields a series of uniform estimates on the approximating 
domains, including bilinear inequalities, Sobolev embeddings, the Kato-Ponce 
inequality and elliptic estimates. In subsequent applications 
of these inequalities, we will suppress explicit mention of the independence 
of these constants from $\pr$.
\end{remark} 
\subsection{Singular states configuration}\label{sec_singularstate}

For a sufficiently small constant $\alpha>0$, we have
$\overline{\loc_-(B^-_{1-\alpha})}\subset\cov_0$ and $\overline{\loc_+(B^+_{1-\alpha})}\subset\cov_0$. Furthermore, $\overline{\loc_l( B^+_{1-\alpha})}\subset\cov_l$ for each $l\le K$, 
and $\overline{\loc_l(B_{1-\alpha})}\subset\cov_l$ for each $K+1\le l\le L$. 
Consequently, the collection of open sets
\begin{equation*}
\loc_-(B^-_{1-\alpha}),\ \loc_+(B^+_{1-\alpha}),\ \loc_l(B^+_{1-\alpha}),\ 1\le l\le K,\ \loc_l(B_{1-\alpha}),\ K+1\le l\le L,
\end{equation*} 
forms an open cover of $\Omega_\ddagger$.  Since the diffeomorphisms $\loc_\pm^\pr$ 
are merely local modifications of $\loc_\pm$ within a very small 
neighborhood of the origin, it is clear that, independently of $\pr>0$, 
the sets
\begin{equation*}
\loc_-^\pr(B^-_{1-\alpha}),\ \loc_+^\pr(B^+_{1-\alpha}),\ \loc_l(B^+_{1-\alpha}),\ 1\le l\le K,\ \loc_l(B_{1-\alpha}),\ K+1\le l\le L,
\end{equation*} 
also constitute an open cover for each $\Omega^\pr$.

Corresponding to this $\alpha>0$, we introduce a smooth cutoff function 
$\zeta\in\mathcal{D}(B(0,1))$ satisfying $0\le\zeta\le 1$, with 
$\zeta(x)=1$ for $|x|<1-\alpha$ and $\zeta(x)=0$ for $|x|\ge1-\frac{\alpha}{2}$. 
Setting $\varsigma=1-\frac{\alpha}{2}$, we deduce that
\begin{equation*} 
\zeta\in\mathcal{D}(B(0,\varsigma))\ \text{and}\ 0\le\zeta\le 1.
\end{equation*}

The singular velocity and magnetic fields $\Paren{u_\ddagger,h_\ddagger}$ on an $H^s$-class domain $\Omega_\ddagger$ satisfy the following properties (see Fig.\,\ref{f10}):
\begin{enumerate}
\item Near the boundary, we require that 
\begin{equation*}
\zeta u_\ddagger\circ\loc_\pm,\zeta h_\ddagger\circ\loc_\pm\in H^{3}(B^\pm),\ \text{and}\ \zeta  u_\ddagger\circ\loc_l,\zeta  h_\ddagger\circ\loc_l\in H^{3}(B^+),\quad 1\le l\le K.
\end{equation*}  
\item For any open set $\cov$ such that $\overline{\cov}\subset\Omega_\ddagger$, we have
\begin{equation*}
u_\ddagger,h_\ddagger\in H^{3}(\cov).
\end{equation*} 
\item Under the motion of the fluid, the sets $\cov_0^+$ and $\cov_0^-$ move toward each other. We require that
\begin{equation}\label{touch}
\Paren{u_\ddagger\cdot\normal_-}\circ\loc_->C_\pm\ \text{in}\ B^-\ \text{and}\ \Paren{u_\ddagger\cdot\normal_+}\circ\loc_+>C_\pm\ \text{in}\ B^+, 
\end{equation} 
where the constant $C_\pm>0$, and we recall that the normal vector field is extended via harmonic extension.
\end{enumerate}

Note that we only define the velocity in the interior domain and do not prescribe its boundary value. At the self-intersection point $x_0$, the velocity field $u_\ddagger$, and hence the boundary velocity $\bdv$, should be understood in terms of the two one-sided traces. The same convention applies to the magnetic field. Since the one-sided normal vectors at $x_0$ are collinear, the tangential magnetic boundary condition remains compatible and it is therefore imposed as follows:
\begin{equation*} 
\begin{cases}
\Paren{h_\ddagger\circ\loc_\pm}(0)\cdot\normal=0,&\text{where}\ x_0=\loc_\pm(0),\\
\Paren{\zeta h_\ddagger\cdot\normal}\circ\loc_\pm=0,  &\text{on}\ B^{\sharp}\setminus\{0\},  \\
\Paren{\zeta h_\ddagger\cdot\normal}\circ\loc_l=0,\ l=1,\ldots,K,  &\text{on}\ B^{\sharp}.  
\end{cases}
\end{equation*}  
Formally, on $\Omega_\ddagger$, the pressure $p_\ddagger$ associated with the velocity $u_\ddagger$ and magnetic field $h_\ddagger$ solves the following problem:
\begin{equation*}
-\lpe p_\ddagger =\partial_j u_\ddagger^i\partial_iu_\ddagger^j-\partial_j h_\ddagger^i\partial_ih_\ddagger^j,\quad  \text{in} \ \Omega_\ddagger.
\end{equation*}
As for the boundary condition, it is natural to prescribe
\begin{equation}\label{splashpressureb1}
\begin{cases}
\zeta p_\ddagger\circ\loc_\pm=\mc,  &\text{on}\ B^{\sharp}\setminus\{0\},  \\
\zeta p_\ddagger\circ\loc_l=\mc,\ l=1,\ldots,K,  &\text{on}\ B^{\sharp}.  
\end{cases}
\end{equation} 
However, at the self-intersection point $x_0$, the condition $p_\ddagger=\mc$ cannot be imposed in the usual sense. Although $x_0=\loc_\pm(0)$ and, under the symmetric configuration, the one-sided mean curvatures satisfy $\mc^+=\mc^-$, the relation
\begin{equation}\label{splashpressureb2}
\mc^+=p_\ddagger\circ\loc_+(0)=p_\ddagger(x_0)=p_\ddagger\circ\loc_-(0)=\mc^-
\end{equation}
does not define a classical boundary condition at $x_0$. Indeed, the self-intersecting boundary is no longer an embedded hypersurface, and hence there is no single normal vector or mean curvature at the contact point. The obstruction is therefore not the equality of the one-sided scalar traces, but the absence of a single geometric boundary structure on which the boundary condition can be imposed. Thus, conditions \eqref{splashpressureb1} and \eqref{splashpressureb2} are not equivalent to the usual condition $p_\ddagger=\mc$ on an embedded free surface. This contrasts with the case without surface tension (e.g., \cite{Coutand2014}), where the boundary condition is simply the vanishing of the pressure and its local-coordinate formulation remains consistent with the entire boundary condition.

To overcome this difficulty, we separate the self-intersection point and impose the boundary velocity and pressure conditions on each approximating boundary. The pressure $p_\ddagger$ is then defined by passing to the limit as $\pr\to 0$. 

For $\pr>0$, 
we construct a sequence of approximations $u_\ddagger^\pr,h_\ddagger^\pr:\Omega^\pr\to\mathbb{R}^3$ by means of the Piola transform. Let $\Phi^\pr$ denote the $H^s$ map from $\Omega_\ddagger\setminus\{x_0\}$ to $\Omega^\pr$ determined by $\Phi^\pr\circ\loc_\pm=\loc_\pm^\pr$ on the distinguished charts and by $\Phi^\pr=\operatorname{Id}$ away from the modified neighborhood. Since the perturbation is supported inside the distinguished patches, these definitions agree on the chart overlaps. We set
\begin{equation}\label{usepsilon} 
u_\ddagger^\pr(\Phi^\pr(y))=J_{\Phi^\pr}(y)^{-1}\nabla\Phi^\pr(y)u_\ddagger(y),\quad
h_\ddagger^\pr(\Phi^\pr(y))=J_{\Phi^\pr}(y)^{-1}\nabla\Phi^\pr(y)h_\ddagger(y),
\end{equation} 
where $J_{\Phi^\pr}=\det\nabla\Phi^\pr$. This Piola transform gives $\divergence u_\ddagger^\pr=\divergence h_\ddagger^\pr=0$ in $\Omega^\pr$ and preserves the normal magnetic trace $h_\ddagger^\pr\cdot\normal^\pr=0$ on $\partial\Omega^\pr$.
 
We next define the approximate pressure $p^\pr_\ddagger$ as the $H^1(\Omega^\pr)$ 
weak solution of
\begin{equation*}
\begin{cases}
-\lpe p^\pr_\ddagger=\partial_j{u^\pr_\ddagger}^i\partial_i{u^\pr_\ddagger}^j-\partial_j{h^\pr_\ddagger}^i\partial_i{h^\pr_\ddagger}^j,\quad &\text{in}\ \Omega^\pr,\\
p^\pr_\ddagger =\mc^\pr,\quad &\text{on}\ \partial\Omega^\pr.
\end{cases}
\end{equation*}
Standard elliptic regularity theory then shows that $p^\pr_\ddagger\in H^{\frac52}(\Omega^\pr)$. Furthermore, since 
\begin{equation*}
\loc_\pm^\pr\rightarrow \loc_\pm,\ \text{and}\ \loc_l^\pr\rightarrow \loc_l\ \text{in}\ H^s \ \text{with}\ s\ge\frac92,\ \text{as}\ \pr\to 0,
\end{equation*} 
we infer from the definition of $u_\ddagger^\pr$ in (\ref{usepsilon}) that 
\begin{equation*}
\zeta \Paren{u_\ddagger^\pr,h_\ddagger^\pr}\circ\loc_\pm^\pr \rightarrow \zeta  \Paren{u_\ddagger,h_\ddagger}\circ\loc_\pm,\ \text{and}\ \zeta \Paren{u_\ddagger^\pr,h_\ddagger^\pr}\circ\loc_l^\pr \rightarrow \zeta  \Paren{u_\ddagger,h_\ddagger}\circ\loc_l\ \text{in}\ H^3\ \text{as}\ \pr\to 0.
\end{equation*} 
Moreover, the mean curvature $\mc^\pr$ of the approximate domains $\Omega^\pr$ converges in $H^2$ as $\pr\to 0$, because $\loc^\pr\to\loc$ in $H^{\frac92}$ and the curvature map from the boundary chart to $H^2$ is continuous at this regularity. By the uniform elliptic estimates on $\Omega^\pr$ and the stability of the corresponding Dirichlet problems under the convergence of the coordinate maps and data, we obtain 
\begin{equation*}
\zeta p^\pr_\ddagger\circ\loc_\pm^\pr \rightarrow \zeta p_\ddagger\circ\loc_\pm,\  \text{and}
\ \zeta p^\pr_\ddagger\circ\loc_l^\pr \rightarrow \zeta p_\ddagger\circ\loc_l\ \text{in}\ H^{\frac52}.
\end{equation*}

\section{Existence of the self-intersection singularity}\label{sec5} 
\subsection{Backward-in-time well-posedness}
Because the ideal MHD equations are time-reversible, on each approximate fluid domain $\Omega^\pr$, we can solve \eqref{eq_MHD} backward in time:
\begin{equation*}
\begin{cases}
\DT u^\pr-h^\pr\cdot\nabla h^\pr+\nabla p^\pr=0,&\text{in}\ \Omega_t^\pr, \\ 
\DT h^\pr=h^\pr\cdot\nabla u^\pr,&\text{in}\ \Omega_t^\pr,  \\ 
\divergence u^\pr=0,\quad \divergence h^\pr=0, &\text{in}\ \Omega_t^\pr,  \\ 
u^\pr_n=\bdv,\quad h^\pr\cdot\normal=0,\quad p^\pr=\mc,  &\text{on}\  \parOmega_t^\pr,  \\ 
u^\pr(\cdot,0)=u_\ddagger^\pr,\quad h^\pr(\cdot,0)=h_\ddagger^\pr, &\text{in}\ \Omega_0^\pr,
\end{cases}
\end{equation*} 
where $t<0$, and the domain $\Omega^\pr_0=\Omega^\pr$ denotes the final-time ($t=0$) domain.

Let the modified variables $(v^\pr,b^\pr,q^\pr)$ be defined as
\begin{equation}\label{eq_changeve}
\begin{cases}
v^\pr(x,t)=u^\pr(-x,-t),\\
b^\pr(x,t)=h^\pr(-x,-t),\\
q^\pr(x,t)=p^\pr(-x,-t),
\end{cases}
\end{equation} 
where $x\in \Xi_t^\pr$ and $t>0$. Here, the domain is
\begin{equation*}
\Xi_t^\pr=\Brace{(-x^1,-x^2,-x^3):(x^1,x^2,x^3)\in \Omega^\pr_{t} },
\end{equation*} 
and $\Xi_0^\pr=-\Omega^\pr=-\Omega^\pr_0$. Meanwhile, for the initial velocity and magnetic fields, we define 
\begin{equation*}
v_\ddagger^\pr(x)=u_\ddagger^\pr(-x),\quad b_\ddagger^\pr(x)=h_\ddagger^\pr(-x).
\end{equation*}
See Fig.\,\ref{f11}. 

Then the modified velocity field, magnetic field, and pressure solve 
\begin{subnumcases}
{\label{eqEulerbackmodi}}
\partial_t v^\pr+v^\pr\cdot\nabla v^\pr-b^\pr\cdot\nabla b^\pr+\nabla q^\pr=0,& \text{in}\ \Xi_t^\pr,\label{eqEulerbackmodi1}\\ 
\partial_t b^\pr+v^\pr\cdot\nabla b^\pr-b^\pr\cdot\nabla v^\pr=0,& \text{in}\ \Xi_t^\pr,\label{eqEulerbackmodi2}\\ 
\divergence v^\pr=0,\quad \divergence b^\pr=0, & \text{in}\ \Xi_t^\pr,\\ 
v^\pr_n=\bdv^\pr,\quad b^\pr\cdot\normal=0,\quad q^\pr=\mc^\pr,   & \text{on}\  \parXi_t^\pr,\label{eqEulerbackmodi4}\\ 
v^\pr(\cdot,0)=v_\ddagger^\pr,\quad b^\pr(\cdot,0)=b_\ddagger^\pr, & \text{in}\ \Xi_0^\pr,\label{eqEulerbackmodi5}
\end{subnumcases} 
since 
\begin{align*}
\operatorname{LHS}\ \text{of}\ \eqref{eqEulerbackmodi1} 
={}&-\Paren{\partial_t u^\pr}(-x,-t)-u^\pr(-x,-t)\cdot\Paren{\nabla u^\pr}(-x,-t)\\
&+h^\pr(-x,-t)\cdot\Paren{\nabla h^\pr}(-x,-t)-\Paren{\nabla p^\pr}(-x,-t)=0, \\
\operatorname{LHS}\ \text{of}\ \eqref{eqEulerbackmodi2} ={}&-\Paren{\partial_t h^\pr}(-x,-t)-u^\pr(-x,-t)\cdot\Paren{\nabla h^\pr}(-x,-t)\\
&+h^\pr(-x,-t)\cdot\Paren{\nabla u^\pr}(-x,-t)=0, \\
\divergence v^\pr(x,t)={}&-\Paren{\divergence u^\pr}(-x,-t)=0,\\
\divergence b^\pr(x,t)={}&-\Paren{\divergence h^\pr}(-x,-t)=0, 
\end{align*}  
in $\Xi_t^\pr$, 
\begin{align*} 
v^\pr_n(x,t)={}&-u^\pr(-x,-t)\cdot n(-x,-t)=-\bdv(-x,-t)=\bdv^\pr(x,t),\\
b^\pr(x,t)\cdot\normal(x,t)={}&-h^\pr(-x,-t)\cdot n(-x,-t)=0,\\
q^\pr(x,t)={}&p^\pr(-x,-t)=\mc(-x,-t)=\mc^\pr(x,t), 
\end{align*}
on $\parXi_t^\pr$, and
\begin{align*}
v^\pr(x,0)&=u^\pr(-x,0)=u_\ddagger^\pr(-x)=v_\ddagger^\pr(x),\\
b^\pr(x,0)&=h^\pr(-x,0)=h_\ddagger^\pr(-x)=b_\ddagger^\pr(x),
\end{align*}
in $\Xi_0^\pr$.
 
Since $\Omega^\pr$ is simply connected, applying the local existence theory \cite[Theorem 2.1]{Liu2023} with the vacuum magnetic field set to zero yields, for each fixed $\pr$, a time $T^\pr>0$ such that
\begin{equation*}
v^\pr,b^\pr\in C^0 (0,T^\pr;H^3(\Xi_t^\pr)),\quad \parXi_t^\pr\in C^0 (0,T^\pr;H^4)
\end{equation*}  
and $\Paren{v^\pr,b^\pr,\Xi_t^\pr}$ solves system \eqref{eqEulerbackmodi} with initial data $\Paren{v_\ddagger^\pr,b_\ddagger^\pr,\Xi_0^\pr}$. 

In the remainder of the section, we will prove that the existence time $T^\pr>0$ has a lower bound $T_\star>0$ independent of $\pr$ and that 
$\norm{v^\pr(\cdot,t)}_{H^3(\Xi^\pr_t)},\norm{b^\pr(\cdot,t)}_{H^3(\Xi^\pr_t)}$ and $\norm{\parXi_t^\pr}_{H^{4}}$ are bounded on $[0,T_\star]$ independently of $\pr$. This will then yield a solution
\begin{equation*}
\Paren{u^\pr(x,t)=v^\pr(-x,-t),h^\pr(x,t)=b^\pr(-x,-t),\Omega_t^\pr=-\Xi^\pr_t}
\end{equation*} on $[-T_\star,0]$ which culminates in the singularity $\Omega_\ddagger$ at $t=0$, starting from the initial data 
\begin{equation*} 
u_{-T_\star}=\lim_{\pr\rightarrow 0} u^\pr (\cdot,-T_\star),\quad h_{-T_\star}=\lim_{\pr\rightarrow 0} h^\pr (\cdot,-T_\star),\quad \Omega_{-T_\star} =\lim_{\pr\rightarrow 0} \Omega^\pr_{-T_\star},
\end{equation*}  
in a suitable limit sense.  

We recall the energy functionals defined in \eqref{eq_Energy} and \eqref{eq_energy}, and for $\Paren{v^\pr,b^\pr,\Xi_t^\pr}$, we denote
\begin{align*}
\energy^\pr(t)\coloneqq\frac12&\left(\norm{\DT^2v^\pr}_{L^2(\Xi_t^\pr)}^2+\norm{\DT^2b^\pr}_{L^2(\Xi_t^\pr)}^2+\norm{\Bdnabla\Paren{\DT v^\pr\cdot\normal}}_{L^2(\parXi_t^\pr)}^2\right.\nonumber\\
&\quad\left.+\norm{\nabla^2\Paren{\vorticity v^\pr}}_{L^2(\Xi_t^\pr)}^2+\norm{\nabla^2\Paren{\vorticity b^\pr}}_{L^2(\Xi_t^\pr)}^2\right).\\
\Energy^\pr(t)\coloneqq{}&\norm{\DT^2v^\pr}_{L^2(\Xi_t^\pr)}^2+\norm{\DT^2b^\pr}_{L^2(\Xi_t^\pr)}^2+\norm{\DT v^\pr}_{H^{\frac32}(\Xi_t^\pr)}^2+\norm{\DT b^\pr}_{H^{\frac32}(\Xi_t^\pr)}^2\nonumber\\
&+\norm{v^\pr}_{H^3(\Xi_t^\pr)}^2+\norm{b^\pr}_{H^3(\Xi_t^\pr)}^2+\norm{\Bdnabla\Paren{\DT v^\pr\cdot\normal}}_{L^2(\parXi_t^\pr)}^2+1.
\end{align*}
We also recall the a priori quantities
\begin{equation*}
\assone^\pr_{T^\pr}\coloneqq \radi^\pr-\sup_{t\in[0,T^\pr)}\norm{\eta^\pr(\cdot,t)}_{L^\infty(\Gamma^\pr)},
\end{equation*}
and 
\begin{equation*}
\asstwo^\pr_{T^\pr}\coloneqq\sup_{t\in[0,T^\pr)}\bigg(\norm{\nabla v^\pr}_{L^{\infty}(\Xi^\pr_t)}+\norm{\nabla b^\pr}_{L^{\infty}(\Xi^\pr_t)}+\norm{\nabla^2 b^\pr}_{L^{2}(\Xi^\pr_t)} +\norm{\eta^\pr(\cdot,t)}_{H^{\frac52}(\Gamma^\pr)}+\norm{v^\pr_n}_{H^{2}(\parXi^\pr_t)}\bigg).
\end{equation*}
Here, we choose $\Gamma^\pr=\parXi_0^\pr$ for simplicity. $\radi^\pr>0$ denotes the uniform exterior and interior ball radius of $\Xi_0^\pr$, and
$\eta^\pr$ denotes the height function defined on the reference surface $\Gamma^\pr$. With $\eta^\pr$ in hand, we can express the free boundary $\parXi_t^\pr$ as
\begin{equation*}
\parXi_t^\pr=\Brace{x+\eta^\pr(x,t)\normal(x,t):x\in \Gamma^\pr},
\end{equation*} 
and we also denote
\begin{equation*}
\assone^\pr_{0}\coloneqq \radi^\pr-\norm{\eta^\pr(\cdot,0)}_{L^\infty(\Gamma^\pr)}.
\end{equation*}
\begin{proposition}\label{pro_5.1}
There are constants $C_\star>0$ and $T_\star>0$ independent of $\pr$ such that, on the time interval $[0,T_\star]$, the free boundary $\parXi^\pr_t$ does not self-intersect, and
\begin{equation*} 
\sup_{t\in[0,T_\star]}\Paren{\Energy^\pr(t)+\norm{q^\pr}_{H^{\frac {5}{2}}(\Xi^\pr_t)}^2+\norm{\sff^\pr}_{H^{2}(\parXi^\pr_t)}^2}\le C_\star.
\end{equation*} 
\end{proposition}
\begin{proof} 
We define the following auxiliary quantity in order to apply the arguments in \textbf{Step 2} of the proof of Theorem \ref{thm_main1} (cf. Section \ref{sec_profthm12}):
\begin{equation*}
\Lambda^\pr(t) \coloneqq \norm{\sff^\pr}_{L^4(\parXi^\pr_t)}^4+\norm{\nabla q^\pr}_{L^2(\Xi^\pr_t)}^2+\norm{\nabla v^\pr}_{L^4(\Xi^\pr_t)}^4+\norm{\nabla b^\pr}_{L^4(\Xi^\pr_t)}^4+1,\quad t\ge 0.
\end{equation*} 
As in \eqref{eq_time_T0}, we define $T^\pr_0 \in (0, 1]$ to be the largest number such that
\begin{equation}\label{e:time_T0}
[0,T^\pr_0]\subset\left\lbrace t\in[0,1]: \Lambda^\pr(t)\le 2\Lambda^\pr(0),\assone^\pr_t \ge \assone^\pr_0/2,\ \text{and}\ \energy^\pr(t)\le 1+\energy^\pr(0)\right\rbrace.
\end{equation}
Then, by \eqref{eqEulerbackmodi4}, \eqref{eq_Hfrac121}, and the first condition in 
\eqref{e:time_T0}, we have
\begin{equation*}
\norm{\mc^\pr}_{H^{\frac 12}(\parXi_t^\pr)}^2\le C\Paren{\norm{\sff^\pr}_{L^2(\parXi_t^\pr)}^2+\norm{\nabla q^\pr}_{L^2(\Xi_t^\pr)}^2}\le C\Paren{\Lambda^\pr(0)}.
\end{equation*} 
We will apply the regularity result in Lemma \ref{lem_boudr} to recover the regularity of the height function $\eta^\pr$. The key point is that the Sobolev bound of $\norm{\parXi^\pr}_{H^s}$ is independent of the boundary distance of $\parXi^\pr$, which is directly related to the fact that curvature is a local property. In other words, the boundary can be arbitrarily close to self-intersection.
Thus, applying Lemma \ref{lem_boudr}, we have 
\begin{equation*} 
\norm{\eta^\pr(\cdot,t)}^2_{H^{\frac52}(\Gamma)}\le C\Paren{\Lambda^\pr(0)}.
\end{equation*}   
As in \eqref{eq_NT0le}, we can also obtain that
\begin{equation*}
\asstwo^\pr_{T^\pr_0}\le  C\Paren{\Lambda^\pr(0)}+\sup_{t\in[0,T^\pr_0)}\Energy^\pr(t)+\sup_{t\in[0,T^\pr_0)}\norm{v^\pr_n}_{H^{2}(\parXi^\pr_t)}.
\end{equation*}

To control the initial data $\Lambda^\pr(0)$, we revisit the proof of Proposition \ref{pro_E(0)} and apply the arguments in Section \ref{sec_uniform}. The constants appearing in the Kato-Ponce inequality, the elliptic estimate, Lemma \ref{lem_sffcontrolledbymc}, and the trace theorem can be shown, using local coordinates, to be independent of the boundary distance $\assone^\pr_{t}$ and of $\pr$. As a result, we can apply the modified Proposition \ref{pro_E(0)} to obtain
\begin{equation*}
\Lambda^\pr(0)\le C^\pr_{\operatorname{initial}} ,
\end{equation*}
where 
\begin{equation*}
C^\pr_{\operatorname{initial}}=f\Paren{\norm{v_\ddagger^\pr}_{H^3(\Xi_0)}, 
\norm{b_\ddagger^\pr}_{H^3(\Xi_0)},\norm{\mc^\pr}_{H^2(\parXi_0)}}
\end{equation*} 
and the function $f$ is independent of $\pr$. Recalling from \eqref{usepsilon} in Section \ref{sec_singularstate} that $C^\pr_{\operatorname{initial}}$ can be uniformly bounded by a constant $C_\star$ independent of $\pr$, we obtain
\begin{equation*}
\norm{\mc^\pr}_{H^{\frac 12}(\parXi_t^\pr)}^2+\norm{\eta^\pr(\cdot,t)}^2_{H^{\frac52}(\Gamma^\pr)}\le C(C_\star).
\end{equation*}
Similarly, we can modify the proof of
\eqref{eq_BHfrac32} to obtain 
\begin{equation*}
\norm{v^\pr_n}_{H^{2}(\parXi_t^\pr)}\le C\Paren{\Energy^\pr(t),C_\star}.
\end{equation*}
as in \eqref{eq_unh2}.
As a consequence, it follows that
\begin{equation*}
\asstwo^\pr_{T^\pr_0}\le C\Paren{\sup_{t\in[0,T^\pr_0)}\Energy^\pr(t),C_\star}.
\end{equation*}
By the conditions in \eqref{e:time_T0} and $\Lambda^\pr(0)\le C_\star$, 
combining the arguments in Section \ref{sec_uniform} with the proof of Proposition \ref{pro_Elee} allows us to obtain
\begin{equation}\label{eq_mainthmEleepr}
\Energy^\pr(t)\le C(C_\star)\Paren{1+\energy^\pr(t)}.
\end{equation}  
We note that the last condition in \eqref{e:time_T0}, \eqref{eq_mainthmEleepr}, together with the modified Proposition \ref{pro_E(0)} implies that 
\begin{equation}\label{eq_mainthmsupEpr}
\sup_{t\in[0,T^\pr_0)}\Energy^\pr(t)\le C\Paren{2+\energy^\pr(0)}\le C\Energy^\pr(0)\le C(C_\star).
\end{equation} 

Combining the above analysis, we conclude that
\begin{equation}\label{eq_NT0pr} 
\asstwo^\pr_{T^\pr_0} \le C(C_\star).
\end{equation} 
Since the a priori assumptions \eqref{eq_aprioriassumption} hold for time 
$T=T^\pr_0$, the claim follows once we show that $T^\pr_0$ specified in \eqref{e:time_T0} 
has a lower bound $c_0>0$. From the definition of $T^\pr_0$, at least one of 
the three conditions is satisfied with equality. 

Let us assume that $\Lambda^\pr(T^\pr_0) = 2\Lambda^\pr(0)$. Since the Reynolds transport formulas \eqref{eq_RT1} and \eqref{eq_RT2} and the commutator formula \eqref{eq_Dt,nabla} are independent of the boundary distance $\assone^\pr_{t}$, and since the constants in the normal trace theorem and the trace theorem have been proved to be independent of $\pr$ in Lemma \ref{lemtracetheorem}, we can repeat the arguments in the proof of \eqref{eq_LambdaleELambda} to obtain
\begin{equation}\label{eq_LambdaleELambdapr}
\frac{d}{dt}\Lambda^\pr(t)\le C\Energy^\pr(t) \Lambda^\pr(t), 
\end{equation}  
where the constant $C$ is independent of $\pr$. Integrating \eqref{eq_LambdaleELambdapr} over $(0,T^\pr_0)$ and using 
$\Lambda^\pr(T^\pr_0) = 2 \Lambda^\pr(0)$, we obtain
\begin{equation*}
\ln 2=\ln \Lambda^\pr(T^\pr_0)-\ln \Lambda^\pr(0)\le C(C_\star)T^\pr_0.
\end{equation*}
This yields
\begin{equation*}
T^\pr_0\ge c_0,
\end{equation*}
where the constant $c_0$ depends only on the initial data.

A similar argument applies if we have an equality in the third condition, i.e.,
\begin{equation*}
\energy^\pr(T^\pr_0)= 1+\energy^\pr(0).
\end{equation*}  
In this case, the a priori assumptions hold by \eqref{eq_NT0pr}, whereby 
we can apply the a priori estimates and \eqref{eq_mainthmsupEpr} to obtain
\begin{equation*}
\frac{d}{dt}\energy^\pr(t)\le C(C_\star).
\end{equation*}
Integrating over $(0,T^\pr_0)$ gives
\begin{equation*}
1=\energy^\pr(T^\pr_0)-\energy^\pr(0)\le C(C_\star)T^\pr_0,
\end{equation*}
which again results in
\begin{equation*}
T^\pr_0 \ge c_0>0.
\end{equation*}

Finally, assume that $\assone^\pr_{T^\pr_0} = \assone^\pr_0/2$ is the condition in \eqref{e:time_T0} that is satisfied with equality. Then, by \eqref{eq_mainthmsupEpr} and \eqref{eq_NT0pr}, we have
\begin{equation*}
\sup_{t\in[0,T^\pr_0)}\Paren{\Energy^\pr(t)+\norm{v^\pr_n}_{H^{2}(\parXi^\pr_t)}}\le C(C_\star).
\end{equation*} 
Note that
\begin{align*}
&v_n^\pr\Paren{x+\eta^\pr(x,t)\normal^\pr(x,0),t}-v_n^\pr(x,0)\\
={}&\Bracket{v^\pr\Paren{x+\eta^\pr(x,t)\normal^\pr(x,0),t}-v^\pr(x,0)}\cdot\normal^\pr\Paren{x+\eta^\pr(x,t)\normal^\pr(x,0),t}\\
&+v^\pr(x,0)\cdot\Bracket{\normal^\pr\Paren{x+\eta^\pr(x,t)\normal^\pr(x,0),t}-\normal^\pr(x,0)},
\end{align*}
and we deduce that
\begin{align*}
&\norm{v_n^\pr\Paren{x+\eta^\pr(x,t)\normal^\pr(x,0),t}-v_n^\pr(x,0)}_{L^\infty_x(\parXi^\pr_0)}\\
\le{}& \norm{v^\pr\Paren{x+\eta^\pr(x,t)\normal^\pr(x,0),t}-v^\pr(x,0)}_{L^\infty_x(\parXi^\pr_0)}\\
&+\norm{v^\pr(x,0)\cdot\Bracket{\normal^\pr\Paren{x+\eta^\pr(x,t)\normal^\pr(x,0),t}-\normal^\pr(x,0)}}_{L^\infty_x(\parXi^\pr_0)}.
\end{align*}
By applying the interpolation inequality, for $t\le T^\pr_0$,
\begin{align*}
&\norm{v^\pr\Paren{x+\eta^\pr(x,t)\normal^\pr(x,0),t}-v^\pr(x,0)}_{L^\infty_x(\parXi^\pr_0)}\\
\le{}& C\norm{v^\pr\Paren{x+\eta^\pr(x,t)\normal^\pr(x,0),t}-v^\pr(x,0)}_{H^1_x(\parXi^\pr_0)}^\alpha\norm{v^\pr\Paren{x+\eta^\pr(x,t)\normal^\pr(x,0),t}-v^\pr(x,0)}_{H^2_x(\parXi^\pr_0)}^{1-\alpha},
\end{align*}
where $\alpha\in (0,1)$. By applying Minkowski's inequality and the bilinear inequality,  
\begin{align*}
&\norm{v^\pr\Paren{x+\eta^\pr(x,t)\normal^\pr(x,0),t}-v^\pr(x,0)}_{H^1_x(\parXi^\pr_0)}\\
\le{}& \int_{0}^{t}\norm{\partial_tv^\pr\Paren{x+\eta^\pr(x,\tau)\normal^\pr(x,0),\tau}}_{H^1_x(\parXi^\pr_0)}d\tau\\
&+\int_{0}^{t}\norm{\partial_iv^\pr\Paren{x+\eta^\pr(x,\tau)\normal^\pr(x,0),\tau}\partial_t\eta^\pr(x,\tau)\normal^\pr_i(x,0)}_{H^1_x(\parXi^\pr_0)}d\tau\\
\le{}& \int_{0}^{t}\norm{\partial_tv^\pr\Paren{\cdot,\tau}}_{H^1(\parXi^\pr_{\tau})}d\tau+C\int_{0}^{t}\norm{\nabla v^\pr\Paren{\cdot,\tau}}_{H^{\frac32}(\parXi^\pr_{\tau})}\norm{v^\pr_n(\cdot,\tau)}_{H^1(\parXi^\pr_{\tau})}d\tau.
\end{align*} 
The second term can be controlled by $C(C_\star)t$. For the first term, we write the acceleration 
\begin{equation*}
\partial_t v^\pr=\DT v^\pr-v^\pr\cdot\nabla v^\pr, 
\end{equation*}
and it follows that
\begin{align*}
\sup_{\tau\in[0,t)}\norm{\partial_tv^\pr}_{H^1(\parXi^\pr_{\tau})}&\le 
\sup_{\tau\in[0,t)}\Paren{\norm{\DT v^\pr}_{H^1(\parXi^\pr_\tau)}+\norm{v^\pr\cdot\nabla v^\pr}_{H^1(\parXi^\pr_\tau)}}\\
&\le \sup_{\tau\in[0,t)}\Paren{\norm{\DT v^\pr}_{H^{\frac32}(\Xi^\pr_\tau)}+\norm{v^\pr}_{H^{\frac52}(\Xi^\pr_t)}^2}\\
&\le C\sup_{\tau\in[0,t)} \Paren{\sqrt{\Energy^\pr(\tau)}+\Energy^\pr(\tau)}\\
&\le C(C_\star).
\end{align*}
Thus, we obtain
\begin{equation*}
\norm{v^\pr\Paren{x+\eta^\pr(x,t)\normal^\pr(x,0),t}-v^\pr(x,0)}_{H^1_x(\parXi^\pr_0)}\le C(C_\star)t.
\end{equation*}
Since
\begin{align*}
\norm{v^\pr\Paren{x+\eta^\pr(x,t)\normal^\pr(x,0),t}-v^\pr(x,0)}_{H^2_x(\parXi^\pr_0)}&\le \norm{v^\pr\Paren{\cdot,t}}_{H^2(\parXi^\pr_{t})}+\norm{v^\pr(\cdot,0)}_{H^2(\parXi^\pr_0)}\\
&\le C\sup_{\tau\in[0,t)}  \sqrt{\Energy^\pr(\tau)}\\
&\le C(C_\star),
\end{align*}
we conclude that
\begin{equation*}
\norm{v^\pr\Paren{x+\eta^\pr(x,t)\normal^\pr(x,0),t}-v^\pr(x,0)}_{L^\infty_x(\parXi^\pr_0)}\le C(C_\star)t^\alpha.
\end{equation*}
In the same way, by applying  $\DT \normal^\pr=-\Paren{\Bdnabla v^\pr}^\top \normal^\pr$ in \eqref{eq_DTn}, we obtain 
\begin{equation*}
\norm{v^\pr(x,0)\cdot\Bracket{\normal^\pr\Paren{x+\eta^\pr(x,t)\normal^\pr(x,0),t}-\normal^\pr(x,0)}}_{L^\infty_x(\parXi^\pr_0)}\le C(C_\star)t^\alpha,
\end{equation*}
and therefore
\begin{equation*}
\norm{v_n^\pr\Paren{x+\eta^\pr(x,t)\normal^\pr(x,0),t}-v_n^\pr(x,0)}_{L^\infty_x(\parXi^\pr_0)}\le C(C_\star)t^\alpha.
\end{equation*}
Combining this with \eqref{touch}, the initial velocity in \eqref{eqEulerbackmodi5} satisfies
\begin{equation*}
v^\pr_n(\cdot,0)>C_\pm,\ \text{on}\ \cov_0^{+}\cap \parXi^\pr_0\ \text{and}\ \cov_0^{-}\cap \parXi^\pr_0.
\end{equation*}
As a consequence, even if the initial boundary velocity that drives the two boundaries apart starts to decrease from time zero, we can still obtain a lower bound for the time at which the velocity diminishes to zero, and the direction reverses. That is, the lower bound for the turning time is
\begin{equation*}
T^\pr_{\text{turn}}\ge \Paren{\frac{C_\pm}{C(C_\star)}}^{\frac{1}{\alpha}}>0.
\end{equation*}
Therefore, we have obtained the $\pr$-independent lower bound.  

Since the quantity that measures the degree to which the boundary approaches self-intersection, $\assone^\pr$, could decrease only if the two boundary segments of $\parXi^\pr_t$ continued to move toward each other after the separation of the self-intersection point $-x_0\in \parXi_\ddagger$, our analysis above shows that, once this point separates, the corresponding boundary portions remain apart at least up to the turning time $T^\pr_{\text{turn}}$. Consequently, the measure $\assone^\pr$ does not decrease. This, in turn, indicates that the entire boundary does not develop a stronger tendency toward self-intersection than that present in the initial configuration (see Fig.\,\ref{f15}). We conclude that
\begin{equation*}
T^\pr_0\ge T^\pr_{\text{turn}}\ge \Paren{\frac{C_\pm}{C(C_\star)}}^{\frac{1}{\alpha}},
\end{equation*}
because any time at which $\assone^\pr$ could decrease must necessarily lie beyond the lower bound of the turning time.

We conclude that there exists a lower bound $T_\star>0$ independent of $\pr$, such that
\begin{equation*}
T_0^\pr\ge T_\star\ \text{and}\
\sup_{t\in[0,T_\star]}\Energy^\pr(t)\le C(C_\star).
\end{equation*}
Moreover, we deduce from \eqref{eqEulerbackmodi1} that
\begin{equation*}
\norm{q^\pr}_{H^{\frac {5}{2}}(\Xi^\pr_t)}^2 +\norm{\sff^\pr}_{H^{2}(\parXi^\pr_t)}^2\le C(C_\star).
\end{equation*} 
Summarizing the above arguments, the proof is complete.
\end{proof}
\begin{remark}
In the proof of Proposition \ref{pro_5.1}, the key point is that, although $\assone^\pr_0$ depends on the parameter $\pr$, we can show that the time during which the velocity does not change direction has a lower bound.
\end{remark} 
\subsection{Proof of Theorem \ref{thm_main3}}
Finally, we prove the third main result in this paper.
\begin{proof}[Proof of Theorem \ref{thm_main3}]
By \eqref{eq_changeve}, we derive the following uniform estimates: on the time interval $[-T_\star,0]$, the free boundary $\parOmega^\pr_t$ does not self-intersect, and
\begin{equation*} 
\sup_{t\in[-T_\star,0]}\Paren{\norm{u^\pr}_{H^3(\Omega^\pr_t)}^2+\norm{h^\pr}_{H^3(\Omega^\pr_t)}^2+\norm{\parOmega^\pr_t}_{H^4}^2}\le C_\star.
\end{equation*}  
The estimates obtained above permit the use of the weak compactness argument in \cite[Section 8]{Coutand2014}. More precisely, after pulling the approximate domains and fields back to the fixed finite atlas, the uniform $H^4$ bounds for the boundaries and the uniform $H^3$ bounds for $u^\pr$ and $h^\pr$ yield a subsequence converging at $t=-T_\star$. As in \cite[Section 8]{Coutand2014}, the limiting coordinate charts define a regular non-self-intersecting fluid domain $\Omega_{-T_\star}$, locally lying on one side of its boundary, since the separation estimates for the approximate boundaries pass to the limit. The corresponding weak $H^3$ limits give the initial velocity and magnetic fields $u_{-T_\star}$ and $h_{-T_\star}$, and the divergence-free and tangential magnetic constraints are inherited from the approximating sequence. The verification of the chartwise compatibility, regularity, and non-self-intersection of the limiting initial domain follows the arguments in \cite[Section 8]{Coutand2014}, with the magnetic field carried along as an additional uniformly bounded $H^3$ field, and is therefore omitted. Evolving MHD equations \eqref{eq_MHD} forward from the initial data $\Paren{u_{-T_\star},h_{-T_\star},\Omega_{-T_\star}}$ produces the self-intersecting domain $\Omega_\ddagger$ with velocity and magnetic fields $u_\ddagger$ and $h_\ddagger$ at the final time $t=0$. This completes the proof. 
\end{proof}

\noindent\textbf{Acknowledgment}
S. Yang was supported by the Beijing Natural Science Foundation under Grant No. 1264051, the Beijing Postdoctoral Research Activity Funding Program, and the Funding Program for Newly Recruited Young Teachers of Beijing University of Technology. T. Luo was supported by a General Research Fund of the Research Grants Council of Hong Kong (Grant No. 11313025).
\medskip

\noindent\textbf{Data Availability} Data sharing is not applicable to this article as no datasets were generated or analyzed during the current study.
\medskip

\noindent\textbf{Declarations}
\medskip

\noindent\textbf{Conflict of interest} The authors declare that there is no conflict of interest.

\end{document}